\documentclass[12pt]{article}
\usepackage{cancel}
\usepackage{tabularx}
\usepackage{tikz}
\usepackage{relsize}
mm \textheight=8.9in

\usepackage{amssymb}
\usepackage{amsmath}
\usepackage{amsthm}
\usepackage{color}

\usepackage{tikz}
\usetikzlibrary{patterns,snakes}

\usepackage{tikz-cd}

\usepackage{array,multirow}

\newcommand{\br}[3]{{$#1$}$\lower4pt\hbox{$\tp\atop\raise4pt \hbox{$\scriptscriptstyle{#2}$}$} ${$#3$}}
\newcommand{\tw}[3]{{$#1$}${\,\scriptscriptstyle {#2}}\atop\raise9pt\hbox{$\scriptstyle\tp$} ${$#3$}}
\newcommand{\ttps}[2]{{#1}\raise5pt\hbox{$\lower12pt\hbox{$\scriptstyle\tp$}\atop \lower0pt\hbox{$\tilde\;$}$}\raise4.5pt\hbox{${\scriptstyle{#2}}$}}
\newcommand{\st}[1]{\mbox{${\,\scriptscriptstyle {#1}}\atop\raise5.5pt\hbox{$*$}$}}

\newcommand{\rd}[1]{\mbox{${\,\scriptscriptstyle {#1}}\atop\raise5.5pt\hbox{$\bullet$}$}}
\newcommand{\rt}[1]{\otimes_\chi}
\newcommand{\lt}[1]{\mbox{${\,\scriptscriptstyle {#1}}\atop\raise5.5pt\hbox{$\ltimes$}$}}
\newcommand{\btr}{\raise1.2pt\hbox{$\scriptstyle\blacktriangleright$}\hspace{2pt}}
\newcommand{\btl}{\raise1.2pt\hbox{$\scriptstyle\blacktriangleleft$}\hspace{2pt}}

\newcommand{\lcr}{\raise1.0pt \hbox{${\scriptstyle\rightharpoonup}$}}
\newcommand{\rcr}{\raise1.0pt \hbox{${\scriptstyle\leftharpoonup}$}}

\newcommand{\ttp}{{\lower12pt\hbox{$\tp$}\atop \hbox{$\tilde\;$}}}

\newcommand{\id}{\mathrm{id}}

\newcommand{\Ic}{\mathcal{I}}
\newcommand{\Xc}{\mathcal{X}}

\newcommand{\Ag}{\mathfrak{A}}
\newcommand{\Dg}{\mathfrak{D}}
\newcommand{\Bg}{\mathfrak{B}}

\newcommand{\Cg}{\mathfrak{C}}
\newcommand{\Sg}{\mathfrak{S}}

\newcommand{\Ng}{\mathfrak{N}}

\renewcommand{\S}{\mathcal{S}}

\newcommand{\Cc}{\mathcal{C}}

\newcommand{\Nc}{\mathcal{N}}

\newcommand{\C}{\mathbb{C}}
\newcommand{\Z}{\mathbb{Z}}

\newcommand{\N}{\mathbb{N}}

\newcommand{\tp}{\otimes}

\newcommand{\zt}{\zeta}

\newcommand{\V}{V}

\newcommand{\ve}{\varepsilon}
\newcommand{\gm}{\gamma}
\newcommand{\dt}{\delta}

\newcommand{\op}{\oplus}
\newcommand{\la}{\lambda}

\newcommand{\End}{\mathrm{End}}

\newcommand{\Span}{\mathrm{Span}}
\newcommand{\Aut}{\mathrm{Aut}}

\newcommand{\rk}{\mathrm{rk}}

\newcommand{\Rm}{\mathrm{R}}

\newcommand{\ad}{\mathrm{ad}}

\newcommand{\La}{\Lambda}

\newcommand{\mg}{\mathfrak{m}}
\newcommand{\g}{\mathfrak{g}}
\renewcommand{\d}{\mathfrak{d}}
\renewcommand{\b}{\mathfrak{b}}
\renewcommand{\k}{\mathfrak{k}}

\newcommand{\h}{\mathfrak{h}}

\newcommand{\mb}{\boldsymbol{m}}

\newcommand{\s}{\mathfrak{s}}

\renewcommand{\o}{\mathfrak{o}}

\newcommand{\m}{\mathbf{m}}
\newcommand{\eps}{\epsilon}

\newcommand{\nn}{\nonumber}
\newcommand{\p}{\mathfrak{p}}
\renewcommand{\l}{\mathfrak{l}}
\renewcommand{\c}{\mathfrak{c}}

\newcommand{\si}{\sigma}
\newcommand{\al}{\alpha}
\renewcommand{\t}{\mathfrak{t}}
\newcommand{\bt}{\beta}

\newcommand{\be}{\begin{eqnarray}}
\newcommand{\ee}{\end{eqnarray}}

\newtheorem{thm}{Theorem}[section]
\newtheorem{propn}[thm]{Proposition}
\newtheorem{lemma}[thm]{Lemma}

\newtheorem{definition}[thm]{Definition}
\newtheorem{example}[thm]{Example}
\newtheorem{remarks}[thm]{Remarks}

\newcount\prg

\newcommand{\parag}{\advance\prg by1 {\noindent\bf\thesection.\the\prg\hspace{6pt}}}

\begin{document}

\title{Affine supersymmetric pairs}
\author{ A. Mudrov, V. Stukopin
 \vspace{10pt}\\
\small
Moscow Institute of Physics and Technology,\\
\small
9 Institutskiy per., Dolgoprudny, Moscow Region,
141701, Russia,
\vspace{10pt}\\
\small
 e-mail:  mudrov.ai@mipt.ru, stukopin.va@mipt.ru
\vspace{10pt}
\\
}

\maketitle

\begin{abstract}
We classify Satake diagrams for general linear and orthosymplectic non-twisted affine  Lie superalgebras
and prove that each of them generates a family of proper spherical subalgebras. Furthermore, we demonstrate that
every  such family contains  subalgebras with  matrix invariants, which are viewed
as  classical analogs of K-matrices solving supersymmetric Reflection equation.
 \end{abstract}

\\


\newpage

\tableofcontents

\newpage
\section{Introduction}
This research is a sequel of our previous papers   \cite{AMS1, AMS2} devoted to spherical pairs of Lie superalgebras.
In this exposition we deal with a non-twisted affine   Lie superalgebra $\hat \g$ extending general linear or orthosymplectic $\g$.
By $\hat \g$ we understand the central extension of the loop algebra of $\g$.
Super-symmetric generalization of  symmetric pairs is a relatively new topic while their even predecessors became a well established
area both in  classical geometry and quantum algebra \cite{VK, Hel,Let,K,BK}. Given a significance of super-symmetry in quantum physics, such a generalization
should be of interest and importance. Already first results indicate that the super-symmetric extension of the conventional
theory of symmetric pairs promises to be very reach and displays non-trivial effects \cite{AMS1,AMS2,Shen,ShenWang}. One  reason for that is a diversity of non-isomorphic
Borel subalgebras in a given generic Lie superalgebra. Another one is a softer selection protocol  discarding trivial Satake diagrams with no proper subalgebras \cite{AMS1}.

As quantum groups grew up from the theory of Yang-Baxter equations \cite{FRT,D1},  quantum symmetric pairs originated  from the theory
of Reflection equation \cite{Cher,  KSkl,KSS} as they put its solutions into a conceptual algebraic context, \cite{NS,NDS}. This way the  classical symmetric pairs associated with
Riemann symmetric spaces and described by Satake diagrams were further generalized for the finite dimensional and many cases of infinite dimensional Lie algebras \cite{K,RV,AV}. This makes symmetric pairs a vibrant area of quantum algebra with non-commutative geometric flavour and applications to mathematical physics.

Super-symmetric extension of the even theory seems very natural because supersymmetry is a corner stone of quantum physics \cite{Dir,Ma}.
There is a growing interest to a supersymmetric version of spherical manifolds \cite{Sh,Sh1}.
Numerous results have been obtained on solutions to  $\Z_2$-graded Reflection equation including in its spectral parameter dependant form \cite{AACLFR, AACLFR1, L, DK, AlgMS1}.
However a $\Z_2$-graded extension of  symmetric pairs had not been attempted until recently. Without such a theory, the supersymmetric
Reflection equation lacks a systematic approach. The main obstruction had been the absence  of the longest Weyl group element with required properties.

Let us point out  a few studies undertaken in super-symmetric pairs. Shen and Wang addressed basic finite dimensional Lie superalgebras, described a class
of Satake diagrams and constructed their quantization as coideal subalgebras in \cite{Shen,ShenWang}. They restricted only to even Levi subalgebras, for which the longest Weyl group element is readily available.
Our first paper \cite{AMS1} introduced super-Satake diagrams for a fixed polarization of the basic matrix Lie superalgebras.
It is the simplest polarization that admits  a straightforward generalization of the longest Weyl group element,
which allows one  to extend the construction beyond even Levi subalgebras. A full description of  $\Z_2$-graded Satake diagrams
for an arbitrary polarization in general linear and orthosymplectic  $\g$ is obtained in \cite{AMS2}.

In this paper, we make a further step in that direction and classify all supersymmetric pairs in the three basic affine matrix Lie superalgebras: general linear and orthosymplectic of even and odd types. We are motivated by the importance of these algebras for the theory of integrable models.
Our study is restricted to the non-twisted case. We consider all possible polarizations of the underlying finite dimensional superalgebra, which
are determined by a choice of grading on their natural module.

Our approach is based on the theory of decorated Dynkin diagrams (DDD) first developed for finite dimensional Lie superalgebras with  a special
"minimal symmetric" polarization in \cite{AMS1} and extended to all polarizations in \cite{AMS2}. A key ingredient of this approach
is a generalization of the longest Weyl group element that is present in contragredient Lie algebras and is lacked in their $\Z_2$-graded counterparts.
We call such an element Weyl operator for brevity. It can be defined at least for basic matrix  Lie superalgebras. The affine case does not much differ
in this regard
from finite dimensional because such an element is needed only for finite dimensional Levi subalgebras $\l\subset \hat \g$.
The Weyl operator $w_\l$ is an automorphism of the root system of and weight lattice of $\l$. For "regular" Levi subalgebras, $-w_\l$  either preserves the root system 
of $\hat \g$  or takes it to an isomorphic. Still it is an involution on the dual to the  Cartan subalgebra and can be used for formalization of spherical conditions. 
In the case when  $-w_\l$  preserves the root basis of $\l$,  it preserves the total root  system of $\hat \g$. Then
it is an involutive isometry of the dual to the Cartan subalgebra of $\hat \g$.
It turns out that  "irregular" Levi subalgebras  do not give rise to non-trivial spherical pairs.

Along with the subalgebra $\l$, a spherical datum includes an involutive "automorphism" $\tau$ of the root system that extends the $-w_\l$ on the root basis of $\l$.
It is an honest automorphism if $-w_\l$ preserves the root basis of $\l$, and only then it is an isometry. Otherwise it takes  the total root system onto  isomorphic,
flipping  the tails of the Dynkin diagram in the orthosymplectic case or reflecting it around an axis for general linear $\hat \g$.
A spherical datum defines a family of subalgebras $\k\subset \hat \g$ generated by $\l$ and a family of elements comprising positive and negative root vectors
with a non-zero complex number (mixture weight), along with $\tau$-negative elements of the Cartan subalgebra. The mixed generators may involve also elements of the Cartan subalgebra centralizing $\l$, but we drop them for the sake of simplicity.
A key question that arises from this description is when the subalgebra $\k$ is proper.
A set of selection rules has been worked out in \cite{AMS1,AMS2} for special polarizations, which proves to be enough to filter out  DDD producing no proper  subalgebras, in all possible polarizations.  These selection rules also work for affine DDD, as demonstrated in this paper.

A direct application of selection rules is quite elaborate especially for orthosymplectic $\g$. It is possible to reformulate the description of finite dimensional admissible Satake diagrams obtained in  \cite{AMS2}
as  an inductive algorithm that allows one to get rid of selection rules completely making them implicit. This is also convenient for the study of affine Satake diagrams.

In order to find out if a given subalgebra is proper, one may search for its invariant in some representation that is not fixed by the total algebra.
This approach has another advantage because matrix invariants play a key role in the quantum version of the theory where
$U(\k)$ is  guantized to a coideal subalgebra in the total quantum (super)group. Quantum invariants satisfying the  Reflection equation
are called K-matrices. They are crucial  for integrability of quantum systems with non-periodical boundary conditions, such
as quantum spin chains \cite{Skl} and quantum field theories \cite{S,RSV}.

The simplest representation of $\hat \g$ where to search for a K-matrix  is the basic evaluation module $V_z$  and its tensor products.
A vast majority of Satake diagrams generate subalgebras $\k$ with invariants in $\End(V_z)$. These affine invariants restrict to those of their  finite dimensional subalgebras. We emphasize that
finding all their invariants is not our goal and confine ourselves with  model ones which we call standard. Some diagrams, especially of small rank, have more invariants which
nevertheless do not survive further extension to bigger diagrams; they are not interesting  in the current context.

 Subalgebras with $\l=\g$ have "essentially affine" tensor  invariants in $\End(V_z\tp V_{-z})$.
There is also a small (twelve) family of  "waif" diagrams in Sections \ref{Sec_Waif}
whose finite dimensional sub-invariants cannot be globalized in $\End(V_z)$, at all $z$.
Yet we prove that they generate proper  subalgebras $\k\subset \hat \g$.
Removing a set of  white roots, we obtain a finite dimensional sub-diagram with a subalgebra $\k'$ that has a nilpotent invariant  $\Nc\in \End(V_z)$, $\Nc^2=0$.
For a special value of the loop parameter $z$, the image $\k_z\in \End(V_z)$ of $\k$
is in $\ker (\ad \Nc)^2$ while $\hat \g$ is not.

In order to simplify our study, we restrict ourselves with  subalgebras whose mixed generators involve only simple root vectors.
As discussed above, not all such subalgebras admit invariant matrices in the evaluation representation even if their image is proper in $\g$.
In view of the  importance of  K-matrices in quantum algebra, we lift this restriction  in Section \ref{Sec_Waif} and
include an element of the Cartan subalgebra  in the affine mixed generator. Such an extension allows for an invariant for the twelve waif diagrams.
Thus we demonstrate that all Satake diagrams are not just non-trivial but allow for a matrix invariant for some of their spherical subalgebras.
This  facilitates their further studies  regard to the Reflection equation.

To compare the list of affine super-Satake diagrams with the conventional even theory, the reader is referred to \cite{RV,AV}.
However, like our previous paper \cite{AMS2}, this study deals with only classical pairs $(\hat \g,\k)$.
We consider quantization  a separate problem deserving a specially devoted study. That concerns all aspects of the theory including
construction of  quantum K-matrices.

The paper is organized as follows. We did not write out a separate list of diagrams because of its  tremendous spread  with a diversity of polarizations. We count it
more  comprehensible to address them case by case along with their "classical K-matrices" in specially designed sections.
In the next section we give a description of basic matrix affine superalgebras in an arbitrary polarization, which
is suitable for our purposes. We present simple root vectors of the underlying orthosymplectic Lie superalgebras preserving a given non-degenerate skew-diagonal bilinear form
of arbitrary signature
and list possible even shapes of their Dynkin diagrams. In  Section \ref{Sec_DDD} we modify the theory of finite dimensional DDD adopting it to the affine case
and remind selection rules. Therein we introduce a category of Satake sub-diagrams of a given diagram, which will help us to do our analysis.
 Section \ref{Sec_Fin-Dim} revises the finite dimensional Satake diagrams from \cite{AMS2} to make the presentation free of selection rules.
The last three sections are devoted to description of affine Satake diagrams and to a proof of their non-triviality.
First we do the general linear case in Section \ref{Sec_GL}. Orthosymplectic Satake diagrams fall into two classes: flipped left and right tails and stable tails.
They are processed in Sections \ref{Sec_OSP-Tails flipped} and \ref{Sec_OSP-id}, respectively.

\section{Affine matrix Lie superalgebras}
\label{Sec_Affine_basic}
For an exposition of the Lie superalgebra theory the reader is referred to  Musson's textbook  \cite{Mus}.
In this paper, $\g$ is  either general linear or ortho-symplectic matrix Lie superalgebra of rank $n$.
It is defined through its natural representation on a graded vector space $V$ of dimension $N$.
The algebra $\g$ admits a triangular decomposition  $\g=\g_-\oplus \h\oplus \g_+$ with respect to the Cartan subalgebra $\h$ represented by diagonal matrices.
The subalgebras $\g_-$ and $\g_+$ are represented by strictly lower and upper triangular matrices, respectively. Their class of isomorphism
is determined by the grading of $V$.

Let  $\{v_i\}_{i=1}^N\subset V$ be a basis of homogeneous vectors, where $N=n+1$ for general linear $\g$ and $N=2n$ or $N=2n+1$ for orthosymplectic $\g$.
The vectors  $v_i$ carry  weights $\zt_i\in \h^*$ subject to  $\zt_{i'}=-\zt_i$ for orthosymplectic $\g$, where $i'=N+1-i$.
For  ortho-symplectic $\g$ of odd $N$, the weight $\zt_{\frac{N+1}{2}}$ is zero.
The weights $\{\zt_i\}_{i=1}^N$ generate a  lattice $\La=\La_\g$, which contains the root system $\Rm$ of $\g$.

Consider  a   matrix $C\in \End(V)$ defined by
\be
\label{C}
C=\sum_{k=1}^n (e_{k,k'}+\eps_{k} e_{k',k})+e_{\frac{N+1}{2},\frac{N+1}{2}},
\ee
where the last term is present only if $N$ is odd. The quantities  $\eps_{k}$ take values in $\{\pm1\}$ and determine the type of $\g$.
We define a $\Z_2$-grading on $V$ subject to $\deg(v_{i'})=\deg(v_{i})$ and $\deg(v_{\frac{N+1}{2}})=0$ for odd $N$.
The relation with $C$ is given by the constraint $\eps_i\eps_j=(-1)^{\bar i +\bar j}$ for all $i,j=1, \ldots, n$,  where $\bar i=\deg(v_i)$.
Thus the matrix $C$  is even as an element of the graded algebra $\End(V)$

We denote by $t$ the super-transposition  acting on matrices by the assignment
$A_{ij}^t=(-1)^{\bar j(\bar i+\bar j) }A_{ji}$, $A\in \End(V)$. It is a  super-involutive anti-automorphism
of the graded associative algebra  $\End(V)$.
The orthosymplectic Lie algebra $\g$ is defined as a set of endomorphisms in  $\End(V)$ preserving the matrix $C$:
$$
A C+C A^t =0, \quad A \in \g.
$$
Thus $\g\simeq \o\s\p(2\mb|N-2\mb)$, where $\mb$ equals the number of minuses in  $\{\eps_i\}_{i=1}^n$.
Varying the signature of $C$ subject to this condition we obtain all possible polarizations of $\g$.

For general linear superalgebras, we consider all possible gradings on $V$. With the number $\mb$ of odd basis vectors in $V$ fixed,
we obtain all polarizations of $\g\simeq \g\l(\mb|N-\mb)$.

An inner product on $\h^*$ is introduced by setting it on the basic weights $\zt_i$:
$$
(\zt_i,\zt_j)=\dt_{ij}(-1)^{\bar i}, 
$$
where $i=1,\ldots, n$ for orthosymplectic $\g$ and $i=1,\ldots, N$ for general linear $\g$.
The basis of simple positive roots in $\g$ is expressed through the basic weights by
$$
     \Pi_{\g\l}=\{\zt_i-\zt_{i+1}\}_{i=1}^{N-1}, \quad
     \Pi_{\s\p\o}=\{\zt_i-\zt_{i+1}\}_{i=1}^{n-1}\cup \{2\zt_n\},
$$
$$
    \Pi_{\o\s\p}=\{\zt_i-\zt_{i+1}\}_{i=1}^{n-1}\cup\{\zt_n\}, \> \mbox{for odd}\>  N,  \quad
     \Pi_{\o\s\p}=\{\zt_i-\zt_{i+1}\}_{i=1}^{n-2}\cup\{\zt_{n-1}\pm \zt_n\},       \> \mbox{for even}\>N.
$$
The set $\Pi$ is a basis in $\h^*=\Span(\zt_i)_{i=1}^n$ if $\g$ is orthosymplectic. For general linear $\g$, the
vector space $\h^*$ is spanned by $\Pi$ and any weight $\zt_i$.
The set $\Pi$ generates the root system  $\Rm$ with the subset  $\Rm^+$ of positive roots.

The Lie superalgebra  $\g$ is  contragredient \cite{Mus} and  generated by $f_{\al}\in \g_-$, $e_\al\in \g_+$, $\al \in \Pi$, which
satisfy relations
$$
[e_\al,f_\al]=h_\al, \quad [h_\al,e_\bt]=(\al,\bt) e_\bt, \quad [h_\al,f_\bt]=-(\al,\bt) f_\bt.
$$
There are also additional relations,  depending on a polarization. Their
particular form is not important for this exposition.

We will be flexible with the choice of root vectors $e_\al,f_\al$ throughout the text and normalize them in a way
that is convenient for a particular task. That will be explicitly stated where needed.

The maximal positive  root $\xi$ of general linear $\g$ is $\zt_1-\zt_{n+1}$. For orthosymplectic it depends on the sign $\eps_1$:
$$
\xi=\zt_1+\zt_2,\quad \eps_1=1,\quad \xi=2\zt_1,\quad \eps_1=-1.
$$
The affine basis $\hat \Pi$ of simple roots is obtained by appending $\al_0$ to  $\Pi$. It carries the parity of $\xi$ and
its inner products with elements of $\Pi$ are
$$
(\al_0, \al_0)=(\xi, \xi), \quad (\al_0,\al)=-(\xi,\al), \quad \al \in \Pi.
$$
The affine algebra $\hat \g$ is a contragredient algebra with a symmetrized Cartan matrix. It is generated by $f_\al, e_\al$ with $\al\in \hat \Pi$ which satisfy
$$
[e_\al,f_\al]=h_\al, \quad [h_\al,e_\bt]=(\al,\bt) e_\bt, \quad [h_\al,f_\bt]=-(\al,\bt) f_\bt,
$$
with additional relations, which will not be used in this exposition. The Cartan subalgebra is spanned by $\h$ and $h_{\al_0}$, which is transversal to $\h$.
The inner product is degenerate on $\hat \h^*$ with  radical
 $\Ng\subset \h^*$ spanned by $\dt=\al_0+\xi$ and,  for $\g=\g\l(m|m)$, by $\sum_{i=1}^{2m} (-1)^{\bar i} \zt_i$.
We extend the assignment $\al\mapsto h_\al$ on roots to a linear map $\hat \h^*\to \h$ sending  $\la\mapsto h_\la\in \hat \h$.
It satisfies the identity $\mu(h_\la)=(\mu,\la)$ for all $\la, \mu\in \hat \h^*$.

It is known that $\hat \g$ is  isomorphic to a central extension of the loop algebra $L\g=\g\tp [t,t^{-1}]$.

\subsection{Basic evaluation representation}
\label{Sec_Basic Rep}
Recall that $\hat\g$ has a family of  representations on any finite dimensional $\g$-module $(V,\rho)$ parameterized by a non-zero complex number
$z$.
The one dimensional center of $\hat \g$ is in the kernel, thus it factors through a representation of the loop algebra $L\g$
which is identical on $\g$ and assigns $t\mapsto z\in \C^\times$.
The affine Chevalley generators
are represented by
$$
e_0\mapsto z f_\xi, \quad f_0\mapsto \frac{1}{z}e_\xi.
$$
 Now suppose that $V$ is the   natural representation of $\g$.
For a subalgebra  $\k\subset \hat \g$,  its image in $\End(V_z)$ is denoted by $\k_z$.
The pullback of this representation takes  $\h^*$ to a subspace in $\hat \h^*$  transversal to $\C\dt$. This pullback is independent of $z$
and makes $\g$-weights into $\hat \g$-weights as well.

For further convenience, we write down   root vectors of  orthosymplectic $\g\subset \End(V)$ preserving $C$ (\ref{C}), with respect to the natural
triangular decomposition:
\be
e_{\zt_k-\zt_m}=-(-1)^{\bar m(\bar k+\bar m)} e_{k,m}+ e_{m',k'}, \quad e_{\zt_k+\zt_m}=-\eps_k (-1)^{\bar k(\bar k+\bar m)} e_{k,m'}+e_{m,k'},
\label{roots-e}\\
f_{\zt_k-\zt_m}=-(-1)^{\bar k(\bar k+\bar m)} e_{m,k}+ e_{k',m'}, \quad f_{\zt_k+\zt_m}=-\eps_{m}(-1)^{\bar k(\bar k+\bar m)} e_{m',k}+e_{k',m},
\label{roots-f}
\ee
for  $1\leqslant k<m\leqslant n$,
and
$$
 e_{2\zt_k}=e_{k,k'},  \quad  f_{2\zt_k}=e_{k',k}, \quad \eps_k=-1, \quad 1\leqslant k\leqslant n,
$$
for even $N=2n$ and
$$
e_{\zt_k}=- e_{k,n+1}+e_{n+1,k'}, \quad f_{\zt_k}=-(-1)^{\bar k} e_{n+1,k}+e_{k',n+1}, \quad \eps_k=(-1)^{\bar k}, \quad 1\leqslant k\leqslant n,
$$
for odd $N=2n+1$.
\subsection{Dynkin diagrams}
A conventional way of graphical presentation for a root basis of a contragredient Lie superalgebra makes use
of colour differentiation between even and odd roots. Even roots are pained white while odd roots  grey if isotropic or
black. As the purpose of this exposition is a description of Satake diagrams which also uses a black-and-white designation,
we get rid of the conventional method and distinguish odd roots by squared nodes. This will not cause a confusion because the only non-isotropic
odd root may be the short one
in  orthosymplectic $\g\in \End(\C^{2n+1})$.

Discarding the parity of nodes gives an even diagram  (the double link connecting two isotropic nodes
in the even orthosymplectic case is removed). We call it shape of the graded Dynkin diagram.
Below we list the shapes of affine Dynkin diagrams up to an isomorphism. The affine root
of orthosymplectic $\g$ is placed on the left.

\noindent
$\g=\g\l(N)$
$$
     \hat \Pi_{ \Ag}=\{\zt_N-\zt_1\}\cup\{\zt_i-\zt_{i+1}\}_{i=1}^{N-1},
\quad
\mathrm{Aut}(\hat \Pi_{ \Ag})=\Z_{n+1}\rtimes \Z_2
$$

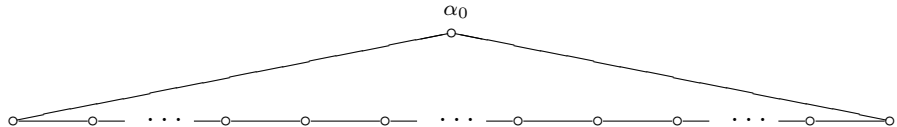
\begin{center}
\begin{picture}(330,50)
\put(1.5,10.5){\line(5,1){162}}\put(328.5,10.5){\line(-5,1){162}}
\put(165,43){\circle{3}}
\put(0,10){\circle{3}}
\put(30,10){\circle{3}}

\put(80,10){\circle{3}}
\put(110,10){\circle{3}}
\put(140,10){\circle{3}}

\put(82,10){\line(1,0){26}}
\put(112,10){\line(1,0){26}}

\put(142,10){\line(1,0){10}}
\put(178,10){\line(1,0){10}}
\put(160,10){$\ldots$}
\put(190,10){\circle{3}}
\put(220,10){\circle{3}}
\put(250,10){\circle{3}}

\put(192,10){\line(1,0){26}}
\put(222,10){\line(1,0){26}}

\put(302,10){\line(1,0){26}}

\put(300,10){\circle{3}}
\put(330,10){\circle{3}}

\put(252,10){\line(1,0){10}}
\put(288,10){\line(1,0){10}}
\put(270,10){$\ldots$}

\put(2,10){\line(1,0){26}}
\put(32,10){\line(1,0){10}}
\put(68,10){\line(1,0){10}}
\put(50,10){$\ldots$}

\put(162, 50){$\scriptstyle \al_0$}

 \end{picture}
\end{center}
$\g=\o\s\p(N)$, $N= 1\mod 2$
$$
   \hat  \Pi_{\Cg-\Bg}=\{-2\zt_1\}\cup\{\zt_i-\zt_{i+1}\}_{i=1}^{n-1}\cup\{\zt_n\},
\quad
\mathrm{Aut}(\hat  \Pi_{\Cg-\Bg})=1
$$
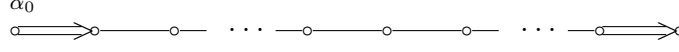
\begin{center}
\begin{picture}(200,30)
\put(-32, 18){$\scriptstyle \al_0$}
\put(0,10){\circle{3}}
\put(30,10){\circle{3}}

\put(80,10){\circle{3}}
\put(110,10){\circle{3}}
\put(140,10){\circle{3}}

\put(82,10){\line(1,0){26}}
\put(112,10){\line(1,0){26}}

\put(142,10){\line(1,0){10}}
\put(178,10){\line(1,0){10}}
\put(160,10){$\ldots$}
\put(190,10){\circle{3}}

\put(2,10){\line(1,0){26}}
\put(32,10){\line(1,0){10}}
\put(68,10){\line(1,0){10}}
\put(50,10){$\ldots$}

\put(220,10){\circle{3}}

\put(-30,10){\circle{3}}
\put(-29,8.5){\line(1,0){24}}
\put(-29,11.5){\line(1,0){24}}
\put(-9,7){$>$}

\put(191,8.5){\line(1,0){24}}
\put(191,11.5){\line(1,0){24}}
\put(211,7){$>$}
 \end{picture}
\end{center}
$$
    \hat \Pi_{\Dg-\Bg}=\{-\zt_1-\zt_2\}\cup\{\zt_i-\zt_{i+1}\}_{i=1}^{n-1}\cup\{\zt_n\},
\quad
\mathrm{Aut}(\hat  \Pi_{\Dg-\Bg})=\Z_2
$$
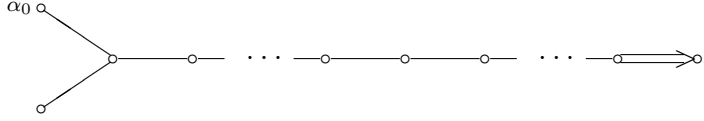
\begin{center}
\begin{picture}(200,30)
\put(-40, 28){$\scriptstyle \al_0$}
\put(0,10){\circle{3}}
\put(30,10){\circle{3}}

\put(80,10){\circle{3}}
\put(110,10){\circle{3}}
\put(140,10){\circle{3}}

\put(82,10){\line(1,0){26}}
\put(112,10){\line(1,0){26}}

\put(142,10){\line(1,0){10}}
\put(178,10){\line(1,0){10}}
\put(160,10){$\ldots$}
\put(190,10){\circle{3}}

\put(2,10){\line(1,0){26}}
\put(32,10){\line(1,0){10}}
\put(68,10){\line(1,0){10}}
\put(50,10){$\ldots$}

\put(191,8.5){\line(1,0){24}}
\put(191,11.5){\line(1,0){24}}
\put(211,7){$>$}
\put(220,10){\circle{3}}

\put(-1,11.5){\line(-3,2){25}}
\put(-1,8.5){\line(-3,-2){25}}
\put(-27,29){\circle{3}}
\put(-27,-9){\circle{3}}
\end{picture}
\end{center}

\noindent
$\g=\o\s\p(N)$, $N= 0\mod 2$
$$
\hat \Pi_{\Dg-\Dg}=\{-\zt_1-\zt_{2}\} \cup\{\zt_i-\zt_{i+1}\}_{i=1}^{n-2}\cup\{\zt_{n-1}\pm \zt_n\},
\quad
\mathrm{Aut}(\hat  \Pi_{\Dg-\Dg})=\Z_2\times \Z_2\times \Z_2
$$
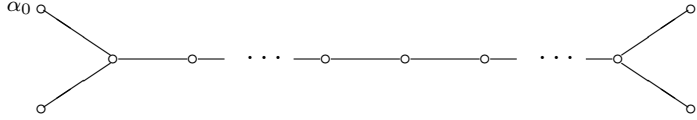
\begin{center}
\begin{picture}(200,30)
\put(-40, 28){$\scriptstyle \al_0$}
\put(0,10){\circle{3}}
\put(30,10){\circle{3}}

\put(80,10){\circle{3}}
\put(110,10){\circle{3}}
\put(140,10){\circle{3}}

\put(82,10){\line(1,0){26}}
\put(112,10){\line(1,0){26}}

\put(142,10){\line(1,0){10}}
\put(178,10){\line(1,0){10}}
\put(160,10){$\ldots$}
\put(190,10){\circle{3}}

\put(2,10){\line(1,0){26}}
\put(32,10){\line(1,0){10}}
\put(68,10){\line(1,0){10}}
\put(50,10){$\ldots$}

\put(191.5,11.5){\line(3,2){25}}
\put(191.5,8.5){\line(3,-2){25}}
\put(217.5,29){\circle{3}}
\put(217.5,-9){\circle{3}}
\put(-1,11.5){\line(-3,2){25}}
\put(-1,8.5){\line(-3,-2){25}}
\put(-27,29){\circle{3}}
\put(-27,-9){\circle{3}}

\end{picture}
\end{center}
$$
\hat \Pi_{\Cg-\Dg}=\{-2\zt_1\} \cup\{\zt_i-\zt_{i+1}\}_{i=1}^{n-2}\cup\{\zt_{n-1}\pm \zt_n\},
\quad
 \mathrm{Aut}(\hat  \Pi_{\Cg-\Dg})=\Z_2
$$
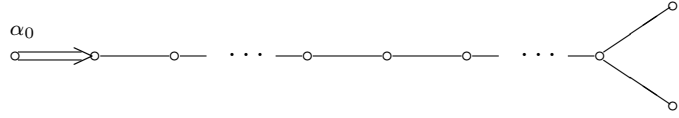
\begin{center}
\begin{picture}(200,30)
\put(-32, 18){$\scriptstyle \al_0$}
\put(0,10){\circle{3}}
\put(30,10){\circle{3}}

\put(80,10){\circle{3}}
\put(110,10){\circle{3}}
\put(140,10){\circle{3}}

\put(82,10){\line(1,0){26}}
\put(112,10){\line(1,0){26}}

\put(142,10){\line(1,0){10}}
\put(178,10){\line(1,0){10}}
\put(160,10){$\ldots$}
\put(190,10){\circle{3}}

\put(2,10){\line(1,0){26}}
\put(32,10){\line(1,0){10}}
\put(68,10){\line(1,0){10}}
\put(50,10){$\ldots$}

\put(191.5,11.5){\line(3,2){25}}
\put(191.5,8.5){\line(3,-2){25}}
\put(217.5,29){\circle{3}}
\put(217.5,-9){\circle{3}}

\put(-30,10){\circle{3}}
\put(-29,8.5){\line(1,0){24}}
\put(-29,11.5){\line(1,0){24}}
\put(-9,7){$>$}

\end{picture}
\end{center}

\begin{center}
$$
\hat \Pi_{\Cg-\Cg}=\{-2\zt_1\} \cup\{\zt_i-\zt_{i+1}\}_{i=1}^{n-2}\cup\{2 \zt_n\},
\quad
\mathrm{Aut}(\hat  \Pi_{\Cg-\Cg})=\Z_2
$$
\begin{picture}(200,30)
\put(-32, 18){$\scriptstyle \al_0$}
\put(0,10){\circle{3}}
\put(30,10){\circle{3}}

\put(80,10){\circle{3}}

\put(110,10){\circle{3}}
\put(140,10){\circle{3}}

\put(82,10){\line(1,0){26}}
\put(112,10){\line(1,0){26}}

\put(142,10){\line(1,0){10}}
\put(178,10){\line(1,0){10}}
\put(160,10){$\ldots$}
\put(190,10){\circle{3}}

\put(2,10){\line(1,0){26}}
\put(32,10){\line(1,0){10}}
\put(68,10){\line(1,0){10}}
\put(50,10){$\ldots$}

\put(220,10){\circle{3}}

\put(-30,10){\circle{3}}
\put(-29,8.5){\line(1,0){24}}
\put(-29,11.5){\line(1,0){24}}
\put(-9,7){$>$}

\put(194,8.5){\line(1,0){24.5}}
\put(194,11.5){\line(1,0){24.5}}
\put(190,7){$<$}
 \end{picture}
\end{center}
The roots are enumerated from left to right, with $\al_n$ put above $\al_{n-1}$ for the right shape $\Dg$.
Adding the parity information does not change the links between the nodes apart from the odd tail roots of shape $\Dg$ which
acquire a double link.

Note that  diagrams are characterizing triangular decompositions  of $\hat \g$. As different Borel subalgebras may be
non-isomorphic, different diagrams can correspond to the same  $\hat \g$.

Removing the node $\al_0$ leaves a Dynkin diagram of $\g$ of four types: $\Ag,\Bg,\Cg,\Dg$.
We distinguish two subalgebras $\s$ and $\t$ in $\g=\s+\t$ calling them shaft and tail respectively:
$\Pi_\t=\varnothing$ (zero $\t$) for $\Ag$, $\Pi_\t=\{\al_n\}$ for $\Bg,\Cg$ (short and long roots, respectively) and $\Pi_\t=\{\al_{n-1},\al_n\}$ for $\Dg$.
Clearly $\s$ is a general linear Lie superalgebra.
Thus the finite dimensional diagram $D_\g$ is the union $D_\g=D_\s\cup D_\t$.

Affine $\hat \g$ has right and left tails $\t^r$, $\t^l$ defined similarly. This breaks down the total diagram accordingly,
$D_{\hat \g}=D_{\t^l}\cup D_\s\cup D_{\t^r}$, where the shaft $D_\s$ is bridging the tails together.

For a diagram with right $\Dg$-tail, the group homomorphism $\Aut(D_{t^r})\to \End(\k_z)$ is generated by the flip $v_n\leftrightarrow v_{n'}$.
For a diagram with left  $\Dg$-tail, the group homomorphism $\Aut(D_{t^l})\to \End(\k_z)$ is generated by the flip
$$
v_1\mapsto -(-1)^{\bar 2(\bar 1+\bar 2)}z v_{1'}, \quad
v_{1'}\mapsto -(-1)^{\bar 2(\bar 1+\bar 2)}\frac{1}{z}v_1.
$$
The exact form of other automorphisms will not be  of use in this exposition.
\section{Spherical subalgebras and Satake diagrams}
\label{Sec_DDD}
 Let $\g$ be a contragredient Lie superalgebra that  with a Cartan subalgebra $\h$  and  $\b\subset \g$ a Borel subalgebra
 containing $\h$.
\begin{definition}
A Lie superalgebra $\k\subset \g$ is called spherical if $\g=\k+\b$.
Then the pair $(\g,\k)$ is called spherical.
\end{definition}
\noindent
Contrary to the non-graded case, this definition of sphericity substantially depends
upon a choice of $\b$.
In this section we present a  theory of spherical subalgebras that are quantizable to coideal subalgebras in quantum supergroups
along the lines of \cite{Let}. We adopt the results  of \cite{AMS1,AMS2} on finite dimensional Lie superalgebras to their affine extension.
\subsection{Weyl operator and spherical data }
\label{SecWOp}
In the classical theory of symmetric pairs of Lie algebras, the longest element of the Weyl group plays a crucial role.
The Weyl group of Lie superalgebras generated by even simple  reflections is too small to accommodate an element with required properties.
Still there is a substitute for it with a special automorphism of the root system that works at least for  basic non-exceptional Lie superalgebras.
Such an automorphism allows one to bypass the aforementioned obstruction and construct spherical subalgebras in general linear and orthosymplectic  $\g$.  That  can be done also for affine  $\hat \g$, because
the automorphism in question is needed only for its finite dimensional subalgebras.
 The only new thing to settle  is  its extension to the root system of  $\hat \g$.
We implement this programme and then proceed to construction of affine spherical pairs in this section.

Let $\g$ be a basic matrix Lie superalgebra and $V$ its natural representation of dimension $N$.
\begin{definition}
\label{Weyl-operator}
 A unique $\Z$-linear map $w_\g\colon \La\mapsto \La$ defined by the assignment
$\zt_i\mapsto \zt_{i'}$, $i=1,\ldots, N,$ is called
Weyl operator.
\end{definition}
\noindent
The endomorphism $w_\g$ is involutive and relates the highest and lowest weights of $V$ by construction.
For orthosymplectic $\g$, it preserves the parity of weights.
The situation is more complicated for general linear $\g$. The highest and lowest weights of $V$ have the same parity
if and only if the number of odd roots in $\Pi_\g$ is even. We call the corresponding  polarization of $\g$ even, otherwise we call it odd.
The grading of $V$ is called symmetric once $\bar i=\bar{i'}$ for all $i=1,\ldots, N$. It is so if and only if $w_\g$ even (preserves parity
of weights).
In that case we call polarization of $\g$ symmetric. A symmetric polarization of general linear $\g$ is always even.

 \begin{lemma}
\label{transposition}
The  operator  $w_\g$  preserves  the root system $\Rm$.
\end{lemma}
\begin{proof}
Obvious for orthosymplectic $\g$. For general linear $\g$ follows from a direct examination of $\Rm$.
\end{proof}
Let $\hat \Pi$ denote the basis of simple positive roots of the affine algebra $\hat \g$.
Pick a subset $\hat \Pi_\bullet \subset \hat \Pi$,  put  $\hat \Pi_\circ=\hat \Pi \backslash \hat \Pi_\bullet$, and generate a
 subalgebra $\l=\langle e_\al,f_\al\rangle_{\al\in \hat \Pi_\bullet}\subset \hat \g$ by the simple root vectors.
It is a direct sum of subalgebras, $\l=\sum_{i}\l_i$, corresponding to the connected components of $\hat \Pi_\bullet$ (of the Dynkin sub-diagram).

For $\al\in \hat \Pi_\circ$ let $V^\pm_\al\subset \hat\g_\pm$ denote the $\l$-submodule generated by $e_{\pm \al}\in \hat\g_{\pm \al}$.
It is a tensor product of modules over  the subalgebras in $\l$ corresponding to the connected components in $D_\l$. Such can be modules
of minimal dimension or its skew-(super)symmetrized tensor square over a component of general linear type \cite{Serg}.
In the special case of $\l\simeq\g$,  $V^\pm_{\al_0}$  is the adjoint $\g$-module.
Consider a finite dimensional $\l$-module  $V_\l=\sum_{\al\in \Pi_\circ}V^-_\al+\sum_{\al\in \Pi_\circ}\V^+_\al+\l$.
\begin{lemma}
  If $\al_0\in \hat \Pi_\bullet$, and $\l\not \simeq \g$, then the evaluation representation $\rho_z\colon V_\l\to \g_z\subset \End(V_z)$ is a
  monomorphism of $\l$-modules for each $z\in \C^\times$.
\end{lemma}
\begin{proof}
It is obviously an isomorphism on $\l$. It is also a monomorphism on irreducibles $V^\pm_\al$ because $\rho_z$ is identical
on their simple root generators
and $\rho_z(V^\pm_\al)\cap \l_z=\{0\}$.
\end{proof}
\noindent
The isomorphism $\rho_z\colon \l\to \l_z$, $\rho_z\colon V_\l\to V_{\l_z}$ relates  the weights $\al_0$ and  $-\xi$.

Suppose that $\l\subset \g$. For the $i$-th connected component of $\Pi_\bullet$ we  extend $w_{\l_i}$ from $\La_{\l_i}$ to   $\La_\g$ as  identical on basic weights
$\zt_i$ that annihilate the Cartan subalgebra $\h_i\subset \l_i$.
We  define the Weyl operator $w_\l\in \End(\La)$  of the subalgebra $ \l$
as the product $w_\l=\prod_i w_{\l_i}$.
Clearly $w_\l$ is involutive and preserves the weight lattice and root system of $\l$.

For general $\l\subset \hat \g$, we extend $w_\l$ to  $\Span(\hat \Pi)+\La$ as follows.
\begin{itemize}
\item
If $\al_0\in \hat \Pi_\circ$, then we set $w_\l(\al_0)=\al_0+\xi-w_\l(\xi)$.
\item
If $\al_0\in \hat \Pi_\bullet$ and $\l\not \simeq \g$, then we apply the module map  $\rho_z\colon V_\l\to \g$ and extend
$w_\l$ to $\hat \Pi_\circ \subset \Pi_\circ $ as for a subalgebra $\l_z$ in $\g$.
\item
If $\al_0\in \hat \Pi_\bullet$ and $\l\simeq \g$, then $\hat \Pi_\circ=\{\al_k\}$ for some $k>0$. We set $w_\l(\al_k)=2\dt-\al_k$.
\end{itemize}

It is easy to see that the operator  $w_\l$ is involutive.
It is an isometry on $\hat Q\tp_\Z\C$ if and only if it is an isometry on $\La_\l\tp_\Z \C$.

For instance, let $\g\subset\End(\C^N)$ be  general linear   and  $\hat \Pi_\bullet=\{\al_0\}$. Then $w_\l(\al_0)=-\al_0$.
The $\l$-modules $V^\pm_{\al_i}$ are  trivial for $i=2,\ldots, N-2$ and $V^\pm_{\al_i}\simeq \C^2$ for $i=1,N-1$.
The operator  $w_\l$ acts on  the simple roots of $\g$ by the assignment  $\al_i\mapsto \al_i+\al_0$ for $i=1,N-1$, and
identically on  $\al_j$ with  $2\leqslant j\leqslant N-2$.
Then $w_\l(\dt)=-\al_0+\al_1+\al_0+\al_{N-1}+\al_0+\sum_{i=2}^{N-2}\al_i=\dt$.

For orthosymplectic $\g$,  $w_\l$ is defined as a linear operator on $\hat \h^*=\Span(\hat\Pi)$. For general linear $\g$,
it is sufficient to define $w_\l$ on some basic weight $\zt_k$ in order to extend  it to $\hat \h^*$.
Suppose that $\al_k=\zt_k-\zt_{k+1}\in \hat \Pi_\circ$. Then we set $w_\l(\zt_k)=\zt_k+\sum_{i< k} \al_i$, where the summation is taken over the connected component in
$D_\l$ linked to $\al_k$ on the left (it is empty if $\al_{k-1}\not \in \hat \Pi_\bullet$). This makes  $w_\l$ consistently defined  as an involution on $\hat \h^*$  in all cases.

\begin{definition}
The subalgebra $\l\subset \hat \g$ is called regular if $w_\l$ maps the root system to an isomorphic  root system.
\end{definition}
\noindent
Otherwise $\l$  is called irregular. Such an anomaly may occur only in
a special situation explained further. It is clear that $\l$ is irregular only if one of $\l_i$
generated by a connected component of $D_\l$ is of  $\Ag$-shape with asymmetric polarization.
\begin{propn}
\label{Root-system-changed-regular}
Suppose that $w_\l$ does not preserve the root basis $\Pi$. Then $\l$ is regular  only  in the following two cases.
\begin{itemize}
  \item $\g$ is orthosymplectic,  $D_\l$ is connected  of $\Ag$-type, and there is a flip of $D_{\hat\g}$ exchanging $D_\t^l\leftrightarrow D_\t^r$ while preserving  $D_\l$
  \item $\g$ is general linear, $D_\l$ consists of one or two connected components, and there is a flip of $D_{\hat\g}$ preserving the connected components of $D_\l$.
\end{itemize}
\end{propn}
\begin{proof}
The presence of a non-empty connected $\Ag$-component $\l_0\subset \l$ induces a flip $w_{\l_0}$ of the sub-diagram $D_{\l_0}\subset D_{\hat \g}$. 
The new diagram should have the same shape as the original meaning that $w_{\l_0}$ should be extendable to a permutation of the
complementary to $\Pi_{\l_0}$. Since the transformed root system is isomorphic, the vertices equidistant from $D_{\l_0}$ have
to be permuted. Induction on the distance from $D_{\l_0}$ says that such a permutation is the left-to-right flip of $D_{\hat \g}$, call it $\varphi$,
extending $w_{\l_0}$. If $\g$ is general linear, then there is a possibility for another $\varphi$-invariant connected component $D_{\l_1}\subset D_\l$
thanks to cyclicity of $D_{\hat \g}$.
The permutation $\varphi$  reverse the order of vertices in $D_{\l_1}$ thus coinciding with $w_{\l_1}$ on it.  
If $ \g$ is orthosymplectic, then there is no room for $D_{\l_1}$ in $D_{\hat \g}$ where $\varphi=w_{\l_1}$, hence  so $D_\l$ is connected. 
The permutation $\varphi$ preserves the inner products amongst roots by construction is is an isomorphism of root systems represented by the 
diagrams 
$D_{\hat \g}$ and $\varphi(D_{\hat \g})$.
\end{proof}
\noindent

Remark that the above definition of regular subalgebras extends the definition given in  \cite{AMS2}, where $w_\l$ is required to preserve $\Rm$.
The rational behind this generalization is to accommodate new effects due to the presence of two tails or cyclicity of affine Dynkin diagrams.
\begin{lemma}
\label{high-low}
Suppose that $\l$ is regular. Then
\begin{itemize}
  \item[i)]For each $\al\in \hat  \Pi_\circ$, the $\l$-module $V^\pm_\al$  is determined by its lowest (highest) weight.
  \item[ii)] The operator $w_\l$ flips  the highest and lowest weights of $V^\pm_\al$ for each $\al \in \hat \Pi_\circ$.
\end{itemize}
\end{lemma}
\begin{proof}
The only $\l$ modules $V^\pm_\al$ that are not isomorphic to those occuring in finite dimensional $\g$ studied in \cite{AMS2}
correspond to either $\hat \Pi_\bullet =\{\al_0\}$ or
$\hat \Pi_\bullet =\{\al_n\}$.
In the first case, $V^+_{\al_0}$ are isomorphic to the adjoint module $\g$. Its lowest weight is $\al_0$ while the highest is $\al_0 +\xi-w_\g(\xi)$,
which is related with $w_\g(\al_0)$ by the evaluation $\l$-module isomorphism. The module $V^-_{\al_0}$ is isomorphic to  $V^+_{\al_0}$.

The second  case reduces to the first unless $\al_n$ is short (the right tail has shape $\Bg$). But then $V^\pm_{\al_n}$ is isomorphic
to the natural $\g$-module. The permutation $-w_\l$ is even and therefore preserves $\hat \Rm$, so the statement is obvious.
\end{proof}
\noindent
It follows from this lemma that the extension of $w_\l$ to an operator on $\hat \h^*$ is unique for regular $\l$ (depends only on
the $\l$-module structure of  $V_\l$).
We will demonstrate that the extension of $w_\l$ depends only on the isomorphism  class of the pair $(\hat \g,\l)$.
Let $D$, $D'$  be Dynkin diagrams of two isomorphic Lie superalgebras $\hat \g$ and $\hat \g'$, respectively,  and  $\varphi\colon D\to D'$ an isomorphism
that is induced by  the isomorphism $\Phi\colon \hat \g\to \hat \g'$.
They take a sub-diagram $D_\l\subset D$ to $D_\l'\subset D'$ and the subalgebra $\l\subset \hat \g$ to $\l'\subset \hat \g'$.
It is clear that $\l'$ is regular if and only if $\l$ is regular.
\begin{propn}
Suppose that $\l\subset \hat \g$ is regular. Then the operators $w_\l\in \End(\hat \h)$ and $w_{\l'}\in \End(\hat \h')$ are conjugated.
\end{propn}
\begin{proof}
Observe that the automorphism $\Phi$ not only relates $\l$ and $\l'$ but also intertwines their modules $V_\l$ and $V_{\l'}.$
Therefore  the extensions of $w_\l$ to $\hat \h$ and $w_{\l'}$ to $\hat \h'$ are $\varphi$-conjugated  once $\al_0\in \hat \Pi_\bullet \cap \hat \Pi_\bullet '\cup   \hat \Pi_\circ \cap \hat \Pi_\circ '$, because they are constructed in a similar way.
The same reasoning applies if $\al_0$ is in the symmetric difference of $\hat \Pi_\bullet$ and $\hat \Pi_\bullet '$,
provided $\l\not \simeq\g$.

Finally, consider the case when $\l\simeq \g$. We can assume that  $\hat \Pi_\circ=\{\al_0\}$ and $\hat \Pi_\circ'=\{\al_k\}$ for some $k$.
Then $w_\l(\al_0)=2\dt -\al_0$ and $w_{\l'}(\al_k)=2\dt' -\al_k$. The transformation $\varphi$ relates the root $\delta$ with $\delta'$, which completes  the proof.
\end{proof}

Now we apply the Weyl operator to construct spherical pairs of Lie superalgebras.
Suppose that  $\tau \in \Aut(\hat \Pi_\circ)$  is a permutation and let $\tilde \al$ denote the highest weight of $V_{\tau(\al)}^+$.
Suppose that $\tilde \al$ and $\al$ have the same parity.
\begin{definition}
\label{triple}
  The triple $(\hat \g,\l, \tau)$ is called super-symmetric if
\be
(\mu+\tilde\mu,\al)&=0,& \quad \forall \mu\in \hat \Pi_\circ, \quad \al\in \Pi_{\bullet},
\label{1nd-cond}\\
(\mu+\tilde\mu,\nu-\tilde \nu)&=0,& \quad \forall \mu,\nu\in \hat \Pi_\circ.
\label{2nd-cond}
\ee
\end{definition}
\noindent
Our goal is to find all $(\hat \Pi_\bullet,\tau)$ solving the system (\ref{1nd-cond}--\ref{2nd-cond}). The Weyl operator introduced above is
specially devised for this task. However, it features the required properties only for regular $\hat \Pi_\bullet$.
\begin{propn}
\label{regular-displaced}
Suppose that $w_\l$ does not preserve $\hat\Rm$. Then $\l$ is a member of a super-symmetric triple  only
if it is regular. Furthermore, $\tau$ then extends to an involution on $\hat \h^*$ commuting with $w_\l$.
\end{propn}
\begin{proof}

If $w_\l$ changes the root system, then  $D_\l$ has a connected component of type $\Ag$ with an asymmetric polarization.
Conditions (\ref{1nd-cond}) and  (\ref{2nd-cond}) imply that $D_\l$ is connected for orthosymplectic $\hat \g$ and may contain a second connected component
for general linear $\hat \g$. All white roots split to subsets of at most two elements that are equidistant from the connected components.
This implies a flip of the shape of $D_{\hat \g}$ taking it to its isomorphic that permutes the equidistant white roots and reverses the order of black roots.
  Condition
(\ref{1nd-cond}) permutes the roots connected to $D_\l$ while condition (\ref{2nd-cond})  is responsible for the flip of roots that are isolated from   $D_\l$.
Thus constructed permutation coincides with the flip $\varphi$ from the proof of Proposition \ref{Root-system-changed-regular}. Therefore  $\l$ is regular.

Finally, we extend $\tau$ as an operator on $\hat \h^*$ setting it equal to $-w_\l$ on $\La_\l\tp_\Z\C$. Thus defined, $\tau$ is involutive on complementary vector spaces $\La_\l\tp_\Z\C$ and $\C\hat \Pi_\circ$ and therefore  on the entire $\hat \h^*$. 
\end{proof}

From now on we assume  that $\hat \Pi_\bullet$ is regular. If $-w_\l$ preserves $\hat \Pi_\bullet$, then  we   extend $\tau$ to $\hat Q\tp_\Z\C$ setting it equal to $-w_\l$ on the $\hat \Pi_\bullet$. Otherwise we extend $\tau$ to $\hat \h^*$ by assuming it equal to $-w_\l$ on the $\La_\l$.
Let us demonstrate this for general linear $\hat \g$ assuming for simplicity connected $D_\l$.
Up to an automorphism, we reduce to the following two cases.
If $\La_\l=\Z\{\zt_m,\ldots, \zt_{m'}\}$, then $\tau(\zt_i)=\zt_{i'}$ for all $i=1, \ldots, N$ and $\tau(\al_0)=\al_0$.
If $\La_\l=\Z\{\zt_m,\ldots, \zt_{m'+1}\}$, then $\tau$ permutes $\al_1$ with $\al_0$ and flips $\zt_i$ with $\zt_{i'+1}$ for $i=2,\dots, N$.
If we arrange basic weights as vertices of the regular rectangle centered at the origin, then $\tau$ is the reflection around a line.
In the first case the line connects the origin with the mid point of the edge $[\zt_1,\zt_N]$ while in the second case it is passing through
the vertex $\zt_1$. This way $\tau$ extends to an involutive endomorphism of the lattice $\Z\al_0+\La$ and further to a linear operator on $\hat \h^*$.

By Lemma \ref{high-low}, we can set $\tilde \al =w_\l\circ \tau(\al)\in \hat \Rm^+$ for all $\al\in \hat  \Pi_\circ$.
\begin{lemma}
\label{tau-commutes w}
Suppose that $V^+_\al\simeq V^-_{\tau(\al)}$ for all $\al\in \hat \Pi_\circ$. Then  $\tau$ commutes with $w_\l$ as a linear endomorphism.
\end{lemma}
\begin{proof}
  Similar to \cite{AMS2}, Lemma 2.9.
\end{proof}
\noindent
Remark that the hypothesis of this lemma is equivalent to condition (\ref{1nd-cond}) by virtue of Lemma \ref{high-low} i).

\begin{example}
\em
Let us illustrate Lemma \ref{tau-commutes w} with general linear $\hat \g$, regular $\l$,  and $\tau\not =\id$.
We do not assume that $w_\l$ preserves the root basis. By this reason we need to extend to $\hat \h^*$.

Denote by $\hat \Pi_\circ'$ the subset of $\hat \Pi_\circ$ of roots isolated from $\hat \Pi_\bullet$.
It is $\tau$-invariant and its elements are fixed by $w_\l$. Thus we can restrict to the subset $\hat \Pi_\circ''$
of  roots that are connected to $\hat \Pi_\bullet$. They  are at most two, and they are permuted by $\tau$.
We can consider only the case $\hat \Pi_\circ''=\hat \Pi_\circ$ since otherwise we reduce the situation to  finite dimensional, that is known from \cite{AMS2}.

Up to an isomorphism, we can think that $\al_0\in  \hat \Pi_\circ$.  For the sake of simplicity assume that $D_\l$ is connected.
Then either  $\tau(\al_0)=\al_0$ or $\tau(\al_0)=\al_1$.
The permutation $\tau$ reflects the shape of  Dynkin diagram with respect to an axis passing through  the vertex $\al_0$ or through the mid-point
of the edge $[\al_0, \al_1]$.

If $\hat \Pi_\circ=\{\al_0\}$, then  $\La_\l$ equals $\La_\g$  while for $\hat \Pi_\circ=\{\al_0,\al_1\}$
it is generated by  $\{\zt_2,\ldots, \zt_{N}\}$.
Since $\tau=-w_\l$ on $\La_\l$, we are left to check that they  commute when applied to $\Phi$.
\noindent
First suppose that $\hat \Pi_\circ=\{\al_0\}$. Then $w_\l(\xi)=-\xi$, and
$$
w_\l(\al_0)=\al_0+\xi-w_\l(\xi)=\al_0+2\xi, \quad \tau(\al_0)=\al_0, \quad \tau(\xi)=\xi.
$$
We have
$$
w_\l \tau(\al_0)=w_\l(\al_0)=\al_0+2\xi = \tau(\al_0+2\xi)=\tau w_\l(\al_0).
$$
Now we look into the case of $\hat \Pi_\circ=\{\al_0, \al_1\}$.
The Weyl operator permutes $\zt_2$ with $\zt_N$ leaving $\zt_1$ fixed. Then
$$
w_\l(\al_0)=\al_0+\xi-w_\l(\xi)=\al_0-\xi_N+\xi_2, \quad \tau(\al_0)=\al_1,  \quad \tau(\al_1)=\al_0,
$$
and we have
$$
\tau w_\l(\al_0)=\al_1+\xi_2 -\xi_N=w_\l(\al_1)=w_\l\tau(\al_0),
$$
$$
\tau w_\l(\al_1)=\tau(\xi_1- \xi_N)=\tau(\xi_1-\xi_2+\xi_2- \xi_N)=\al_0 -\xi_N+\xi_2=w_\l(\al_0)=w_\l \tau(\al_1),
$$
Thus that $\tau$ and $w_\l$ commute on $\hat \h^*$ regardless
of a polarization of $\l$.
\end{example}

A linear endomorphism  $\theta=-w_\l\circ \tau$ of $\hat Q\tp_\Z\C$  preserves $\hat Q$ by construction. Furthermore,
$\theta(\al)=-\tilde \al$ for $\al \in \hat \Pi_\circ$
and $\theta(\al)=\al$ for $\al \in \hat \Pi_\bullet$.
 The system of equalities  (\ref{1nd-cond}) and (\ref{2nd-cond})
 translates to a single identity
\be
\label{gen-cond}
\bigl(\al+\theta(\al),\bt-\theta(\bt)\bigr)=0,\quad \forall \al,\bt\in \hat \Pi.
\ee

\begin{propn}
\label{tau-invol-orth}
Condition (\ref{gen-cond})
is  fulfilled if and only if $\tau$ is involutive, commutes with $w_\l$, and the composition  $\theta=-w_\l\circ \tau$ is
 an involutive isometry on $\hat Q\tp_\Z\C$.
\end{propn}
\begin{proof}
Notice that a linear operator being orthogonal and involutive with respect to a non-degenerate form is the same as being symmetric and involutive,
or orthogonal and symmetric simultaneously. I the form is degenerate, an invertible involutive isometry is symmetric. 
Thus the statement holds true in one direction. Let us prove the opposite.

The operator $\theta$ is identity on $\hat\Pi_ \bullet$ and takes a simple root  $\al\in \hat \Pi_\circ$ to the root $-\tilde \al$.
Therefore it preserves $\hat Q$.

Condition (\ref{gen-cond}) is bilinear and therefore holds true for any pair of vectors from $\hat Q\tp_\Z \C$.
Setting $\al=\bt$ for any vectors from  $\hat Q\tp_\Z \C$ proves that $\theta$ is an isometry. Then (\ref{gen-cond})
translates to
$$
\bigl(\theta(\al), \bt\bigr)=\bigl(\al,\theta(\bt)\bigr), \quad \forall \al, \bt\in \hat\Pi.
$$
It means that  $\theta$ is a symmetric operator relative to the form. Let us prove that it is  involutive. That is true if $-w_\l$ does not preserve $\hat \Pi_\bullet$, 
by Lemma \ref{regular-displaced}.
Thus we restrict to the case when $-w_\l$ preserves $\hat \Rm$ and $\tau$ is a permutation of $\hat \Pi$.

Expand the root $\dt\in \Ng$ over the basis  of simple roots: $\dt=\sum_{i=0}^n k_i \al_i\subset \hat \Rm^+$ with $k_i\in \N$.
Since $\theta $ preserves $\Ng$ and the root lattice, we arrive at the equality
$$
\theta(\dt) =\sum_{i=0}^n k_i' \al_i= k\sum_{i=0}^n k_i\al_i,
$$
with some integers $k_i'$ and some $k\not =0$.
Not that $\theta$ is identical on $\Z\hat \Pi_\bullet$. Modulo the black root lattice, $-\theta$ coincides with $\tau\in \Aut(\hat \Pi_\circ)$.
Thus 
$$
-\theta(\dt) =\sum_{i=0}^l k_{\tau^{-1}(i)} \al_i = -k\sum_{i=0}^l k_i\al_i \quad\mbox{modulo}\quad  \Z\hat \Pi_\bullet,
$$
where $i=1,\ldots,l$ enumerates  roots from $\hat \Pi_\circ\not =\varnothing$.
This gives  $- k_{\tau^{-1}(i)}=kk_i$ for all such $i$ forcing that  $k$ is rational. 
There is an index $i$ such that $0 \not = k_i=(-1)^m k^m k_i$. Since $k$ is rational, $k=\pm 1$.
Thus $\theta$ is an involution on $\C\dt$ and on the quotient by $\C\dt$, where the form is non-degenerate. 
Now notice that $\theta$ is semi-simple being a product of semi-simple operators $-w_\l$ and $\tau$  commuting 
by Lemma \ref{tau-commutes w}. Therefore it is involutive. 
\end{proof}
\noindent
We conclude that if $w_\l$ preserves $\hat \Pi_\bullet$, and consequently is an isometry on $\hat Q\tp_\Z\C$, so does $\tau$.
That is, $\tau$ is an involutive automorphism of the Dynkin diagram of $\hat \g$.
Otherwise $D_\l$ is connected and $\tau$ takes the diagram to isomorphic: it is a flip exchanging the tails for orthosymplectic $\hat \g$ and the reflection
around an axis for general linear $\hat \g$. In both cases, $\tau$ is an involutive  endomorphism of $\hat \h^*$.

Here are a few comments about $\theta$ satisfying (\ref{gen-cond}).
\begin{remarks}
\em
\begin{itemize}
  \item
The map $\theta$ preserves $\hat \Rm$. That is obvious if $-w_\l$ preserves $\hat \Pi_\bullet$. Otherwise the involution $\tau\not =\id$ extends
to $\Z\al_0+\La_\g$ as $-w_\l$ on $\La_\l$. Then $\theta$ permutes the basic weights orthogonal to $\La_\l$ leaving those from $\La_\l$ fixed.
Therefore it preserves $\Rm=\hat \Rm \mod \Z\dt$, so the claim follows.

  \item
For regular $\l$ our requirement for roots $\al\in \hat\Pi_\circ$ and $\tilde \al=-\theta(\al)$
to be of  the same parity is redundant. It is then automatically fulfilled  because  $\theta$ is an isometry.
 \item
In the case of general linear $\g$ and $\tau=\id$, $\tau$ cannot be extended to $\hat \h^*$ unless $\hat \Pi_\bullet$ is empty.
If $\hat \Pi_\bullet\not =\varnothing$, then $D_\l$ is  a union of isolated even roots, so $w_\l$ is well defined on $\hat \Pi_\bullet$
and preserves the root system $\hat \Rm$. It is natural to reduce  $\g\l$ to  $\s\l$ in this case.

\end{itemize}
\end{remarks}

Suppose that the pair $(\hat\Pi_\bullet,\tau)$  solves the system (\ref{1nd-cond}--\ref{2nd-cond}).
Denote by $\hat\c\subset \hat\h$ the centralizer of $\l$ in $\hat\h$ and by  $\Omega\subset \hat \Pi_\circ$ the subset of even roots fixed by $\tau$
and orthogonal to $\hat \Pi_\bullet$. Equivalently, $\Omega $ consists of even $\tau$-fixed white roots such that the $\l$-modules $V^\pm_\al$ are trivial.
Fix a map
$$
\hat \Pi_\circ \ni \al\mapsto u_\al\in \hat\c,
\quad
u_\al=0, \quad  \forall \al \not\in \Omega.
$$
For each $\al\in \hat \Pi_\circ $ pick  $c_\al\in \C^\times $, $\grave c_{\al}\in \C$ and put
\begin{equation}
\begin{array}{rcl}
y_\al &=& h_\al - h_{\tilde \al}, \\
x_\al &=& e_\al + c_\al f_{\tilde \al} + \grave c_{\al} u_\al,
\end{array}
\label{gen-spher-superpairs}
\end{equation}
for all $\al\in \hat\Pi_\circ $.
Define a Lie superalgebra $\k\subset \hat\g$ as the one generated by $\l$ and by $x_\al$, $y_\al$ with $\al\in  \hat \Pi_\circ$.
Remark that the elements $y_\al$ are presentable as  $y_{\al}=h_\al + h_{\theta( \al)}= h_\al - h_{\tau( \al)}\mod \l$, which
will be convenient for calculations. Since $\tau$ is involutive, $y_\al=-y_{\tau(\al)}\mod \l$.
\begin{definition}
  The pair of  Lie  superalgebras $\k\subset \hat \g$ determined by a super-symmetric triple $(\hat \g,\l,\tau)$ and by the set of
  $c_\al, \grave{c}_\al,u_\al$, $\al \in \hat\Pi_\circ$, is called super-symmetric.
\end{definition}
\noindent
The  complex numbers $c_\al,\grave c_{\al}$ in (\ref{gen-spher-superpairs}) are called mixture parameters.
Typically $u_\al$ is chosen equal to $h_\al$, for $\al\in \Omega$. With this convention, we may think that
$\k$ depends on $c_\al,\grave c_{\al}$. Setting all $\grave c_{\al}$ to zero, we obtain a subfamily of algebras parameterized
by a point $\vec c= (c_\al)_{\al\in \hat\Pi_\circ}$ of an algebraic  torus $\mathfrak{T}$. We call this subfamily toric.

 \begin{propn}
\label{ps-sym=sup-sph}
  A Lie superalgebra $\k$ determined by a super-symmetric triple and a vector of mixture parameters is spherical in $\hat \g$.
\end{propn}
\begin{proof}
See \cite{AMS1}, Proposition 4.16 for a special finite dimensional case. It is suitable for $\hat \g$ as well because
the nilpotent subalgebras $\hat \g_\pm$ are generated over $\l$ by  $V^\pm_\al$ with  $\al\in \hat \Pi_\circ$. Construction 
of $\l$-invariant filtrations of $\k$ and the parabolic subalgebra of $\l$ whose sum exhausts all of $\hat \g$ is literally 
the same with the only difference that the filtrations are finite  for finite dimensional $\g$.
\end{proof}
\noindent
The rest of the paper is an answer to the question when the subalgebra $\k\subset \hat\g$ is proper.
For the sake of simplicity, we will consider only  the toric parametrization of $\k$, with  few exceptions in
Section \ref{Sec_waif_inv}.
 \subsection{Decorated diagrams and selection rules}
\label{SecDDD-SR}
Like in the non-graded case the permutation $\tau$ entering a super-symmetric triple $(\hat \g,\l,\tau)$
 can be conveniently visualized via decorated Dynkin diagrams (DDD).
We use black colour for nodes in $\hat \Pi_\bullet$ and white for $\hat \Pi_\circ$ regardless of their parity.
The parity  will be  either described in words or via the following convention:
a circle $\>
\begin{picture}(5,10)
\put(0.5,3){\circle{3}}
\end{picture}
$ designates an even root while square
 $
\begin{picture}(5,10)
\put(0,1.5){\framebox(3,3)}
\end{picture}
$ stands for odd; a rhombus  $
\begin{picture}(5,10)
\put(0,.5){$\scriptstyle\lozenge$}
\end{picture}
$
 means  a root of  arbitrary parity.
The number $|\hat \Pi_\bullet|$ of black nodes in the diagram is called its black rank. The cardinality $|\hat \Pi_\circ|$ is called its white rank.

As the subalgebra $\k$ is defined by a set of generators, it is not {\em a priory} obvious when it is proper, i.e. strictly less than $\hat \g$.
Otherwise $\k$ is not interesting, and such pairs $(\hat \g,\k)$ are called trivial.
Below we formulate criteria that rule out DDD  giving rise to trivial pairs  for all values
of the mixture parameters. Such diagrams are considered as trivial and should be discarded.

Selection rules that filter out trivial DDD were formulated for finite dimensional $\g$ with the minimal symmetric polarization in \cite{AMS1} and generalized for any choice of Borel subalgebra in $\g$, \cite{AMS2}.
We recall  them without proof for reader's convenience. In a subsequent section, we work out an inductive description of admissible DDD without appealing to the selection rules.

It is convenient for the study of DDD to reduce them to smaller parts.
Let $C(\al)$ denote  the union of connected components of  $D_\l\subset D_\g$ that are  linked to  $\{\al,\tau(\al)\}\subset \hat \Pi_\circ$.
 \begin{definition}
A decorated Dynkin  sub-diagram is a subgraph $S'$ in the total decorated diagram $S$ such that $\tau(\al)\in S'$ and $C(\al)\subset S'$ for every white node $\al \in S'$.
\end{definition}
\noindent
In particular, one can consider sub-diagrams with zero number of white roots. They comprise connected components $D_\l$ and their unions.

For instance, the graph $\Sg(\al)$ with nodes $C(\al)\cup \{\al,\tau(\al)\}$ is the minimal decorated sub-diagram that includes $\al\in \hat \Pi_\circ$.
 More generally, $\Sg(\al_1,\ldots, \al_k)$ will designate the decorated sub-diagram generated by $\al_1,\ldots, \al_k\in \hat \Pi_\circ$.
It is the union of $\al_i,\tau(\al_i)$ and $C(\al_i)$, over $i=1,\ldots, k$.

Every decorated sub-diagram $S'\subset S$  defines subalgebras $\g'\subset \hat \g$ and  $\l'\subset \l$, whose simple root generators
are the nodes of $S'$ and $D_\l\cap S'$, respectively. Given
a spherical subalgebra $\k\subset \g$  determined by $S=(D_\g,D_\l,\tau)$ and a   vector $\vec c$ of  mixture parameter, we define
$\k'$ as generated by $\l'$ and by (\ref{gen-spher-superpairs}) with all white $\al \in S'$. Its vector $\vec c'$ of mixture parameter
is obtained from $\vec c$ by the natural projection.
Clearly $(\g',\k', \tau'=\tau|_{S'})$ is a spherical triple (finite dimensional once $S'\not =S$).

Thus, with a given decorated Dynkin diagram $S$, we associate a small thin category (poset) $\Cc(S)$ whose objects are decorated sub-diagrams in $S$
 with  embedding morphisms and a polyfunctor $\stackrel{\vec c}{\rightsquigarrow}$  to the category of Lie superalgebras:
$$
\begin{array}{ccc}
   S & \rightsquigarrow & \k(\vec c) \\
\uparrow &&\downarrow\\
S' &\rightsquigarrow  & \k'(\vec c')
\end{array}
$$
The union operation $S',  S''\mapsto S'\boxplus S''=S'\cup S''$
is a binary functor $\Cc(S)\times \Cc(S)\to \Cc(S)$.
Category $\Cc(S)$ will be a key tool of our analysis.

The next statement is a supersymmetric rectification of the only  non-graded selection rule from  \cite{RV}.
\begin{lemma}
\label{non-grad-sel-rule}
Suppose that a decorated Dynkin diagram is such that
\be
\label{RVSR}
\Sg(\bt) \simeq
\begin{picture}(35,10)
\put(2,1){$\scriptstyle\blacklozenge$}
\put(27,1){$\scriptstyle\lozenge$}
 \put(7,3){\line(1,0){21}}
 \put(1,7){$\small \al$} \put(30,7){$\small \bt$}
 \end{picture}
\ee
 for some $\bt\in \hat  \Pi_\circ $.
 Then  $\k=\hat \g$  unless $\bt$ is odd and $\al$ is even.
 \end{lemma}
Other selection rules were established in \cite{AMS1,AMS2}.
\begin{lemma}
\label{isolated odd}
  Suppose that a decorated Dynkin diagram contains a sub-graph isomorphic to
\be
\label{ISO-ODD}
\begin{picture}(35,10)
\put(2,1){$\scriptstyle\lozenge$}
\put(30.5,1.5){\framebox(3,3)}
\multiput(7,3)(6,0){4}{\line(1,0){3}}
 \put(1,7){$\small \al$} \put(30,7){$\small \bt$}
 \end{picture}
\ee
 where  $\{\bt\} =\Sg(\bt)$ is an isotropic odd node, and
  $(\bt,\al)\not =0$.  Then $\k=\hat \g$.
\end{lemma}
\noindent
Let us emphasize that this lemma is  true for whatever is the sub-diagram $\Sg(\al)$ generated by $\al$.
Simply put, if a DDD produces a proper $\k$, then its  $\tau$-fixed isotropic odd white node cannot be isolated from $\hat \Pi_\bullet$ but linked to its complement in
$\hat \Pi_\circ$.
\begin{lemma}
  \label{le-b-d}
Suppose that decorated Dynkin diagram contains a sub-graph isomorphic to
\be
\label{4NODES}
\begin{picture}(350,15)
\put(160,3){\circle*{3}}
\put(161.5,3){\line(1,0){27}}

\put(188.5,1.5){\framebox(3,3)}

\put(191.5,3){\line(1,0){24}}
\put(217,3){\circle*{3}}
\put(240,1){$\scriptstyle\lozenge$}

\multiput(217,3)(6,0){4}{\line(1,0){3}}

\put(155,8){$\al$}
\put(185,8){$\bt$}

\put(212,8){$\gm$}
\put(242,8){$\sigma$}
 \end{picture}
\ee
where $\Sg(\bt)=\{\al,\bt,\gm\}$ and $(\gm,\si)\not=0$. Suppose that $\tau(\si)=\si$ and  $\al\not \in \Sg(\si)$. Then $\k=\hat \g$.
\end{lemma}
\noindent
Let us emphasise that the graphs (\ref{RVSR}, \ref{ISO-ODD},\ref{4NODES}) are meant up to isomorphism (regardless of their orientation on the plane).
The next lemma specially addresses Dynkin diagrams with even orthogonal shape of one of its tails.
\begin{lemma}
\label{sel-rul-d}
 Suppose that a decorated Dynkin diagram with a tail of shape  $\Dg$  contains one of the sub-diagrams
\be
\label{D-TAIL}
\begin{picture}(60,10)
\put(0,3){\circle*{3}}
\put(1.5,3){\line(1,0){12}}
\put(13.5,1.5){\framebox(3,3)}
\put(16.5,3){\line(1,0){12}}
\put(30.5,3){\circle*{3}}
\put(32.5,3.5){\line(1,1){10}}
\put(32.5,3.2){\line(1,-1){10}}
\put(42,12.5){\framebox(3,3)}
\put(42,-9.5){\framebox(3,3)}
\put(42,12.5){\line(0,-1){18}}\put(45,12.5){\line(0,-1){18}}
\qbezier(46,-7)(54,3)(46,12)
\put(47,10.5){\vector(-2,3){2}}
\put(47,-5.5){\vector(-2,-3){2}}
\end{picture}
\quad\quad
\begin{picture}(55,10)
\put(0,3){\circle*{3}}
 \put(1.5,3){\line(1,0){12}}
\put(13.5,1.5){\framebox(3,3)}
 \put(16.5,3){\line(1,0){12}}

 \put(30.5,3){\circle*{3}}

\put(32.5,3.5){\line(1,1){10}}\put(32.5,3.2){\line(1,-1){10}}
\put(43,14){\circle{3}}\put(43,-8){\circle{3}}

\qbezier(46,-6)(53,4)(46,13)
\put(47,11.5){\vector(-2,3){2}}\put(47,-4.5){\vector(-2,-3){2}}
\end{picture}
\quad\quad\quad
\begin{picture}(38,10)
\put(0,3){\circle*{3}}
 \put(1.5,3){\line(1,0){12}}

 \put(14,1.5){\framebox(3,3)}

\put(17.5,3.5){\line(1,1){10}}\put(17.5,3.2){\line(1,-1){10}}
\put(28,14){\circle{3}}\put(28,-8){\circle*{3}}

\end{picture}
\ee
 Then $\k=\hat \g$.
\end{lemma}
\noindent
Note with care that, contrary to the even case, topologically isomorphic Dynkin diagrams do not imply isomorphism of the algebras they describe.
For instance   Lemma \ref{sel-rul-d} is not applicable  to diagram
$
\quad\begin{picture}(30,15)
\put(0,3){\circle{3}}
 \put(1.5,3){\line(1,0){12}}

 \put(14,1.5){\framebox(3,3)}

\put(17.5,3.5){\line(1,1){10}}\put(17.5,3.2){\line(1,-1){10}}
\put(28,14){\circle*{3}}\put(28,-8){\circle*{3}}

\end{picture}
$
despite it topologically coincides with one in (\ref{D-TAIL}). These two diagrams generate subalgebras with drastically different properties.
\begin{definition}
We call a triple  $(\hat \g,\l,\tau)$ and the corresponding decorated Dynkin diagram trivial if the subalgebra $\k$
they generate coincides with $\hat \g$ for all values of mixture parameters $c_\al\in \C^\times$, $\grave c_\al \in \C$,  $\al\in \hat \Pi_\circ$.
\end{definition}

We will say that a decorated diagram $S$ violates selection rules if either
 $S$ contains a sub-diagram (\ref{RVSR}) distinct from
$
\simeq \begin{picture}(15,10)
\put(2,3){\circle*{3}}
\put(3,3){\line(1,0){7}}
\put(10.5,1.5){\framebox(3,3)}
\end{picture}
$
 or one of the sub-graphs (\ref{ISO-ODD}), (\ref{4NODES}), (\ref{D-TAIL}).
These four lemmas can be reworded as a single statement.
 \begin{propn}
\label{violating SR implies triviality}
  A decorated Dynkin diagram is trivial if it violates selection rules.
\end{propn}
\noindent
The proof reduces to the finite dimensional case. It was done in   \cite{AMS1},  Proposition 4.24, for a special polarization
and it works for the general case too.

Thus a triple $(\hat\g,\l,\tau)$ (respectively, the decorated Dynkin diagram) is trivial if for each vector of mixture parameters there is
a white simple root $\al$ such that $e_\al,f_\al\in \k$.
The converse to Proposition  \ref{violating SR implies triviality} is also true. Its proof takes the rest of this presentation.
\begin{definition}
  Decorated Dynkin diagrams that obey the selection rules are called  Satake diagrams.
\end{definition}
\noindent
It is clear that if a DDD is Satake, then its every sub-diagram is Satake too.

\section{Spherical subalgebras in $\g$}
\label{Sec_Fin-Dim}
In this section we recall Satake diagrams  for finite dimensional $\g$
based on \cite{AMS2}.
The purpose of this piece is to work out a compact presentation without appealing to the selection rules,
which lead to somewhat cumbersome classification  especially for orthosymplectic $\g$.

For all $\g$, an admissible involution $\tau$ preserves the shaft part $D_\s\subset D_\g$ of the Dynkin diagram. Let $\tau_\s$ denote its restriction to $D_\s$.
We consider separately two cases: when  $\tau_\s$ is identical or not.
The second case is possible only for general linear $\g$.

\subsection{Nonidentical $\tau_\s$ (general linear $\g=\s$)}
General linear Lie superalgebras have two types of polarization, even and odd, depending on the number of odd roots in the basis $\Pi$.
If it is even,   the permutation $\tau$ is an involutive isometry on $\Pi_\circ$ (not necessarily on $\Pi$). For odd polarization, only $\theta\in \Aut(\Rm)$
is an (total) isometry, while $\tau$ is not even on $\Pi_\circ$.

Suppose first that $\tau\not =\id$. If $\Pi_\bullet=\varnothing$ and   $\rk\>\g =2m-1$, then
the Satake diagram looks as
shown on the left in (\ref{GL-I}). The nodes linked with arcs have the same parity.
The  node  $\al_m=\tau(\al_m)$ can only be even.
The sub-diagram   $D_\l\subset D_\g$ contains a connected component. If it is not empty, then the Satake diagram looks as shown on the right in (\ref{GL-I}).
\be
\begin{picture}(120,30)
\put(-2,1){$\scriptscriptstyle\lozenge$}\put(1.5,3){\line(1,0){12}}\put(13,1){$\scriptscriptstyle\lozenge$}
\put(16.5,3){\line(1,0){9}}\put(28,0){$\cdots$} \put(44.5,3){\line(1,0){9}}
\put(53,1){$\scriptscriptstyle\lozenge$}\put(56.5,3){\line(1,0){12}}\put(68,1){$\scriptscriptstyle\lozenge$}

\put(-2,27){$\scriptscriptstyle\lozenge$}\put(1.5,29){\line(1,0){12}}\put(13,27){$\scriptscriptstyle\lozenge$}
\put(16.5,29){\line(1,0){9}}\put(28,26){$\cdots$} \put(44.5,29){\line(1,0){9}}
\put(53,27){$\scriptscriptstyle\lozenge$}\put(56.5,29){\line(1,0){12}}\put(68,27){$\scriptscriptstyle\lozenge$}

\put(71.5,29){\line(1,-1){12}}
\put(71.5,3){\line(1,1){12}}\put(84.5,16){\circle{3}}

\put(-2,20){\vector(1,3){2}}\put(-1.8,11){\vector(1,-3){2}}\put(13,20){\vector(1,3){2}}\put(13.2,11){\vector(1,-3){2}}
\put(53,20){\vector(1,3){2}}\put(53.2,11){\vector(1,-3){2}}\put(68,20){\vector(1,3){2}}\put(68.2,11){\vector(1,-3){2}}
\qbezier(0,6)(-5,16)(0,26)\qbezier(15,6)(10,16)(15,26)\qbezier(55,6)(50,16)(55,26)\qbezier(70,6)(65,16)(70,26)
\put(90,15){$\scriptstyle \al_m$}
\end{picture}
\quad
\begin{picture}(80,32)
\put(-2,1){$\scriptscriptstyle\lozenge$}\put(1.5,3){\line(1,0){12}}\put(13,1){$\scriptscriptstyle\lozenge$}
\put(16.5,3){\line(1,0){9}}\put(28,0){$\cdots$} \put(44.5,3){\line(1,0){9}}
\put(53,1){$\scriptscriptstyle\lozenge$}\put(56.5,3){\line(1,0){12}}\put(68,1){$\scriptscriptstyle\blacklozenge$}
\put(70,28){\line(0,-1){5}}\put(70,4){\line(0,1){5}}
\put(68.5,11){\vdots}
\put(-2,27){$\scriptscriptstyle\lozenge$}\put(1.5,29){\line(1,0){12}}\put(13,27){$\scriptscriptstyle\lozenge$}
\put(16.5,29){\line(1,0){9}}\put(28,26){$\cdots$} \put(44.5,29){\line(1,0){9}}
\put(53,27){$\scriptscriptstyle\lozenge$}\put(56.5,29){\line(1,0){12}}\put(68,27){$\scriptscriptstyle\blacklozenge$}

\put(-2,20){\vector(1,3){2}}\put(-1.8,11){\vector(1,-3){2}}\put(13,20){\vector(1,3){2}}\put(13.2,11){\vector(1,-3){2}}
\put(53,20){\vector(1,3){2}}\put(53.2,11){\vector(1,-3){2}}
\qbezier(0,6)(-5,16)(0,26)\qbezier(15,6)(10,16)(15,26)\qbezier(55,6)(50,16)(55,26)
\put(74,14){$\scriptstyle D_\l$}
\put(53,34){$\scriptstyle \al_m$}

\end{picture}
\label{GL-I}
\ee
For an even polarization of $\l$, the  white nodes adjacent to $D_\l$ have the same parity, otherwise their parities are different.
All other pairs of nodes (isolated from $\Pi_\bullet$) connected with arcs have the same parities.
If the polarization of $\l$ is asymmetric, then $\tau$ does not preserve $\Pi$ and takes $D_\g$ to an isomorphic Dynkin diagram.

Let $\al_m$ be the white root preceding the black block, counting from the left. The involution $\theta$ acts on the weights of $V$
by the assignment $\theta(\zt_i)= \zt_{i'}$, $i\leqslant m$, and $\theta(\zt_i)=\zt_i$, $m<i<m'$.
The weights $\zt_i$ and $\zt_{i'}$ with $i\leqslant m$ have the same parity. The parity of weights $\zt_i$ with  $m<i<m'$ is arbitrary.

If $|\Pi|\geqslant 3$, $|\Pi_\bullet|=0$,  and $\al_m=\tau(\al_m)$ as in the left diagram,  then the root $\al_m$ is necessarily even.

Finally, let us comment on the special case  $|\Pi|=|\Pi_\circ|=1$. It should  be attributed to the class of diagrams with $\tau=\id$, and the only node can be of any parity.
However, unlike other diagrams with $\tau =\id$, the diagram
$
\begin{picture}(7,10)
\put(1.5,1.5){\framebox(3,3)}
\end{picture}
$
has an odd (!) adjoint invariant and no twisted, see Section \ref{Sec_twisted-shaft}.

\subsection{Identical $\tau_\s$}
\label{algorithm}
Finite dimensional orthosymplectic Satake diagrams with $\tau_\s=\id$ can be constructed via the following inductive procedure that
extends a  sub-diagram $S_\ell$ of white rank $\ell$ by a sub-diagram of white rank 1 via the operation $\boxplus$.
At each step, the initial diagram has either white or even black node on the left. This affects the type of  transition,
as explained below.

 For the base of induction we take the empty sub-diagram $\g$ of shape $\Ag$ and  a sub-diagram $S^r$
of white rank $1$ or $2$ (for shape $\Dg$) that contains the tail.
If the tail nodes are of the same colour, then we  take for $S^r$ the minimal Satake sub-diagram of positive white rank that contains $ D_\t$.
If the  tail nodes are of different colour (for shape $\Dg$), we take for $S^r$ the sub-diagram
that comprises  the white tail node and the white one in $D_\s$ that is adjacent to it. Specifically,
\begin{itemize}
  \item
For black tails: $S^r=\Sg(\al)$, where $\al$ is the rightmost white node.
\item
For white tails: $S^r=\Sg(\Pi_\t)$.
\item
For a mixed colour tail with white $\al_n$: $S^r=\Sg( \al_{n-2},\al_n)$.

\end{itemize}
We call $S^r$ the right end sub-diagram and list all admissible  $S^r$  below.
The quantity $\eps^r$ stands for the signature $C_{1',1}$ of the invariant $\g$-form $C$ for orthosymplectic $\g$ (meaning $S=S^r$).
The signs above the  tail roots indicate relation $\mu=\pm \la$ between eigenvalues of the standard invariant $\Ic$ associated with the diagram,
see Section \ref{Sec_standard}.
Big black rhombi $\blacklozenge$ designate   arbitrary finite dimensional  Dynkin sub-diagrams of black nodes whose shape is either arbitrary or clear from the context. In this section,
they are is used for the orthosymplectic shapes.
\begin{itemize}
\item $\Ag$-type (empty tail):
$\varnothing$
\item
Black tail:
$
\begin{picture}(18,10)
\put(0,0.5){$\scriptstyle\lozenge$}
\put(4,3){\line(1,0){7}}
\put(10.5,-0.5){$\blacklozenge$}
\put(-5,8){$\scriptstyle -\eps^r$}
\end{picture}
\quad
\begin{picture}(28,10)
\put(11,8){$\scriptstyle \eps^r$}
\put(2,3){\circle*{3}}
\put(4,3){\line(1,0){7}}
\put(10,0.5){$\scriptstyle\lozenge$}
\put(14,3){\line(1,0){7}}
\put(20.5,-0.5){$\blacklozenge$}
\end{picture}
$
  \item$\Bg$-shape white tail:
$
\begin{picture}(15,10)
\put(0,0.5){$\scriptstyle\lozenge$}
\put(-5,8){$\scriptstyle -\eps^r$}
\end{picture}
\quad
\begin{picture}(15,10)
\put(2,3){\circle*{3}}
\put(3,2){\line(1,0){7}}
\put(3,4){\line(1,0){7}}
\put(6.5,1){$\scriptstyle >$}
\put(11.5,0.5){$\scriptstyle\lozenge$}
\put(11,8){$\scriptstyle \eps^r$}
\end{picture}
$
 \item $\Cg$-shape white tail:
$
\begin{picture}(15,10)
\put(3.5,3){\circle{3}}
\put(-4,8){$\scriptstyle -1$}
\end{picture}
\quad
\begin{picture}(15,10)
\put(1.5,3){\circle*{3}}
\put(5.5,2){\line(1,0){7}}
\put(5.5,4){\line(1,0){7}}
\put(2.5,1){$\scriptstyle <$}
\put(13.5,3){\circle{3}}
\put(6,8){$\scriptstyle -1$}
\end{picture}
$

\item $\Dg$-shape white tail, $\tau=\id$:
$
\begin{picture}(10,28)
 \put(3,14){\circle{3}}\put(3,-8){\circle{3}}
\put(-3,20){$\scriptstyle -1$}
\end{picture}
\quad
\begin{picture}(15,28)
\put(0,3){\circle*{3}}
\put(0,3.5){\line(1,1){11}}\put(0,3.5){\line(1,-1){11}}
\put(11,13){\framebox(3,3)}\put(11,-8){\framebox(3,3)}
\put(11,13){\line(0,-1){18}}\put(14,13){\line(0,-1){18}}
\put(6,20){$\scriptstyle -1$}
\end{picture}
\quad
$
\item Mixed colour tail:
$
\begin{picture}(25,25)
\put(3,1.6){\framebox(3,3)}
    \put(6.5,3.5){\line(1,1){10}}\put(6.5,3.2){\line(1,-1){10}}
\put(17.5,14){\circle{3}}\put(17.5,-8){\circle*{3}}
\put(-3,8){$\scriptstyle -1$}
\end{picture}
\quad
\begin{picture}(50,25)
\put(0,3){\circle*{3}}
 \put(1,3){\line(1,0){12}}
\put(15,3){\circle{3}}
    \put(16.5,3.5){\line(1,1){10}}\put(16.5,3.2){\line(1,-1){10}}
\put(27.5,14){\circle{3}}\put(27.5,-8){\circle*{3}}
\put(7,8){$\scriptstyle -1$}
\end{picture}
\quad\quad
$
\item $\Dg$-shape twisted white tail ($\tau\not =\id$):
$
\begin{picture}(22,30)
 \put(0,13){$\scriptstyle\lozenge$}\put(0,-10){$\scriptstyle\lozenge$}
\multiput(2.5,-3.5)(0,4){4}{\line(0,1){3}}
\put(7.6,13){\vector(-2,3){2}}\put(7.6,-6){\vector(-2,-3){2}}
\qbezier(7,-7)(14,4)(7,14)
\put(-7,20){$\scriptstyle -1, \> even$}
\end{picture}
\quad
\begin{picture}(22,30)
\put(0,3){\circle*{3}}
    \put(1.5,3.5){\line(1,1){10}}\put(1.5,3.2){\line(1,-1){10}}
 \put(11,13){$\scriptstyle\lozenge$}\put(11,-10){$\scriptstyle\lozenge$}
\multiput(13.5,-3.5)(0,4){4}{\line(0,1){3}}
\qbezier(18,-7)(25,4)(18,14)
\put(18.6,13){\vector(-2,3){2}}\put(18.6,-6){\vector(-2,-3){2}}
\put(5,20){$\scriptstyle -1, \> odd$}
\end{picture}
$
\end{itemize}
Dashed lines mean no link if the tail nodes are even and  the double link if they are odd.
The sings lead to a matching condition between standard invariants (cf. Section \ref{Sec_standard}) of  left and right ends of an affine diagram in order to guarantee that
the affine spherical subalgebra $\k\subset \hat \g$ has an invariant in $\End(V_z)$.

Each step $S_\ell\to S_{\ell+1}=S_1\boxplus S_\ell$ is done with one of four shaft sub-diagrams $S_1$ of white rank 1,
depending on the connection colour  of (the leftmost node in) $S_\ell$:
\be
\label{shaft_transition}
\begin{picture}(4,10)
\put(1,3){\circle{3}}
\end{picture}
\quad\quad
\begin{picture}(15,10)
\put(2,3){\circle*{3}}
\put(10.5,1.5){\framebox(3,3)}
\put(3,3){\line(1,0){7}}
\end{picture}
\quad \quad
\begin{picture}(15,10)
\put(0.5,1.5){\framebox(3,3)}
\put(4,3){\line(1,0){7}}
\put(12,3){\circle*{3}}
\end{picture}
\quad\quad
\begin{picture}(25,10)
\put(2,3){\circle*{3}}
\put(3,3){\line(1,0){7}}
\put(12,3){\circle{3}}
\put(14,3){\line(1,0){7}}
\put(22,3){\circle*{3}}
\end{picture}
\ee
The transition rule is
\begin{itemize}
  \item if the leftmost node in $S_\ell$ is white, then $S_1$ is  one of the two diagrams on the left in (\ref{shaft_transition}) (with white node on the right); it is simply appended to $S_\ell$ with the appropriate links.
  \item if the leftmost node in $S_\ell$ is black, then it is identified with the rightmost black node of $S_1$, which is  one of the two  diagrams on the right
  in (\ref{shaft_transition}).
\end{itemize}
In the affine case, a similar classification holds for the left end of the diagram. The left end is appended to
the finite dimensional diagram using the same connection colour prescription.  This will exhaust affine diagrams with separable ends of shape
$\Bg,\Cg,\Dg$. Diagrams with inseparable ends
are processed directly in Section \ref{Sec_no shaft}.
\subsection{Shaft Satake diagrams}
 There are two special cases of Satake diagrams of shape $\Ag$ with $\tau=\id$,
\be
\label{3NODES}
\begin{picture}(7,10)
\put(1.5,1.5){\framebox(3,3)}
\end{picture}
\quad \quad
\begin{picture}(60,15)
\put(0,3){\circle*{3}}
\put(1.5,3){\line(1,0){27}}

\put(28.5,1.5){\framebox(3,3)}

\put(31.5,3){\line(1,0){24}}
\put(57,3){\circle*{3}}
 \end{picture}
\ee
They cannot be obtained by the above extension algorithm and are not subdiagrams
in a bigger finite dimensional Satake diagram.
\begin{propn}
  The algorithm above produces all non-trivial Satake diagrams with identical $\tau$ on $D_\s$ distinct from (\ref{3NODES}).
\end{propn}
\begin{proof}
  This results from examination of admissible diagrams found  in \cite{AMS2}.
\end{proof}
Let us give one more useful description of diagrams occurring for general linear $\g$.
\begin{itemize}
  \item All odd nodes are white.
  \item The even part of the diagram splits to an ordered sequence of connected components (possibly empty in the case
of two neighbouring odd roots). If an odd root occupies the leftmost (respectively rightmost position), we assume that the even connected component on
the left (respectively on the right) is  empty.
  \item
  Each even component  either consists of white nodes  or is an alternating diagram
$
 \begin{picture}(45,10)
\put(2,3){\circle*{3}}
\put(3,3){\line(1,0){5}}
\put(11,0){$\cdots$}
\put(26,3){\line(1,0){5}}
\put(32,3){\circle{3}}
\put(34,3){\line(1,0){7}}
\put(42,3){\circle*{3}}
  \end{picture}
$
of odd rank.
\item The types of even components (an empty component is regarded as white) must alternate.
\end{itemize}
This description results from application of (\ref{non-grad-sel-rule}), (\ref{isolated odd}) and (\ref{le-b-d}) and
alternatively, through  the above inductive algorithm.
\begin{definition}
A diagram of $\Ag$-type constructed by the  algorithm of Section \ref{algorithm}, is called shaft Satake diagram. Subalgebras
$\k$ they generate  are called shaft spherical subalgebras.
\end{definition}
\noindent
Remark that  $\begin{picture}(7,10)
\put(1.5,1.5){\framebox(3,3)}
\end{picture}
$
is not admissible Satake diagram relative to $\tau=\id$, because it separates two (empty) even components of the same type.

It is convenient to note that shaft diagrams have the following three types of  rightmost blocks:
\be
\label{rightmost_block}
 \begin{picture}(45,10)
\put(2,3){\circle*{3}}
\put(3,3){\line(1,0){5}}
\put(11,0){$\cdots$}
\put(26,3){\line(1,0){5}}
\put(32,3){\circle{3}}
\put(34,3){\line(1,0){7}}
\put(42,3){\circle*{3}}
  \end{picture}
\quad
 \begin{picture}(56,10)
\put(2,3){\circle*{3}}
\put(3,3){\line(1,0){5}}
\put(11,0){$\cdots$}
\put(26,3){\line(1,0){5}}
\put(32,3){\circle{3}}
\put(34,3){\line(1,0){7}}
\put(42,3){\circle*{3}}
\put(43.5,3){\line(1,0){7}}
\put(51.5,1.5){\framebox(3,3)}
  \end{picture}
\quad
 \begin{picture}(45,10)
\put(2,3){\circle{3}}
\put(3,3){\line(1,0){5}}
\put(11,0){$\cdots$}
\put(26,3){\line(1,0){5}}
\put(32,3){\circle{3}}
 \end{picture}
\ee
(formally speaking, the middle diagram has an empty block of the third type).
Possible leftmost blocks are obtained by vertical reflection.
Shaft sub-diagrams will serve as a bridge between the left and right ends in orthosymplectic affine Satake diagrams provided they can be separated from each other, see
Section \ref{Sec_separable}.

Since transition through an odd root changes the type of even blocks, shaft diagrams have extreme nodes of the same colour if and only if
the subalgebra $\s\subset \g$ has even polarization. This fact will be useful in the study of affine diagrams.

\subsection{Standard matrix invariants}
\label{Sec_standard}
Invariants of spherical subalgebras are of  two-fold importance. They can be used for verification if a given subalgebra is proper.
They are also classical limits of K-matrices solving Reflection equation upon quantization. Of course, not every invariant of the corresponding coideal
subalgebra solves quadratic Reflection equation. On the other hand, quantization
allows for different comultiplications and quasitriangular structures which eliminate
degeneration of the classical limit.
In this section we describe standardized invariants of finite dimensional spherical subalgebras, which will be used for their affine extensions.
For finite dimensional $S$, standard invariants can be viewed through a functor from $\Cc(S)$ to a category whose objects are
representations of the corresponding spherical pairs $(\g,\k)$ of minimal dimension for $\g$, along with
$\k$-fixed matrices.
If a diagram has  zero white rank, then  the standard invariant is taken to be zero matrix. This is consistent with triviality
of such pairs.

We do not mean to find a comprehensive list of invariants for all spherical subalgebras. Such a task is more interesting
for quantum K-matrices. It is sufficient
for our current goals to find at least some of them for a particular subalgebra or a particular diagram.
The key thing for us is their seriality with regard to the set
of Satake sub-diagrams: how they are built up with a growth of the white rank.
The total invariant is obtained by a gluing up the invariant of the smaller diagram with the invariant of the transition diagram
over overlapping blocks. This inductive  process is described  in this section.
\subsection{Adjoint general linear $\k$-invariants}
\label{Sec_ad-gl-in}
In this section we  recall  invariants of a spherical subalgebra $\k$ in general linear Lie superalgebra $\g$ corresponding to Satake diagrams (\ref{GL-I}) with $\tau\not =\id$.
We consider the adjoint action of $\k$ on $\End(V)$.

The diagrams (\ref{GL-I}) of white rank $\ell$  can be split as $S_{\ell+2}=\Sg(\al_1)\boxplus S_\ell$.
\be
\begin{picture}(110,40)
\put(20,12){$S_{\ell+2}=$}
\put(30,16){\oval(25,40)}

\put(58,1){$\scriptscriptstyle\lozenge$}\put(61.5,3){\line(1,0){12}}
\put(58,27){$\scriptscriptstyle\lozenge$}\put(61.5,29){\line(1,0){12}}
 \put(84,16){\oval(20,40)}
 \put(58,20){\vector(1,3){2}}\put(58.2,11){\vector(1,-3){2}}
\qbezier(60,6)(55,16)(60,26)

\put(78,12){$S_{\ell}$}
\put(58,34){$\scriptstyle \al_1$}
\put(105,12){with}
\end{picture}
\quad
\begin{picture}(90,40)
\put(64,22){$\scriptstyle \al_1$}
\put(30,16){\oval(25,40)}
\put(25,12){$S_{1}$}
\put(50,12){$=$}
 \put(65,15){\circle{3}}
\put(85,12){ or }
\end{picture}
\quad
\begin{picture}(80,40)
\put(30,16){\oval(25,40)}
\put(25,12){$S_{2}$}
\put(50,12){$=$}
\put(63,1){$\scriptscriptstyle\lozenge$}\put(66.5,3){\line(1,0){12}}\put(78,1){$\scriptscriptstyle\blacklozenge$}
\put(80,28){\line(0,-1){5}}\put(80,4){\line(0,1){5}}
\put(78.5,11){\vdots}
\put(63,27){$\scriptscriptstyle\lozenge$}\put(66.5,29){\line(1,0){12}}\put(78,27){$\scriptscriptstyle\blacklozenge$}
\put(63,20){\vector(1,3){2}}\put(63.2,11){\vector(1,-3){2}}
\qbezier(65,6)(60,16)(65,26)
\put(63,34){$\scriptstyle \al_1$}
\end{picture}
\ee
Empty $\Pi_\bullet$ in $S_2$ also admissible.
This process  inductively defines two series of diagrams starting with either $S_1$ or $S_2$.
In each case the transition is implemented by appending the diagram $\Sg(\al_1)$ of white rank two.

Let the rank $\rk \l$ of the black block in $S_2$ be $l-1$. The (adjoint)  standard invariants is normalized to have the zero lowest diagonal entry
(trivial scalar invariants subtracted). They are
$$
\Ic_{\ell}=
 \left(
\begin{array}{ccc}
\mu+\la &  & a_\ell\\
 & \Ic_{\ell-2} & \\
-\frac{\mu \la}{a_\ell}  &  & \\
\end{array}
\right),
\quad
\Ic_{1}=
 \left(
\begin{array}{cc}
\mu+\la & a_1\\
-\frac{\mu \la}{a_1}  & \\
\end{array}
\right),
\quad
\Ic_{2}=
 \left(
\begin{array}{ccc}
\mu+\la &  & a_2\\
 & \la 1_l   & \\
-\frac{\mu \la}{a_2}  &  & \\
\end{array}
\right).
$$
Here  $\la$ and $\mu$ are arbitrary non-zero complex numbers (the eigenvalues), $1_l$ stands for the identity matrix of size $l$.
Along with the parameter $a_\ell$,  they are related with the mixture parameters of $\k$.

Simple root vectors and mixture parameters can be chosen in the form
$$
x_k=e_{k,k+1}+c_k e_{k',(k+1)'},\quad k=1,1',\quad c_1=c_{1'}=\frac{a_{\ell}}{a_{\ell-2}},
$$
for $\al_1$ the node in the transition diagram on the left. For the two base diagrams with  $\ell=1,2$, they are
$$
x_1=e_{1,2}+c_1 e_{1',2}, \quad x_{1'}=e_{2',1'}+c_{1'} e_{2',1}, \quad \la= c_{1'} a_\ell,
\quad
\mu
=-c_1 a_\ell,
\quad \ell =2.
$$
For the diagram with $\al_1=\tau(\al_1)$  the    mixed  generator can be taken as
$$
x_1=(-1)^{\bar 1+\bar 1'}e_{1,1'}+c_1^2 e_{1',1},   \quad
\mu=-\la, \quad
\la^2=(-1)^{\bar 1+\bar 1'}a_\ell^2 c_1^2,
\quad \ell=1.
$$
One can add  a scalar matrix to the standard invariant. This will make sense in the affine setting.

Remark that transition from $\Ic_\ell$ to $\Ic_{\ell+2}$ can be described as gluing up $\Ic_{\ell}$ with the standard invariant
of the transition Satake diagram $\Sg(\al_1)$ (a matrix of size $4$). Such an interpretation is redundant in this particular situation because the transition diagram $S_1$ is unique.
But it will be of use for twisted invariants and for orthosymplectic $\g$ where the number of transition diagrams will  be  4.

\subsection{Twisted general linear  $\k$-invariants}
Spherical subalgebras of $\Ag$-shape with identical $\tau$, with only one exception,  have twisted (even) matrix invariants satisfying
\be
\label{t-twisted inv}
x\Ic+ \Ic x^t=0, \quad x\in \k.
\ee
where $t$ is the matrix super-transposition. The only exceptional case is $\k\subset \g\l(1|1)$ corresponding to left diagram in
(\ref{3NODES}). This subalgebra has an odd adjoint invariant and no twisted.
Next we  describe the inductive algorithm of constructing twisted invariants for the shaft Satake diagrams.

Each of the four elementary shaft diagrams (\ref{shaft_transition}) generates a subalgebra $\k=\k(c)$ acting
on $\C^{m_1|m_2}$. The vectors $(m_1,m_2)$ of dimensions assigned to the diagrams are
\be
\label{elem shaft diagr}
\begin{picture}(4,10)
\put(1,3){\circle{3}}
\put(-1,7){$\scriptstyle \al$}
\end{picture}
\mapsto (1,1),
\quad
\begin{picture}(15,10)
\put(2,3){\circle*{3}}
\put(10.5,1.5){\framebox(3,3)}
\put(3,3){\line(1,0){7}}
\put(10,7){$\scriptstyle \al$}
\end{picture}
\mapsto (2,1)
,\quad
\begin{picture}(15,10)
\put(0.5,1.5){\framebox(3,3)}
\put(4,3){\line(1,0){7}}
\put(12,3){\circle*{3}}
\put(0,7){$\scriptstyle \al$}
\end{picture}
\mapsto (1,2)
,
\quad
\begin{picture}(25,10)
\put(2,3){\circle*{3}}
\put(4,3){\line(1,0){7}}
\put(12,3){\circle{3}}
\put(14,3){\line(1,0){7}}
\put(22,3){\circle*{3}}
\put(10,7){$\scriptstyle \al$}
\end{picture}
\mapsto (2,2)
\ee
The mixed generator can be chosen as
$
x_\al=e_\al+(-1)^{\bar \al} c_\al f_{\tilde \al},
$
with $f_{\tilde \al}=f_{\al}$ for the first diagram,
$f_{\tilde \al}=[f_{\bt},f_\al]$ for the only black root $\bt\in \Pi_\bullet$, and
$f_{\tilde \al}=[f_{\bt_2},[f_{\bt_1},f_\al]]$ for $\bt_1,\bt_2\in \Pi_\bullet$.
With the representation
$$
f_{\al_i}\mapsto e_{i+1,i}, \quad  e_{\al_i}\mapsto (-1)^{\bar i} e_{i,i+1},
$$
the invariants are block diagonal matrices
\be
\label{shaft-trns-inv}
 \left(
\begin{array}{cc}
a_1 \varsigma_{m_1} & \\
  & a_2\varsigma_{m_2}
\end{array}
\right),
\quad a_1,a_2\in \C^\times,
\ee
where
each  matrix $\varsigma_{m_i}$ is of size $m_i$; specifically,  $\varsigma_1=1$ and $\varsigma_2=\left(
\begin{array}{ccc}
  0 & 1 \\
  -1 & 0
\end{array}
\right)
$.
The relation between $c$ and $a_1,a_2$ is then $a_1=c_\al a_2$.

Each diagram  $S_{\ell+1}$ is obtained by appending an elementary diagram $S_1$ to already constructed $S_{\ell}$ with $\ell\geqslant 0$.
The $S_\ell$ invariant is  a block diagonal matrix
$$
\Ic_\ell=
 \left(
\begin{array}{cc}
a \varsigma_{m} & \\
  & \Ic_{\ell-1}\\
\end{array}
\right)
\in \End(\C^{m}\oplus \C^{l}),
$$
where $l$ is the dimension of the natural $\k_{\ell-1}$-module.
In order to  extend $\Ic_\ell$  to $\Ic_{\ell+1}$, one should identify  its upper block with the lower $m=m_2$-dimensional block of the $S_1$-invariant and
set  $a=a_2$. Since $a_2$ can be taken arbitrary and $a_1$ is uniquely determined by $a_2$ and $c_\al$, this is doable for all values of  mixture parameters.
\subsection{Orthosymplectic  $\k$-invariants}
An orthosymplectic Satake $S_{\ell+1}$ diagram of white rank $\ell+1$ is obtained  by extension of $S_\ell$
with one of the four transition diagrams (\ref{shaft_transition}) of white rank 1, call it  $S_1$.
The base of this procedure is one of right end  diagrams $S^r$ listed in Section \ref{algorithm}.

The standard form of orthosymplectic $S_\ell$-invariant $\Ic_{\ell}$  is
\be
\label{I-osp-block}
\Ic_{\ell}=
 \left(
\begin{array}{ccc}
(\mu+\la)1_{m} &  & a_\ell \si_{m}\\
 & \Ic_{\ell-1} & \\
-\frac{\mu \la}{a_\ell}\si_{m}  &  & \\
\end{array}
\right)
\in \End(\C^{m}\oplus \C^{l}\oplus \C^{m}),
\quad
\ee
where $a_\ell\in \C^\times$ is a complex parameter depending on the vector of mixture parameters of $\k_\ell$;
$m=1,2$ and $l$ depends on $\k_{\ell-1}$.
The matrix $\si_{1}$ equals $1$ and  $\si_2$ is $\left(
\begin{array}{ccc}
  1 & 0 \\
  0 & -1
\end{array}
\right)
$

The extension of  $\Ic_\ell$ to $\Ic_{\ell+1}$ is performed by identification of the blocks
in $\Ic_\ell$ and $\Ic_1$ (the $S_1$-invariant (\ref{transition-bridge}) below) acting on the subspace
$\C^{m}\oplus \C^{m} \subset \C^{m}\oplus \C^l\oplus \C^{m}$ with $m=m_2$ and $a_\ell=a_2$.

The eigenvalues $\la+\mu$ are generally not arbitrary: some base diagrams require $\mu=\pm \la$. This dependence is displayed
in Section \ref{algorithm} and is attributed to the entire series of $\Ic_\ell$.
Further we describe the transition and base invariants for each of the diagrams in Section \ref{algorithm}.

\subsubsection{Shaft transition matrix}
Here we describe invariants of subalgebras in $\s$ generated by
(\ref{shaft_transition}) viewed  as transition diagrams in the extension algorithm.
Their twisted invariants in $\End(\C^{m_1}\oplus \C^{m_2})$ are given in (\ref{shaft-trns-inv}) with
the  vector of dimensions $(m_1,m_2)$ taken from (\ref{elem shaft diagr}).
We construct their embedding in $\End(\C^{m_1}\oplus \C^{m_2}\oplus \C^{m_2}\oplus \C^{m_1})$ as
a subspace in the basic module of orthosymplectic  $\g$.

We take for white node the root $\al=\al_i$, $1\leqslant i\leqslant n-2$ for a $\Dg$-shape diagram and $1\leqslant i\leqslant n-1$ otherwise.
 The algebra $\k=\k(c_\al)$ is generated over the Levi core (the black subalgebra) by $x_\al=e_\al+c_\al f_{\tilde \al}$.
Its  standard invariant is
\be
\label{transition-bridge}
\Ic_1=
\left(
\begin{array}{cccc}
(\mu+\la) 1_{m_1}&& & a_1 \si_{m_1} \\
 & (\mu+\la)1_{m_2}&  a_2\si_{m_2} & \\
&-\frac{\mu \la}{a_2}\si_{m_2} &&\\
-\frac{\mu \la}{a_1}\si_{m_1} &&&
\end{array}
\right),
\ee
where $c_\al=\pm\frac{a_2}{a_1}$, and  the sign depends on the diagram.
The eigenvalues $\mu,\la$ are arbitrary.
They are equated to eigenvalues of $\Ic_\ell$, while the parameter $a_2$ is identified with $a_\ell$.
There is no ambiguity in this identification because $\Ic_1$ depends on $\la$ and $\mu$ through symmetric functions.
\subsubsection{Tail base invariants}
\label{Sec-Tail-base-Inv}
All mixed generators are taken in the form
$$x_{\al_i}=e_{\al_i}+c_{\al_i}f_{\tilde \al_i}.$$
The basic representation of root vectors is as in Section \ref{Sec_Basic Rep}.
The signature of the invariant form $C$ will be presented as a sequence of signs $(\eps_1,\ldots, \eps_n)$.

\noindent
\underline{\bf Black tails.} The standard invariant $\Ic_1\in \End(\C^{m+l})$  for black tailed diagrams has the form (\ref{I-osp-block})
with $\Ic_0=\la 1_{l}$, where $l$ is the dimension of the natural $\l$-module.
The dimension $m$ equals $1$ if the the leftmost node is white and 2 otherwise.
$$
\begin{array}{rllllccc}
  \begin{picture}(18,10)
\put(1,8){$\scriptscriptstyle \al$}
\put(0,1){$\scriptstyle\lozenge$}
\put(4,3){\line(1,0){7}}
\put(10.5,-0.5){$\blacklozenge$}
\end{picture}
:
&
\mu=-\eps_1\la,
&
 a=-\eps_{2}(-1)^{\bar 1(\bar 1+\overline{2})}\la c_\al^{-1},
&
\frac{\mu \la}{a} = (-1)^{\overline{2}(\bar 1+\overline{2})}\la c_\al,
&\\[5pt]
\begin{picture}(28,10)
\put(11,8){$\scriptscriptstyle \al$}
\put(3,3){\circle*{3}}
\put(4,3){\line(1,0){7}}
\put(10,1){$\scriptstyle\lozenge$}
\put(14,3){\line(1,0){7}}
\put(20.5,-0.5){$\blacklozenge$}
\end{picture}
:
&
\mu =\la\eps_{1},
& a=-\eps_{1}(-1)^{\bar 1(\bar 1+\bar 2)}\la c_\al^{-1}
,
&
\frac{\mu \la}{a}=-(-1)^{\bar 2 (\bar 1+\bar 2)}\la c_\al.
\end{array}
$$
where $\eps_1=C_{1',1}$ is the matrix entry  of the invariant form.

\noindent
\underline{\bf White $\Bg$-tail.}
The standard invariant $\Ic_1\in \End(\C^{m+1})$  for white tail diagrams of the $\Bg$-shape has the form (\ref{I-osp-block}),
where $\Ic_0=\la $.
The dimension $m$ equals $1$ if the the leftmost node is white and  2 otherwise.
$$
\begin{array}{rllllccc}
\begin{picture}(15,10)
\put(0,0.5){$\scriptstyle\lozenge$}
\end{picture}
:
   &
\mathrm{grad}=  (\ve,0,\ve),
  &
C:\>\bigl(1,1,\eps\bigr),
&
\mu=-\eps  \la,
   &
c_\al=\frac{\mu}{a},   \\
\begin{picture}(15,10)
\put(2,3){\circle*{3}}
\put(3,2){\line(1,0){7}}
\put(3,4){\line(1,0){7}}
\put(6.5,1){$\scriptstyle >$}
\put(11.5,0.5){$\scriptstyle\lozenge$}
\end{picture}
:
&
\mathrm{grad}=  (\ve,\ve,0,\ve,\ve),
&
C:\>\bigl(1,1,1,\eps, \eps\bigr),
&
\mu=\eps \la,
&
c_\al=\frac{\mu}{a},
\end{array}
$$
with $\ve\in \{0,1\}$ and $\eps =(-1)^\ve$.

\noindent
\underline{\bf White $\Cg$-tail.}
Base diagrams with white tail of shape $\Cg$ have standard invariant $\Ic_1 \in \End(\C^{m+0})$
as  in (\ref{I-osp-block}) with  $\Ic_0=0$ and $m=2$   if the leftmost node is black and $m=1$ otherwise.
$$
\begin{array}{rllllccc}
\begin{picture}(15,10)
\put(13.5,3){\circle{3}}
\end{picture}:
&
\mathrm{grad}=  (\ve,\ve),
&
C:\>\bigl(1,\eps\bigr),
&  \mu=-\la,
&c_\al= -\frac{\mu \la}{a^2},
\\
\begin{picture}(15,10)
\put(1.5,3){\circle*{3}}
\put(5.5,2){\line(1,0){7}}
\put(5.5,4){\line(1,0){7}}
\put(2.5,1){$\scriptstyle <$}
\put(13.5,3){\circle{3}}
\end{picture}
:
&
\mathrm{grad}=  (\ve,\ve,\ve,\ve),
&
C:\>\bigl(1,1,\eps,\eps\bigr),
&  \mu=-\la,
&
 c_\al= \frac{\mu \la}{a^2},
\end{array}
$$
with $\ve\in \{0,1\}$ and $\eps=(-1)^\ve$.

\noindent
\underline{\bf Twisted white $\Dg$-tail.}
Base diagrams with twisted white tail of shape  $\Dg$  have a standard invariant $\Ic_2$ as (\ref{I-osp-block}),
where the central block is
$$\Ic_{\ell-1}=
\left(
\begin{array}{cccc}
 0&  \la  \\
\mu &0
\end{array}
\right)
$$
and $m=2$   if the leftmost node is black and $m=1$ otherwise.
It satisfies
$$
x^{\vartheta}\Ic=\Ic x, \quad x\in \g,
$$
where $x\mapsto x^\vartheta$ is conjugation with the flip $v_{n}\leftrightarrow v_{n'}$.
The trivial invariant proportional to the flip matrix is subtracted.
The generator $y_n$ is proportional to $e_{nn}-e_{n',n'}$. Since $y_n^\theta=-y_n$, the skew-diagonal central block is obviously $y_n$-invariant,
as well as the entire matrix $\Ic_\ell$.

The gradings on the underlying vector space and the invariant forms are considered to be
$$
\begin{array}{rllllccc}
\begin{picture}(10,30)
 \put(0,13){$\scriptstyle\lozenge$}\put(0,-10){$\scriptstyle\lozenge$}
\multiput(2.5,-3.5)(0,4){4}{\line(0,1){3}}
\put(7.6,13){\vector(-2,3){2}}\put(7.6,-6){\vector(-2,-3){2}}
\qbezier(7,-7)(14,4)(7,14)
\end{picture}
:
&
\mathrm{grad}=  (\ve_{n-1},\ve_n,\ve_n,\ve_{n-1}),
&
C:\>\bigl(1,1,1, \eps_{n-1}\eps_n\bigr),
\\[10pt]
\begin{picture}(22,30)
\put(0,3){\circle*{3}}
    \put(1.5,3.5){\line(1,1){10}}\put(1.5,3.2){\line(1,-1){10}}
 \put(11,13){$\scriptstyle\lozenge$}\put(11,-10){$\scriptstyle\lozenge$}
\multiput(13.5,-3.5)(0,4){4}{\line(0,1){3}}
\qbezier(18,-7)(25,4)(18,14)
\put(18.6,13){\vector(-2,3){2}}\put(18.6,-6){\vector(-2,-3){2}}
\end{picture}
:
&
\mathrm{grad}=  (\ve_{n-1},\ve_{n-1},\ve_n,\ve_n,\ve_{n-1},\ve_{n-1}),
&
C:\>\bigl(1,1,1,1,\eps_{n-1}\eps_n,\eps_{n-1}\eps_n\bigr).
\end{array}
$$
where $\ve_i\in \{0,1\}$ and $\eps_i=(-1)^{\ve_i}$. The relation between eigenvalues and parameters of $\Ic$ and $\vec c$
are as follows.

\noindent
For $n=2$, even tail or $n=3$, odd tail:
$\mu=-\la, \quad c_{n-1} = -\frac{\la}{a},
\quad
c_n =  \frac{\mu}{a}.
$

\noindent
For $n=2$, odd tail or $n=3$, even tail:
$c_{n-1} = (-1)^{\ve_{n}}\frac{\la}{a},
\quad
c_n =  (-1)^{\ve_{n} }\frac{\mu}{a}.
$

\noindent
\underline{\bf White $\Dg$-tail.}
Diagrams with white tail of shape $\Dg$ and  $\tau=\id$ have a standard invariant $\Ic_2$ as (\ref{transition-bridge})
with $m_2=1$ and  $m_1=2$   if the leftmost node is black and $m=1$ otherwise.
$$
\begin{array}{rllllccc}
\begin{picture}(3,20)
\put(0,14){\circle{3}}\put(0,-8){\circle{3}}

\end{picture}
:&
\mathrm{grad}=  (\ve,\ve,\ve,\ve),
&
C:\>\bigl(1,1,1,1\bigr),
&
\mu = -\la,
&
\\[10pt]
\begin{picture}(15,20)
\put(0,3){\circle*{3}}
\put(0,3.5){\line(1,1){11}}\put(0,3.5){\line(1,-1){11}}
\put(11,13){\framebox(3,3)}\put(11,-8){\framebox(3,3)}
\put(11,13){\line(0,-1){18}}\put(14,13){\line(0,-1){18}}
\end{picture}
:&
\mathrm{grad}=  (\ve,\ve,\ve+1,\ve+1,\ve,\ve),
&
C:\>\bigl(1,1,1,1,-1,-1\bigr),
&
 \mu = -\la
\end{array}
$$
with $\ve\in \{0,1\}$.
The relation between the parameters is
$$
c_{n-1} = -\frac{a_2}{a_1}, \quad c_n = -\frac{\la^2}{a_1a_2},
\quad\mbox{and} \mbox \quad
c_{n-1} = -(-1)^{\ve}\frac{a_2}{a_1},\quad
c_n = -(-1)^{\ve}\frac{\la^2}{a_1a_2}.
$$
respectively.

\noindent
\underline{\bf Mixed colour $\Dg$-tail.}
Base diagrams with mixed colour tail have standard invariants as (\ref{transition-bridge})
with  $m_1=2$   if the leftmost node is black and $m_1=1$ otherwise. One has
$$
\begin{array}{rllllccc}
\begin{picture}(50,10)
\put(13,1.6){\framebox(3,3)}
    \put(16.5,3.5){\line(1,1){10}}\put(16.5,3.2){\line(1,-1){10}}
\put(27.5,14){\circle{3}}\put(27.5,-8){\circle*{3}}
\put(32,13){$\scriptstyle \al_{n}$}
\put(32,-10){$\scriptstyle \al_{n-1}$}
\put(12,8){$\scriptstyle \al$}
\end{picture}
:
&
\mathrm{grad}=  (\ve+1,\ve,\ve,\ve,\ve,\ve+1),
&
C:\>\bigl(1,1,1,1,1,-1\bigr),
&
\mu = -\la,
\\[10pt]
\begin{picture}(50,20)
\put(0,3){\circle*{3}}
 \put(1,3){\line(1,0){12}}
\put(15,3){\circle{3}}
    \put(16.5,3.5){\line(1,1){10}}\put(16.5,3.2){\line(1,-1){10}}
\put(27.5,14){\circle{3}}\put(27.5,-8){\circle*{3}}
\put(32,13){$\scriptstyle \al_{n}$}
\put(32,-10){$\scriptstyle \al_{n-1}$}
\put(12,6){$\scriptstyle \al$}
\end{picture}
:
&
\mathrm{grad}=  (\ve,\ve,\ve,\ve,\ve,\ve,\ve,\ve),
&
C:\>\bigl(1,1,1,1,1,1,1,1\bigr),
&\mu = -\la,
\end{array}
$$
with $\ve\in \{0,1\}$.
The parameters of $\Ic$ and $\k$ are related by
$$
c_{n-2} = -(-1)^{\ve}\frac{a_2}{a_1},
\quad
c_n = \frac{\la^2}{a_2^2}
\quad\mbox{and} \mbox \quad
c_{n-2} = \frac{a_2}{a_1}
,
\quad
c_n = \frac{\la^2}{a_2^2},
$$
respectively.

We use these invariants of  spherical subalgebras in $\g$ to construct invariants of $\hat \g$ in   the remaining three sections of the article.
\section{General linear affine Satake diagrams}
\label{Sec_GL}

In this section, $\hat \g$ is a general linear affine Lie superalgebra with the set of simple roots
$$
\al_0, \quad \al_1, \quad \ldots,  \quad \al_n.
$$
The parity of $\al_0$ equals the parity of the root $\al_1+\ldots+\al_n$, which implies
that the number of odd simple roots  in $\hat \Pi$ is even.
The Dynkin diagram $D_{\hat \g}$ of $\hat \g$ is conveniently presented as an equilateral polygon with the center
at the origin. Then $\Aut(D_{\hat \g})$ becomes the dihedral group $\Z_{n+1}\rtimes \Z_2$.

There are the following involutive elements in $\Aut(D_{\hat \g})$:
reflections with respect to axes and the central symmetry if $n\in 2\Z+1$.
If $n$ is even, then all axes are passing through a vertex, otherwise there are $\ell=\frac{n+1}{2}$ axes passing through vertices
and $\ell$ axes connecting mid-points of opposite edges.
If $n$ is odd, then $\Aut(D_{\hat \g})$ contains the central symmetry acting  as a cyclic shift  $\al_i\mapsto \al_{i+\ell\mod n+1}$.
Big rhombi $\blacklozenge$ denote finite dimensional general linear subalgebras in $\l$  of arbitrary polarization.

\subsection{Free  $\tau$-action on  $\hat \Pi_\circ$}
In this subsection we consider general linear Satake diagrams of even white rank without $\tau$-fixed white nodes.
There are  two  $\tau$-invariant black blocks, which are allowed to be empty.
\subsubsection{White rank 2}
Consider first the case when there is only one pair of white nodes.
Without loss of generality, we an assume  that $\hat \Pi_\circ=\{\al_0, \al_m\}$ and $\tau(\al_0)=\al_m$.
This corresponds to  the diagram
$$
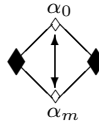
\begin{picture}(40,30)
\put(0.5,10.5){$\blacklozenge$}\put(17,-1){$\scriptscriptstyle\lozenge$}\put(17,26){$\scriptscriptstyle\lozenge$}\put(30.5,10.5){$\blacklozenge$}
\put(6,16){\line(1,1){12}}\put(6,12){\line(1,-1){12}}
\put(21,28){\line(1,-1){12}}\put(20.5,0){\line(1,1){12}}
\put(19.0,4){\vector(0,1){20}}\put(19.0,24){\vector(0,-1){20}}
\put(16,32){$\scriptstyle\al_0$} \put(16,-7){$\scriptstyle\al_m$}
\end{picture}
$$
The black blocks have the same parity of polarizations. Depending on wether they are even or odd, white nodes have either the same or different parities.
Choose $f_{\tilde \al_0}$ such that $\rho_z(f_{\tilde \al_0})=e_{n,1}$ and $e_{\tilde \al_m}=e_{1,n}$.
Set
$$
x_{\al_0}=e_{\al_0}-c_0 f_{\tilde \al_m}, \quad x_{\al_m}=e_{\al_m}-c_m f_{\tilde \al_0},
$$
$$
\rho_z(x_{\al_0})=ze_{n,1}-c_0e_{n,1}, \quad \rho_z(x_{\al_m})=\frac{1}{z}e_{1,n}-c_m e_{1,n}.
$$
With $z=c_0$ and $c_m=\frac{1}{z}$, the affine mixed generator is vanishing on $V_z$: $\rho_z(x_{\al_0})=\rho_z(x_{\al_m})=0$.
So do the Cartan generators $\rho_z(y_{\al_0})=0=\rho_z(y_{\al_m})$. A diagonal  matrix $\Ic=\la\sum_{i=1}^{m}e_{ii}+\mu \sum_{i=1}^{m}e_{jj}$ with different $\la,\mu\in \C$ is a non-trivial $\k$-invariant.

\subsubsection{White rank $\geqslant 4$}
Now suppose that the number of white pairs is two or greater:
$$
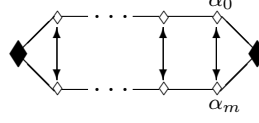
\begin{picture}(40,30)
\put(0.5,10.5){$\blacklozenge$}
\put(17,-1){$\scriptscriptstyle\lozenge$}\put(17,26){$\scriptscriptstyle\lozenge$}
\put(6,16){\line(1,1){12}}\put(6,12){\line(1,-1){12}}
\put(60.7,0.5){\line(1,0){17}}\put(60.7,28){\line(1,0){17}}

\put(19.0,4){\vector(0,1){20}}\put(19.0,24){\vector(0,-1){20}}

\put(20.7,0.5){\line(1,0){10}}\put(20.7,28){\line(1,0){10}}

\put(57,-1){$\scriptscriptstyle\lozenge$}\put(57,26){$\scriptscriptstyle\lozenge$}
\put(59.0,4){\vector(0,1){20}}\put(59.0,24){\vector(0,-1){20}}

\put(76,32){$\scriptstyle\al_0$} \put(76,-7){$\scriptstyle\al_m$}

\put(57.5,0.5){\line(-1,0){10}}\put(32.5,0){$\ldots$}\put(32.5,27.5){$\ldots$}\put(57.5,28){\line(-1,0){10}}

\put(77,-1){$\scriptscriptstyle\lozenge$}\put(77,26){$\scriptscriptstyle\lozenge$}\put(90.5,10.5){$\blacklozenge$}

\put(81,28){\line(1,-1){12}}\put(80.5,0){\line(1,1){12}}
\put(79.0,4){\vector(0,1){20}}\put(79.0,24){\vector(0,-1){20}}
\end{picture}
$$
The black blocks can have nodes of arbitrary parity and  be of arbitrary rank including zero. The $\tau$-linked white nodes have the same parity if they are adjacent to an even black block and different parities if polarization of  the even block is odd.

Assuming that $\al_0$ and $\al_m$ are white and $\al_i$, $i=m+1,\ldots, n$ are black, denote by  $S^\flat$ the a Satake sub-diagram $\Sg(\al_1,\ldots, \al_{m-1})$.
Let $\Ic$ be an invariant of the subalgebra $\k$ corresponding to $S$, and $\Xc_0=\rho_z(x_0)$, $\Xc_m=\rho_z(x_m)$ the matrices representing the mixed generators.
They have  the following block form
$$
\Ic:=
\left(
\begin{array}{cc}
  A & 0 \\
  0 & \eta \>1_k
\end{array}
\right),\quad
\mathcal{X}_m:=\left(
\begin{array}{cc}
  0 & X_m \\
  0 & 0
\end{array}
\right)
,
\quad
\mathcal{X}_0:=\left(
\begin{array}{cc}
  0 & 0 \\
  X_0 & 0
\end{array}
\right),
$$
where $A$ a sum of the standard is $\k^\flat$-invariant matrix of size $m$, $\eta \in \C$, and  $k=n+1-m$.
These matrices are presentable as
$$
A=
\left(
\begin{array}{cccccc}
 \nu+\mu+\la & 0& a \\
  0 & B&0 \\
  -\frac{\mu\la}{a}&  0&\nu
\end{array}
\right),
\quad
X_m =
\left(
\begin{array}{cccccc}
 \frac{c_m}{z} & 0& \ldots \\
  0 & 0&\ldots \\
  1&  0&\ldots
\end{array}
\right),
\quad
X_0 =
\left(
\begin{array}{cccccc}
 \vdots & \vdots& \vdots \\
  0 & 0&0 \\
  z&  0&c_0
\end{array}
\right).
$$
Here $\la+\nu$ and $\mu+\nu$ are eigenvalues and the complex number $a$ depends on the mixture parameters of $\k^\flat$.
The number of columns (rows) in the matrix $\Xc_m$ ($\Xc_0$) is $k$.
Evaluating commutators of $\Ic$ with $\Xc_m$ and $\Xc_0$
$$
\Ic \Xc_m-\Xc_m \Ic
=
\left(
\begin{array}{cc}
  0 & A X_m -\eta X_m \\
  0 & 0
\end{array}
\right),\quad
\Ic \Xc_0-\Xc_0 \Ic
=
\left(
\begin{array}{cc}
  0 & \\
   \eta X_0 - X_0 A & 0
\end{array}
\right),
$$
and equating them to zero we arrive at a system of equations
$$
(\nu+\mu+\la)\frac{c_m}{z}+a=\eta \frac{c_m}{z}, \quad -\frac{c_m}{z}\frac{\la\mu}{a}+\nu=\eta,
$$
$$
z(\nu+\mu+\la)-c_0\frac{\la\mu}{a}=\eta z, \quad  z a+ c_0\nu=c_0\eta.
$$
This system admits two solutions providing in particular  the value of the loop parameter:
$$
\eta=\la+\nu, \quad \mu\frac{c_m}{z}+a=0, \quad \Rightarrow \quad  z= -c_m\frac{\mu}{a}=c_0\frac{\la}{a}, \quad c_0\la=-c_m\mu,
$$
$$
\eta=\mu+\nu, \quad \la\frac{c_m}{z}+a=0, \quad \Rightarrow \quad z=-c_m\frac{\la}{a}=c_0\frac{\mu}{a}, \quad c_0\mu=-c_m\la.
$$
Finally, the generators $y_0$ modulo $\l$ is proportional to $e_{11}+e_{mm}$, which commutes with $A$.

\subsection{Diagrams with $\tau$-fixed white roots}
In this section we study Satake diagrams of shape $\hat \Ag$ in the situation when $\tau(\al)=\al$ for some $\al \in  \hat \Pi_\circ$.
In a presentation of $S$ by a regular polygon  on a plane, $\tau$ is a reflection around the line passing through $\al$.
 Without loss of generality we can assume $\al=\al_0$.
\subsubsection{White rank 1}
\label{Sec_WR1}
Suppose that   $\hat \Pi_\circ=\{\al_0\}$ and $|\hat \Pi_\bullet|=n$.
Its parity can be arbitrary, equal to  the parity of polarization of $\l=\g$.
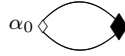
$$
\begin{picture}(40,30)
\put(3.5,12.5){$\scriptscriptstyle\lozenge$}

\qbezier(6,16.5)(19,30)(33,17)
\qbezier(6,12)(19,-1)(33,11)

 \put(30.5,10.5){$\blacklozenge$}
\put(-7,13){$\scriptstyle\al_0$}
\end{picture}
$$
The $\g$-module generated by $e_0$ is isomorphic to adjoint $\g$.
The only mixed generator of the spherical subalgebra $\hat \k$ is
$$
x_0=e_{\al_0}-c^2 f_{\tilde \al_0}, \quad c\in \C^\times.
$$
The evaluation representation on $V_z$ assigns
$$
e_0\mapsto z e_{N+1,1}, \quad f_{\tilde \al_0}\mapsto \frac{1}{z}e_{N+1,1}, \quad x_0\mapsto   (z- \frac{c^2}{z}) e_{N+1,1}.
$$
With $z=c$ the operator $x_0$ vanishes on $V_{\pm z}$ and hence on $W=V_z\tp  V_{-z}$.
The permutation $\pi\in \End(V_z\tp  V_{-z})$ is fixed by $\l$ but not by $\hat \g$:
$$
e_0\pi=z (e_0\tp 1 -  1\tp e_0)\pi=-z \pi (e_0\tp 1 -  1\tp e_0)=-\pi e_0.
$$
Therefore the subalgebra $\hat \k$ is proper in $\hat \g$.

\subsubsection{Odd white rank $\geqslant 3$}
\label{Sec_w_r=3}
Suppose that $\tau(\al_0)=\al_0$ is isolated from $\hat \Pi_\bullet$, which is allowed to be empty. Let us demonstrate that $\al_0$ must be even. Indeed, if $\al_0$ is odd, then $\zt_1$ and $\zt_{1'}=\theta(\zt_1)$ have different parities. But that is impossible because $\theta$ is an even involutive isometry. See also the left diagram in (\ref{GL-I})
and comments therein.

The Satake diagram reads
$$
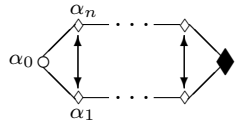
\begin{picture}(40,30)
\put(6,14){\circle{4}}
\put(17,-1){$\scriptscriptstyle\lozenge$}\put(17,26){$\scriptscriptstyle\lozenge$}
\put(6,16){\line(1,1){12}}\put(6,12){\line(1,-1){12}}

\put(19.0,4){\vector(0,1){20}}\put(19.0,24){\vector(0,-1){20}}

\put(20.7,0.5){\line(1,0){10}}\put(20.7,28){\line(1,0){10}}

\put(57.5,0.5){\line(-1,0){10}}\put(32.5,0){$\ldots$}\put(32.5,27.5){$\ldots$}\put(57.5,28){\line(-1,0){10}}

\put(-7,13){$\scriptstyle\al_0$}
\put(16,32){$\scriptstyle\al_n$} \put(16,-7){$\scriptstyle\al_1$}

\put(57,-1){$\scriptscriptstyle\lozenge$}\put(57,26){$\scriptscriptstyle\lozenge$}\put(70.5,10.5){$\blacklozenge$}

\put(61,28){\line(1,-1){12}}\put(60.5,0){\line(1,1){12}}
\put(59.0,4){\vector(0,1){20}}\put(59.0,24){\vector(0,-1){20}}
\end{picture}
$$
where the black block can be empty. It  designates a general linear subalgebra $\l$ of arbitrary polarization. The white nodes that are $\tau$-linked have same parity
unless they are adaicent to the black block with odd polarization of $\l$.

Remove the affine root and consider the remaining finite dimensional Satake diagram. Let $\k^\flat\subset\g$ be its spherical subalgebra
and  $\Ic$ the   standard adjoint $\k^\flat$-invariant as in Section \ref{Sec_ad-gl-in}:
$$
\Ic=(\mu+\la)\sum_{i=1}^{m}e_{ii} + \la\sum_{i=m+1}^{m'-1}e_{ii}+ \sum_{i=1}^{m}(a_i e_{i,i'}+a_{i'} e_{i',i}),
$$
The integer $m$ is related to the white rank $|\hat \Pi_\circ|$ of the diagram equal to  $2m+1 $.
Relations between the mixture parameters of $\k^\flat$ and the parameters of $\Ic$ are described in Section \ref{Sec_ad-gl-in}.

Take
$
x_0=e_0-c_0 f_{ 0}+\grave{c}_0 h_0$, $c\in \C^\times
$
for the  mixed affine generator of $\k$ that is represented on  $V_z$ by the assignment
$$
x_0\mapsto   ze_{1',1}- \frac{c_0}{z}e_{1,1'}+\grave{c}_0 (e_{1,1}-e_{1',1'}).
$$
The matrix $\Ic$ commutes with $x_0$ if and only if
$$
\la\mu c_0+z^2 a_1^2=0, \quad 2\mu \la \grave{c}_0\bigl(1+(-1)^{\bar\al_0}\bigr)+(\la+\mu) z a_1=0.
$$
Since $\al_0$ is even, the loop parameter is determined uniquely provided $\grave{c_0}\not=0$. Otherwise
one should set $\mu=-\la$ with arbitrary $z$. Finally, since $\tau(\al_0)=\tau_0$, the generator $y_0$ is zero.
Thus $I$ is $\k$-invariant.

\subsection{Twisted invariants for identical $\tau$}
\label{Sec_twisted-shaft}
Satake diagrams of shape $\hat \Ag$ with $\tau=\id$ generate spherical subalgebras with twisted invariants with respect to the
action (\ref{t-twisted inv}).

Removing a white node we obtain a finite dimensional Satake diagram with $\tau=\id$.
Let us first consider the special case  of the right diagram in (\ref{3NODES}), which is not extendable to a bigger finite dimensional diagram.
Its  only affine extension  is
$$
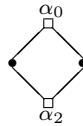
\begin{picture}(40,30)
\put(6,13.5){\circle*{3}}\put(18,-3){\framebox(3,3)}\put(18,27){\framebox(3,3)}\put(32.5,13.5){\circle*{3}}
\put(6,15){\line(1,1){12}}
\put(6,12){\line(1,-1){12}}
\put(21,27){\line(1,-1){12}}\put(21,0){\line(1,1){12}}
\put(16,32){$\scriptstyle\al_0$} \put(16,-8){$\scriptstyle\al_2$}
\end{picture}
$$
  The diagram (\ref{3NODES}) generates a subalgebra  with the only mixed generator $x_2=e_{23}+c_2e_{41}$. As argued in \cite{AMS2} it has a twisted invariant
$\Ic=e_{12}-e_{21}-c_2(e_{34}-e_{43}$).
Removing $\al_2$ we obtain an isomorphic diagram. Pick root vectors $e_{\al_0}$ and $f_{\tilde \al_0}$    represented  on $V_z$ by
$x_0=ze_{41}+\frac{c_0}{z}e_{23}$. With  the loop parameter fixed to $z=\pm \sqrt{c_0 c_2}$, the operator $\rho_z(x_0)$ is proportional to $\rho_z(x_2)$.
Then the matrix $\Ic$ is $\k_z$- and hence $\k$-invariant.

Further we examine the case when removing a white node leaves a diagram of the shaft type.
Suppose that $n>2$ and there is at least one white root. We set it to be $\al_0$.
Up to an isomorphism, $\al_0$ has the following  three possible neighbourhoods:
\be
\label{3neighbourhoods}
\begin{picture}(120,20)
\put(0.5,3){\line(1,0){12}}\put(11,1){$\scriptscriptstyle\lozenge$}\put(14.5,3){\line(1,0){12}}\put(28,3){\circle{3}}
\put(29.5,3){\line(1,0){12}}\put(41,1){$\scriptscriptstyle\lozenge$}\put(44.5,3){\line(1,0){12}}
\put(25,10){$\scriptstyle\al_0$}
\end{picture}
\begin{picture}(120,20)
\put(25,10){$\scriptstyle\al_0$}
\put(0.5,3){\line(1,0){12}}\put(11,1){$\scriptscriptstyle\lozenge$}\put(14.5,3){\line(1,0){12}}\put(26.5,1.5){\framebox(3,3)}
\put(29.5,3){\line(1,0){12}}\put(43,3){\circle*{3}}\put(44.5,3){\line(1,0){12}}
\end{picture}
\begin{picture}(120,20)
\put(25,10){$\scriptstyle\al_0$}
\put(0.5,3){\line(1,0){12}}\put(13,3){\circle*{3}}\put(14.5,3){\line(1,0){12}}\put(28,3){\circle{3}}
\put(29.5,3){\line(1,0){12}}\put(43,3){\circle*{3}}\put(44.5,3){\line(1,0){12}}
\end{picture}
\ee
The corresponding mixed operators on $ V_z$ can be chosen as
$$
x_0=z e_{n,1}+\frac{c_0}{z} e_{1,n}
, \quad
x_0=z e_{n,1}+\frac{c_0}{z} e_{1,n-1}
, \quad
x_0=z e_{n,1}+\frac{c_0}{z} e_{n-1,n-1}
.
$$
Remove $\al_0$ and get a finite dimensional Satake sub-diagram $S^\flat$ of white rank $\ell$ with the subalgebra $\k^\flat$.
The standard twisted $\k^\flat$-invariant  matrix has a block-diagonal structure with the upper  block $a_1\varsigma_{m_1}$ and lower block $a_\ell\si_{m_\ell}$
where $\varsigma_m$ are as in (\ref{shaft-trns-inv}) and where $a_i\in \C^\times $ are scalar multipliers.
In particular, $m_1=1,1,2$ and $m_\ell=1,2,2$, respectively  for the three diagrams  above.

Let $\ve$ be the parity of weight $\zt_n$.
Then the constraints
$$
bc_0=-a z^2(-1)^\ve, \quad bc_0=-a z^2(-1)^\ve, \quad bc_0=a z^2(-1)^\ve,
$$
make $\Ic$ a twisted $\k$-invariant.

If $n=2$, there are two possibilities:
$
\begin{picture}(30,20)
\put(5,17){$\scriptstyle\al_0$}
\put(8,10){\circle{3}}
\put(9.5,10){\line(1,0){16}}\put(27,10){\circle{3}}
\put(17.5,-4.5){\circle{3}}
\put(9,9){\line(2,-3){8}}\put(26.5,9){\line(-2,-3){8}}
\end{picture}
$
$\begin{picture}(30,20)
\put(5,17){$\scriptstyle\al_0$}
\put(6,8.5){\framebox(3,3)}
\put(9.5,10){\line(1,0){16}}\put(27,10){\circle*{3}}
\put(16,-5){\framebox(3,3)}
\put(8,8.5){\line(2,-3){8}}\put(27.5,9){\line(-2,-3){8}}
\end{picture}
$.
They are processed similarly to the  left and middle diagrams in (\ref{3neighbourhoods}).

Finally, suppose that $n=1$. The case $|\hat \Pi_\bullet|=1$ was considered in Secion \ref{Sec_WR1}, so we assume that all roots are white.
We have  two diagrams consisting of  even and odd roots, respectively:
\be
\label{gl-rk-1}
\begin{picture}(36,10)
\put(5.5,2){\circle{3}}

\qbezier(6,4)(19,18)(33,4)
\qbezier(6,0)(19,-13)(33,0)

 \put(33,2){\circle{3}}
\put(-7,1){$\scriptstyle\al_0$}
\end{picture}
\quad\quad\quad
\begin{picture}(36,10)
\put(4,0.5){\framebox(3,3)}

\qbezier(6,4)(19,18)(33,4)
\qbezier(6,0)(19,-13)(33,0)

 \put(31.5,0.5){\framebox(3,3)}
\put(-7,1){$\scriptstyle\al_0$}
\end{picture}
\ee
The one on the right is affinization of an exotic diagram
$\begin{picture}(7,10)
\put(1.5,1.5){\framebox(3,3)}
\end{picture}
$
from (\ref{3NODES})
that stands away from the general shaft type.
With the choice of the loop parameter $z^2=\frac{c_0}{c_1}$ the operators
$$
x_1=e_{12}+c_1e_{21}, \quad x_0=ze_{12}+\frac{c_0}{z}e_{21},
$$
are proportional. Then the standard twisted $2\times 2$-matrix invariant (\ref{shaft-trns-inv}) of diagram
$
\begin{picture}(4,10)
\put(1,3){\circle{3}}
\end{picture}
$ is invariant with respect to the even $\k$.
 Both diagrams  admit adjoint invariant $e_{12}+(-1)^{\bar \al_1}c_1 e_{21}$. Note that this matrix is odd for $\g=\widehat{\g\l(1|1)}$ corresponding
 to the diagram on the right.
We conclude that the diagrams (\ref{gl-rk-1}) are both non-trivial.
\subsection{Central symmetry}
The affine Dynkin diagram of shape $\hat \Ag$ of even rank $2\ell$ admits an involutive automorphism
$\tau\colon \al_i\mapsto \al_{i+\ell\mod 2\ell}$. If one places the roots in the vertices of a regular plane $2\ell$-gon with the
center at the origin, this transformation becomes  the central symmetry.

A Satake diagram with such $\tau$ is necessarily white:  $\hat \Pi_\bullet=\varnothing$.
The basic module $\C^N$ of  dimension $N=2\ell$ is endowed with a "quasi-cyclic" grading $\deg(v_{i+\ell})=\deg(v_i)\mod 2$ for all  $i=1,\ldots,\ell$,
with $\bar \ell+\overline {\ell+1}$ being the parity of $\al_\ell$  (and therefore of  $\al_0$).

Before we approach the diagram $S$, consider first a direct sum $\d=\g\op \g$ of general linear Lie superalgebra of rank $2\ell-2$
with roots $\al_1,\ldots, \al_{\ell-1},\al_{\ell+1},\ldots \al_{2\ell-1}$ and the following disconnected Satake diagram
\be
\begin{picture}(120,25)
\put(0,1){$\scriptscriptstyle\lozenge$}
\put(3.5,3){\line(1,0){9}}\put(15,0){$\cdots$} \put(31.5,3){\line(1,0){9}}
\put(40,1){$\scriptscriptstyle\lozenge$}

\put(60,1){$\scriptscriptstyle\lozenge$}
\put(63.5,3){\line(1,0){9}}\put(75,0){$\cdots$} \put(91.5,3){\line(1,0){9}}
\put(100,1){$\scriptscriptstyle\lozenge$}

\put(4.8,6.5){\vector(-3,-2){2}}\put(59.2,6.5){\vector(3,-2){2}}
\put(44.8,6.5){\vector(-3,-2){2}}\put(99.2,6.5){\vector(3,-2){2}}
\qbezier(4,6)(30,22)(60,6)\qbezier(44,6)(70,22)(100,6)
\put(-3,-7){$\scriptstyle \al_1$}\put(34,-7){$\scriptstyle \al_{\ell-1}$}\put(56,-7){$\scriptstyle \al_{\ell+1}$}
\put(95,-7){$\scriptstyle \al_{2\ell-1}$}
\end{picture}
\ee
where the $\tau$-linked nodes have the same parity.
Denote by $S_{2\ell}$ such a Satake diagram of rank $2\ell$.
It is clear that $S_{2\ell}=S_{2\ell-2}\boxplus S_2$ with  $S_2=\{\al_\ell, \al_{2\ell}\}$.
The basic module for $S_{2\ell-2}$ is $(\C^{\ell-1}\oplus \C)\oplus (\C^{\ell-1}\oplus \C)$ while
 $\S_2$ naturally acts on $(\C\op\C)\op (\C\op \C)$.

The study of   $\d$ reduces to a diagram of rank $2$, which we start with.
Let $\{v_1,v_2\}\subset \C^2$ be the weight basis of homogeneous elements and $\{v_{1''},v_{2''}\}\subset \C^2$ be basis in the second copy of $\C^2$.
The grading on $\C^2\op \C^2$ is subject to the conditions $\bar i= \overline {i''}$.
Regardless of the grading, we can take
$$
x_1=e_{12}+c_1 e_{2'',1''}, \quad x_{1''}=e_{1'',2''}+c_1 e_{2,1}
$$
for the mixed generators of roots $\al=\al_1$ and $\tau(\al_1)$.

We are searching for an even matrix
$$
\Ic_2=a_1e_{1,1''}+a_2 e_{2,2''}+b_1e_{1'',1}+b_2 e_{2'',2} \in \End(\C^2\oplus \C^2)
$$
that  satisfies (\ref{t-twisted inv})
where the super-transposition is pulled back via the embedding $\C^2\op \C^2\hookrightarrow \C^{4}$ of the graded vector spaces.
This embedding is identical on the first copy of $\C^2$ and assigns  $v_{i''}\mapsto v_{i+\ell}$ for $i=1,2$.
The matrix $\Ic_2$  is $y_1$-invariant, where
$$
y_1=(-1)^{\bar 1}(e_{1,1}-e_{1'',1''})-(-1)^{\bar 2}(e_{2,2}-e_{2'',2''}).
$$
It is preserved by $x_1$ and $x_{1''}$  provided
$$
a_{2}=-(-1)^{\bar 1 \bar \al}a_{1}c_1, \quad b_{2}=-(-1)^{\bar 2\bar \al } b_{1} c_1, \quad c_{1''}=(-1)^{\bar \al }c_1.
$$
Induction on $\ell$ gives a twisted  invariant  $\Ic_{2\ell-2}=\sum_{i=1}^\ell a_i e_{i,i''}+ \sum_{i=1}^\ell b_ie_{i'',i}$
with respect to the mixed generators
$$
x_i=e_{i,i+1}+c_i e_{(i+1)'',i''}, \quad x_{i''}=e_{i'',(i+1)''}+c_i e_{i+1,i},
$$
where the parameters $a_i$, $b_i$ are related with the mixture parameters by
$$
 a_{i+1}=-(-1)^{\bar i \bar\al_i}a_{i}c_i, \quad b_{i+1}=-(-1)^{ \overline {i+1} \bar \al_i} b_{i} c_i, \quad c_{i''}=(-1)^{\bar \al_i}c_i.
$$
It follows that
\be
\frac{a_{i+1}}{a_i}=\frac{b_{i+1}}{b_i}(-1)^{\bar \al_i} \Rightarrow \frac{a_{\ell}}{a_1}=\frac{b_{\ell}}{b_1}(-1)^{\bar 1+\bar \ell}.
\label{buit-up}
\ee

Now we return to the affine diagram $S$ of rank $2\ell$.
Removing  the nodes $ \al_0,\al_\ell$ from $S$ we arrive at the diagram $S_\d$.
We transfer  its  invariant $I_{2\ell-2}$ to $\End(\C^{2\ell})$ and verify  against the twisted action of the elements
$$
x_{0} =e_{2\ell,1}+c_0 z e_{\ell+1,\ell}, \quad x_{\ell}=e_{\ell,\ell+1}+\frac{c_\ell}{z} e_{1,2\ell}.
$$
The calculation reduces to a 4-dimensional vector space $\Span(v_1,v_\ell,v_{\ell+1},v_{2\ell})$ and delivers  constraints on the loop and the mixture parameters:
$$
z=c_\ell(-1)^{\bar 1(\bar 1+\bar \ell)}\frac{a_\ell}{b_1},\quad c_0=c_\ell(-1)^{\bar 1+\bar \ell},
$$
with an additional condition $b_\ell b_1=a_\ell a_1(-1)^{\bar 1+\bar \ell}$. Along with  (\ref{buit-up}), it
implies $a_1=\pm b_1$.
Its invariance with respect to the generator
$$
y_0=(-1)^{\bar 1}(e_{1,1}-e_{\ell+1,\ell+1})-(-1)^{\bar \ell}(e_{\ell,\ell}-e_{2\ell,2\ell})
$$
is verified similarly to $y_1$.
\begin{propn}
\begin{itemize}
  \item
  The affine decorated Dynkin diagram of type $\hat \Ag$  of rank $2\ell$ with  polarization $\bar \al_i=\bar \al_{i+\ell\mod 2\ell}$, empty $\hat \Pi_\bullet$, and
  $\tau(\al_i)=\al_{i+\ell\mod 2\ell}$ is non-trivial.
  \item
  It produces a proper subalgebra $\k(\vec c)$, where the mixture parameters are related by
$$
c_{i+\ell}=(-1)^{ \bar \al_i}c_i, \quad 0<i<\ell, \quad c_\ell=(-1)^{\bar 1+\bar \ell}c_0.
$$
\item
The algebra $\k(\vec c)$ has a twisted matrix invariant
$\Ic_{2\ell-2}=\sum_{i=1}^{\ell}a_i e_{i,i+\ell}+b_i e_{\ell+i,i}\in\End(V_z)$ for
$z=c_\ell(-1)^{\bar 1(\bar 1+\bar \ell)}\frac{a_\ell}{b_1}$, where $a_1=\pm b_1$ and
$
\frac{a_{i+1}}{a_i}=\frac{b_{i+1}}{b_i}(-1)^{\bar \al_i}
$
for all $i=1,\ldots,\ell-1$.
\end{itemize}

\end{propn}
\section{Orthosymplectic  Satake diagrams with $\tau_\s=\id$}
\label{Sec_OSP-id}
In this section we classify spherical subalgebras in orthosymplectic  affine Lie superalgebras in the situation
when the involution $\tau_\s$ on the  shaft part the Dynkin diagram is identical.
This includes diagrams with  black  tails.

\subsection{Matching  ends with degenerate bridge}
\label{Sec_no shaft}
Suppose that  $S$ is a Satake diagram with $\tau_\s=\id$. We described
all possible right end sub-diagrams $S^r\subset S$ in Section \ref{algorithm}. When applying those results
to affine $S$, we have to
give a special consideration to the situation
when the tails are close to each other.
With that in view, we have to modify  our break down scheme $S=S^l\cup S^m\cup S^r$.

Consider
for instance the right end of the diagram.
Suppose that its tail $D_\t$ is black and let $\Theta\subset \hat \Pi$ be the subset of white nodes that are linked to the connected component of
$\hat \Pi_\bullet$
containing  $D_\t$.
Then we set $S^r=\Sg(\Theta)$. The left end sub-diagram $S^l$ is defined similarly.

If $S^l\cap S^r=\varnothing$, we can process them separately
and solve a problem of connecting the ends together either directly or through a shaft sub-diagram.
Note that the definition of $S^l$ and $S^r$  in this case coincides with what it was for finite dimensional $S$.
A similar situation occurs when  $S^l\cap S^r\not =\varnothing$, but the intersection
$S^l\cap S^r$ contains no white node.
\begin{definition}
  We say that an orthosymplectic Satake diagram has inseparable ends
  if $S^l\cap S^r\cap \hat \Pi_\circ \not =\varnothing$.
\end{definition}
\noindent
Here we give the full list of such diagrams:
\be
\begin{picture}(45,10)
\put(0.5,-0.5){$\blacklozenge$}
\multiput(8,3)(4,0){3}{\line(1,0){2}}
\put(18.5,.5){$\scriptstyle\lozenge$}
\multiput(24,3)(4,0){3}{\line(1,0){2}}
\put(34.5,-0.5){$\blacklozenge$}
\end{picture}
\quad
\begin{picture}(25,10)
\put(0.5,-0.5){$\blacklozenge$}
\put(19.,.5){$\scriptstyle\lozenge$}
\put(7,2){\line(1,0){10}}
\put(7,4){\line(1,0){10}}
\put(13.5,1){$\scriptstyle >$}
\end{picture}
\quad
\begin{picture}(24,10)
\put(2,3){\circle{3}}
\put(3,2){\line(1,0){10}}
\put(3,4){\line(1,0){10}}
\put(9.5,1){$\scriptstyle >$}
\put(14.5,-0.5){$\blacklozenge$}
\end{picture}
\quad
\begin{picture}(22,10)
\put(12,3.){\line(-1,-1){10}}\put(12,3.){\line(-1,1){10}}
\put(1,14){\circle{3}}\put(1,-8){\circle*{3}}
\put(12.,-0.5){$\blacklozenge$}
\end{picture}
\quad
\begin{picture}(36,10)
\put(12,3.){\line(-1,-1){10}}\put(12,3.){\line(-1,1){10}}
\put(1,14){\circle{3}}\put(1,-8){\circle*{3}}
\put(13.5,3){\circle{3}}
\multiput(16,3)(4,0){3}{\line(1,0){2}}
\put(26.5,-0.5){$\blacklozenge$}
\end{picture}
\quad
\begin{picture}(30,10)
\put(12,3.){\line(-1,-1){10}}\put(12,3.){\line(-1,1){10}}
\put(1,14){\circle{3}}\put(1,-8){\circle*{3}}
\put(12.5,1.5){\framebox(3,3)}

\put(16,2){\line(1,0){10}}
\put(16,4){\line(1,0){10}}
\put(22.5,1){$\scriptstyle >$}
\put(27,1.5){\textcolor{black}{\rule{3pt}{3pt}}}
\end{picture}
\quad
\begin{picture}(37,10)
\put(12,3.){\line(-1,-1){10}}\put(12,3.){\line(-1,1){10}}
\put(1,14){\circle{3}}\put(1,-8){\circle*{3}}
\put(12.5,1.5){\framebox(3,3)}

\put(18,2){\line(1,0){10}}
\put(18,4){\line(1,0){10}}
\put(15,1){$\scriptstyle <$}
\put(28,3){\circle*{3}}
\end{picture}
\quad
\begin{picture}(22,10)
 \put(-2,13){$\scriptstyle\lozenge$}\put(-2,-10){$\scriptstyle\lozenge$}
\multiput(.5,-3.5)(0,4){4}{\line(0,1){3}}

\qbezier(-3,-7)(-10,4)(-3,14)

\put(-3.6,13){\vector(2,3){2}}\put(-3.6,-6){\vector(2,-3){2}}

\put(12,3.){\line(-1,-1){10}}\put(12,3.){\line(-1,1){10}}
\put(12.,-0.5){$\blacklozenge$}
\end{picture}
\quad
\begin{picture}(22,10)
 \put(-2,13){$\scriptstyle\lozenge$}\put(-2,-10){$\scriptstyle\lozenge$}
\multiput(.5,-3.5)(0,4){4}{\line(0,1){3}}
\put(12,3.){\line(-1,-1){10}}\put(12,3.){\line(-1,1){10}}
\put(12.,-0.5){$\blacklozenge$}
\end{picture}
\label{no-bridge}
\ee
Remark that an  even  diagram
$
 \begin{picture}(25,25)
    \put(12,3.){\line(-1,-1){10}}\put(12,3.){\line(-1,1){10}}
\put(1,14){\circle{3}}\put(1,-8){\circle*{3}}
\put(13.5,3){\circle{3}}
    \put(15,3.){\line(1,1){10}}\put(15,3.){\line(1,-1){10}}
\put(26,14){\circle{3}}\put(26,-8){\circle*{3}}

\end{picture}
$
 is isomorphic to
$
 \begin{picture}(35,10)
    \put(12,3.){\line(-1,-1){10}}\put(12,3.){\line(-1,1){10}}
\put(1,14){\circle{3}}\put(1,-8){\circle{3}}
\put(13.5,3){\circle{3}}
    \put(15,3.){\line(1,1){10}}\put(15,3.){\line(1,-1){10}}
\put(26,14){\circle*{3}}\put(26,-8){\circle*{3}}
\end{picture}
$
with  $S^l\cap S^r=\varnothing$ that falls into a class with separable ends, see Section \ref{Sec_separable}.
\begin{propn}
  The Satake diagrams (\ref{no-bridge}) are non-trivial.
\end{propn}
\begin{proof}
  Consider first the diagrams of white rank 1 and let $\al$ be the only white node.
  In this case the root vector $\rho_z(f_{\tilde \al})$ and $\rho_z(e_\al)$ are collinear. Therefore for
  each value of the mixture parameter $c_\al$ one can choose the loop parameter  to make the mixed generator $x_\al$ vanish on  $V_{\pm z}$.
  If the only white node $\al$ is affine, $\l=\g$, and there is no invariant in $\End(V_{\pm z})$. However
  it is present in $V_{z}\tp V_{-z}$: it is the permutation operator.
  If $\al$ is not affine, the diagram $D_\l$ consists of  two connected components. There is a diagonal
  $\l$-invariant in $V_z$ with two eigenvalues. It is therefore $\k$-invariant as $\rho_z(x_\al)$=0.

  We are left to work out diagrams of white rank $2$.
Consider the four diagrams of the non-twisted left shape $\Dg$. Their affine root is isolated from the black sub-diagram: $\Sg(\al)=\al$.
Removing $\al$ we obtain a finite dimensional Satake diagram whose standard invariant $\Ic$
equals
$$
\Ic=(\la+\mu) \sum_{i=1}^{2}e_{i,i}+a\sum_{i=1}^{2}e_{i,i'}-\frac{\la\mu}{a}\sum_{i=1}^{2}e_{i',i}+\ldots,
$$
possibly with a relation $\mu=\pm \la$ depending on $S^r$.  The evaluation  module $V_z$ affords a representation of the mixed affine generator $\rho_z(x_\al)=\frac{c_0^2}{z}\sum_{i=1}^{2}e_{i,i'}-z\sum_{i=1}^{2}e_{i',i}$.
Setting $z=\pm \frac{c_0\sqrt{\mu \la}}{a}$ we make $\Ic$ invariant with respect to $x_\al$.

Consider the remaining two diagrams on the right. The one with $\tau \not =\id$ allows for $\rho_z(x_{\al_0})=0=\rho(x_{\al_1})$ with a choice of the loop parameter
and an appropriate relation $c_0=\pm c_1$ depending on a scaling of the root vectors and the grading. Therefore any $\l$-invariant matrix, say, diagonal,
is also $\k$-invariant.
We are left to consider the case of $\tau=\id$.

Let us choose the root vectors as prescribed in Section \ref{Sec_Basic Rep}. and set $x_\al=e_\al+c_\al f_{\tilde \al}$, $\al=\al_0,\al_1$.
Remark that $\tilde \al_0=\zt_1-\zt_2$ and $\tilde \al_1=-\zt_1-\zt_2$.
Remove the affine root $\al_0$ from the diagram and restrict to  a finite dimensional spherical pair $(\g, \k^\flat)$.
The matrix
$$\Ic=a e_{1,1'}+\frac{\la^2}{a} e_{1',1} +\la\sum_{i=2}^{2'} e_{ii}$$
is a standard $\k^\flat$-invariant  for an appropriate choice of $a$ depending on the mixture parameter $c_1$.
It will be also $\k$-invariant for an appropriate choice of the loop parameter $z$ depending on $c_0$ and $c_1$:
the requirement is that  $x_0$ and $x_1$ generate the same mixed $\l$-submodules in $\hat \g_z$, which is feasible.
\end{proof}
\subsection{Affine diagrams with separable ends}
\label{Sec_separable}
 In this section we study  affine Satake diagrams whose left and right ends are separable.
By this we mean the  white rank of the graph $S^l\cap S^r$ is zero.
Then the intersection $S^l\cap S^r$  consists of a black node if not-empty. That node will be treated as a degenerate shaft bridge
and viewed as a connection clamp.

In order to obtain possible left end sub-diagrams,
one should take $S^r$ from Section \ref{algorithm} other than the two $\Bg$-shape diagrams and reflect them with respect to a vertical axis.
Define $S^m\subset D_\s$ the minimal Satake sub-diagram in $S$ containing $S\backslash (S^l\cup S^r$).
We say that $S^m$ is trivial if its white rank is zero.

\begin{definition}
\begin{itemize}
  \item  We say that $S^r$ (respectively $S^l$) has black/white connection if the leftmost node in $S^r$ (respectively rightmost
  in $S^l$) are of the corresponding colour.
  \item A shaft diagram has left and right black/white connections if such is the colour of the leftmost and, respectively, rightmost node.
\end{itemize}
\end{definition}
As we already mentioned, the  left and right connections in a shaft Satake diagram  have the same colour if and only if it has an even number of odd nodes.
Now we can formulate the connection rule of recovering $S$ from $S^l,S^m, S^r$:
\begin{lemma}
Let $S$ be a non-trivial Satake diagram. Then its sub-diagrams
 $S^l,S^m, S^r$ have the same type of connections.
\end{lemma}
\begin{proof}
If $S^l\cap S^r\not =\varnothing$ it is exactly the black connection node, so the statement holds true.
Suppose the opposite. First consider the case of trivial  $S^m\subset S$. If $S^l$ and $S^r$ have different
type of connections, say $S^l$ is white and $S^r$ is black, then the black node from $S^r$ is connected to white node
in $S^l$ and must be included in $S^l$, which contradicts to the assumption that $S^l$ and $S^r$ are separable.

If the shaft $S^m$ is non-trivial, then the finite dimensional  Satake sub-diagrams $S^l\boxplus S^m$ and $S^m\boxplus S^r$
satisfy the requirements.
\end{proof}

\subsection{Reconstruction of affine invariants from finite dimensional}
\label{Sec_reconstruction}

 The first  question about a subalgebra  $\k$ is if it is proper subalgebra in $\hat \g$.
The main method to answering it is to search
for a matrix invariant in some representation that is preserved by $\k$.
The simplest representation of $\hat \g$ is $V_z$ for some $z\in \C^\times$. It is therefore worthy to check if $\k_z$ is
already  proper in $\g$.
This leads us to the following categorification
of spherical pairs.
\begin{definition}
Suppose that $\k\subset \hat \g$ is spherical subalgebra corresponding to a Satake diagram $S$.
  We call $\k$ (respectively S) homey if there is $z\in \C^\times$ and non-trivial $\Ic\in \End(V_z)$ preserved by $\k$. Otherwise
  $\k$ and $S$ are called waif.
\end{definition}
In other words, $\k$ is homey if endomorphisms from  $\k_z$ commute with the  matrix $\Ic$.  As a consequence, for
every proper  diagram $S'\subset S$,  the corresponding finite dimensional subalgebra $\k'\subset \k_z$
has an invariant that extends to an invariant of $\k_z$. Since $S$ can be  in general reconstructed from its smaller parts,
the problem of finding $\k$-invariant $\Ic$ reduces to the problem of matching invariants of finite dimensional spherical
subalgebras in $\k$.

Suppose that $S^r\subset S'$ and let $\Ic^r$, $\Ic'$ be their  standard $\k^r$- and $\k'$-invariants. The diagram $S'$ is obtained from $S^r$ by
 a shaft extension, so the relation $\mu=\pm \la$ on the eigenvalues of the $\Ic^r$  is preserved for $\Ic'$. If it is compatible with
the relations between eigenvalues of $\Ic^l$ associated with $S^l$, for some choice of the mixture and loop parameters, then $\Ic'$ is extended to a $\hat \k$-invariant $\hat \Ic$.

\begin{propn}
The following Satake diagrams are homey:
\begin{itemize}
  \item[i)]
The left tail is black while the right tail is  black or of shape  $\Bg$.
  \item[ii)]
$S^l$ and $S^r$ have shape  $\Cg$ or $\Dg$ with $\tau=\id$, or $\Dg$ of  mixed colour, or white twisted even  $\Dg$ of rank $2$ or white twisted odd $\Dg$ or rank $3$.
\item[iii)]
Either $S^l$ or $S^r$ is white twisted $\Dg$-shape odd of  rank $2$ or even of  rank $3$.
\end{itemize}
\begin{proof}
For each diagram from this list, the sub-diagrams $S^l$ and $S^r$ produce subalgebras $\k^l$ and $\k^r$ respectively.
They have standard invariants $\Ic^l$ and $\Ic^r$  whose eigenvalues $\mu,\la$ are subject to relation $\mu=\pm \la$, where
the sign is indicated for a particular end in Section \ref{Sec_reconstruction}. If this sign  is the same for both ends, the matrices $\Ic^l$ and
 $\Ic^r$ can be connected via the bridge invariant $\Ic^m$.  So we need to prove the coincidence of the signs within the hypothesis.

i)    Suppose,  for instance, that both tails are black. We have four different combinations depending on the colour of the  connection nodes
of $S^l$ and $S^r$.
Let $S^l\ni \al^l\not =\al^r\in S^r$ be the white nodes.  They belong to the shaft basis of simple roots, so we can present them as
 $\al^l=\zt_i-\zt_{i+1}$ and $\al^l=\zt_j-\zt_{j+1}$ with $i+1\leqslant j$. The signatures $\eps^l=\eps_{i+1}$ and $\eps^r=\eps_{j}$ of the invariant form $C$ are $\eps^l=(-1)^{\overline{i+1}}$ and $\eps^r=(-1)^{\overline{j}}$.
Consider the following two examples
$$
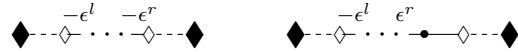
\begin{picture}(80,10)
\put(0.5,-0.5){$\blacklozenge$}
\multiput(8,3)(4,0){3}{\line(1,0){2}}
\put(18.5,.5){$\scriptstyle\lozenge$}
\put(22,3){\line(1,0){5}}
\put(30,0){$\cdots$}
\put(46,3){\line(1,0){5}}
\put(50,.5){$\scriptstyle\lozenge$}
\multiput(55.5,3)(4,0){3}{\line(1,0){2}}
\put(66,-0.5){$\blacklozenge$}
\put(20,8){$\scriptstyle -\eps^l$}\put(42,8){$\scriptstyle -\eps^r$}
\end{picture}
\quad\quad
\begin{picture}(80,10)
\put(20,8){$\scriptstyle -\eps^l$}\put(42,8){$\scriptstyle \eps^r$}
\put(0.5,-0.5){$\blacklozenge$}
\multiput(8,3)(4,0){3}{\line(1,0){2}}
\put(18.5,.5){$\scriptstyle\lozenge$}
\put(22,3){\line(1,0){5}}
\put(30,0){$\cdots$}
\put(47,3){\line(1,0){5}}
\put(53,3){\circle*{3}}
\put(65,.5){$\scriptstyle\lozenge$}
\multiput(70.5,3)(4,0){3}{\line(1,0){2}}
\put(54,3){\line(1,0){12}}
\put(81,-0.5){$\blacklozenge$}

\end{picture}
$$
The ends of the left diagram have the same  connection colour, therefore the number of odd roots in the bridge between them
is even. As a consequence, $\eps^l=\eps^r$, and the invariant of the sub-diagram $S\backslash S^l$ can be extended to an invariant
of $\k_z$ and hence of $\hat\k$. A similar reasoning applies for all types of tails from i).

ii) These diagrams suggest $\mu=-\la$ for both ends.

iii) These types of tail allow for arbitrary eigenvalues of the standard invariant, which can be adjusted
to  the constraints for the other tail.
\end{proof}
\end{propn}
In the next section we study diagrams whose end invariants have conflicting eigenvalue constraints.

\subsection{Waif $\k$}
\label{Sec_Waif}
We classify waif Satake  diagrams, that allow for no invariants in $\End(V_z)$,  in the (minimalist) toric parametrization of $\k$.
Let us  list possible combinations of left and right ends of affine diagrams whose standard invariants  may potentially conflict each other:
$$
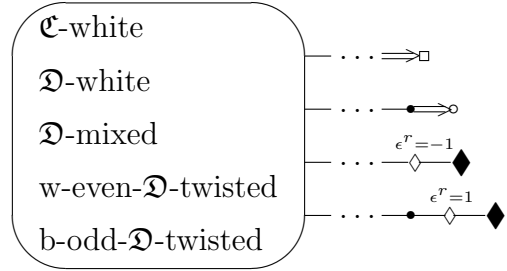
\begin{picture}(150,85)
\put(55,33){\oval(110,100)}
\put(10,70){$\Cg$-white}
\put(10,50){$\Dg$-white}
\put(10,30){$\Dg$-mixed}
\put(10,10){w-even-$\Dg$-twisted}
\put(10,-10){b-odd-$\Dg$-twisted}

\put(110,63){\line(1,0){10}}
\put(123,62){$\ldots$}
\put(153.5,61.5){\framebox(3,3)}
\put(139,62){\line(1,0){12}}
\put(139,64){\line(1,0){12}}
\put(147.5,61){$\scriptstyle >$}


\put(110,43){\line(1,0){10}}
\put(123,42){$\ldots$}
\put(139,43){\line(1,0){10}}

\put(150,43){\circle*{3}}
\put(151,42){\line(1,0){12}}
\put(151,44){\line(1,0){12}}
\put(159.5,41){$\scriptstyle >$}
\put(166,43){\circle{3}}

\put(110,23){\line(1,0){10}}
\put(123,22){$\ldots$}
\put(139,23){\line(1,0){10}}
\put(149,20.5){$\scriptstyle\lozenge$}
\put(153,23){\line(1,0){12}}
\put(165,19.5){$\blacklozenge$}
\put(144,28){$\scriptscriptstyle \eps^r=-1$}

\put(110,3){\line(1,0){10}}
\put(123,2){$\ldots$}
\put(139,3){\line(1,0){10}}
\put(150,3){\circle*{3}}
\put(151,3){\line(1,0){12}}
\put(162,0.5){$\scriptstyle\lozenge$}
\put(166,3){\line(1,0){12}}
\put(178,-.5){$\blacklozenge$}
\put(157,8){$\scriptscriptstyle \eps^r=1$}

\end{picture}
$$

\vspace{10pt}
\begin{propn}
  Only  $\Cg$-white and $\Dg$-mixed left ends may appear in waif Satake diagrams.
\end{propn}
\begin{proof}
  Let us show that the other three types of left ends cannot appear.  We conduct a detailed proof for the $\Dg$-white end only. The reasoning for the remaining
  $\Dg$-twisted ends is literally the same.  The relevant diagrams are
$$
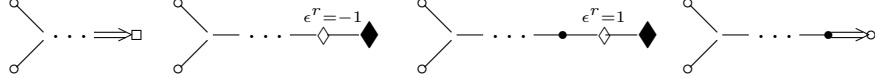
\begin{picture}(60,25)
\put(9,25){\circle{3}}\put(9,0){\circle{3}}
\put(10,24){\line(1,-1){10}}
\put(10,1){\line(1,1){10}}
\put(23,12){$\ldots$}
\put(53.5,11.5){\framebox(3,3)}
\put(39,12){\line(1,0){12}}
\put(39,14){\line(1,0){12}}
\put(47.5,11){$\scriptstyle >$}
\end{picture}
\quad
\begin{picture}(80,25)
\put(-1,25){\circle{3}}\put(-1,0){\circle{3}}
\put(0,24){\line(1,-1){10}}
\put(0,1){\line(1,1){10}}

\put(12,13){\line(1,0){10}}
\put(25,12){$\ldots$}
\put(41,13){\line(1,0){10}}
\put(51,10.5){$\scriptstyle\lozenge$}
\put(55,13){\line(1,0){12}}
\put(67,9.5){$\blacklozenge$}
\put(46,18){$\scriptscriptstyle \eps^r=-1$}

\end{picture}
\quad
\begin{picture}(100,25)
\put(-1,25){\circle{3}}\put(-1,0){\circle{3}}
\put(0,24){\line(1,-1){10}}
\put(0,1){\line(1,1){10}}

\put(12,13){\line(1,0){10}}
\put(25,12){$\ldots$}
\put(41,13){\line(1,0){10}}
\put(52,13){\circle*{3}}
\put(53,13){\line(1,0){12}}
\put(65,10.5){$\scriptstyle\lozenge$}
\put(68,13){\line(1,0){12}}
\put(80,9.5){$\blacklozenge$}
\put(59,18){$\scriptscriptstyle \eps^r=1$}
\end{picture}
\begin{picture}(80,25)
\put(-1,25){\circle{3}}\put(-1,0){\circle{3}}
\put(0,24){\line(1,-1){10}}
\put(0,1){\line(1,1){10}}

\put(12,13){\line(1,0){10}}
\put(25,12){$\ldots$}
\put(41,13){\line(1,0){10}}

\put(52,13){\circle*{3}}
\put(53,12){\line(1,0){12}}
\put(53,14){\line(1,0){12}}
\put(61.5,11){$\scriptstyle >$}
\put(68,13){\circle{3}}
 \end{picture}
$$
The suppressed bridge in the two diagrams on the left is a shaft Satake sub-diagram with an even number of odd nodes (because it has white connections on the left and on the right).
But this implies $\eps_1=-1$, which contradicts the $\Dg$-shape of the left tail. Analogously, the two diagrams on the right has a bridge with white and black connections
on the left and on the right, respectively. It is a shaft Satake sub-diagram with an odd number of odd nodes. This again implies $\eps_1=-1$, which is impossible.
\end{proof}
In particular,
the bridge can be empty or consist of one black node. Then the ends are connected directly.
Minimal end-to-end  diagrams with same connection colour:
$$
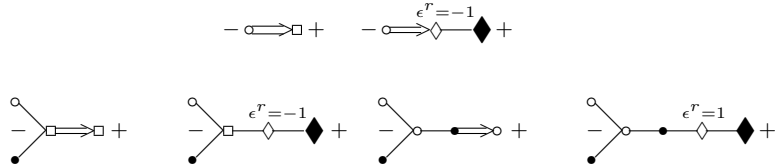
\begin{picture}(40,10)
\put(-10,1){$\scriptstyle -$}
\put(0,3){\circle{3}}
\put(1,1.5){\line(1,0){12}}
\put(1,4.4){\line(1,0){12}}
\put(10.5,1){$\scriptstyle >$}
 \put(16.5,1.5){\framebox(3,3)}
\put(22,1){$\scriptstyle +$}
\end{picture}
\quad
\begin{picture}(40,10)
\put(-10,1){$\scriptstyle -$}
\put(0,3){\circle{3}}
\put(1,2){\line(1,0){12}}
\put(1,4){\line(1,0){12}}
\put(10.5,1){$\scriptstyle >$}
\put(16,0.5){$\scriptstyle\lozenge$}
\put(20,3){\line(1,0){12}}
\put(32,-0.5){$\blacklozenge$}
\put(11,8){$\scriptscriptstyle \eps^r=-1$}
\put(41,1){$\scriptstyle +$}
\end{picture}
\quad
$$
$$
\begin{picture}(55,25)
\put(0,1){$\scriptstyle -$}
    \put(13.5,3.2){\line(-1,1){10}}\put(13.5,3.2){\line(-1,-1){10}}
\put(2.,14){\circle{3}}\put(2.,-8){\circle*{3}}

\put(14,1.5){\framebox(3,3)}

\put(17,2){\line(1,0){12}}
\put(17,4){\line(1,0){12}}
\put(26,1){$\scriptstyle >$}
\put(32,1.5){\framebox(3,3)}
\put(38,1){$\scriptstyle +$}
\end{picture}
\quad
\begin{picture}(60,25)
\put(0,1){$\scriptstyle -$}
    \put(13.5,3.2){\line(-1,1){10}}\put(13.5,3.2){\line(-1,-1){10}}
\put(2.,14){\circle{3}}\put(2.,-8){\circle*{3}}

\put(14,1.5){\framebox(3,3)}

\put(17,3){\line(1,0){12}}

\put(28,0.5){$\scriptstyle\lozenge$}
\put(32,3){\line(1,0){12}}
\put(44,-0.5){$\blacklozenge$}
\put(23,8){$\scriptscriptstyle \eps^r=-1$}
\put(54,1){$\scriptstyle +$}
\end{picture}
\quad
\begin{picture}(55,25)
\put(0,1){$\scriptstyle -$}
\put(29,3){\circle*{3}}
 \put(16,3){\line(1,0){12}}
\put(15,3){\circle{3}}
    \put(13.5,3.2){\line(-1,1){10}}\put(13.5,3.2){\line(-1,-1){10}}
\put(2.,14){\circle{3}}\put(2.,-8){\circle*{3}}

\put(29,2){\line(1,0){12}}
\put(29,4){\line(1,0){12}}
\put(38,1){$\scriptstyle >$}
\put(45,3){\circle{3}}
\put(50,1){$\scriptstyle +$}
\end{picture}
\quad
\quad
\begin{picture}(55,25)
\put(0,1){$\scriptstyle -$}
    \put(13.5,3.2){\line(-1,1){10}}\put(13.5,3.2){\line(-1,-1){10}}
\put(2.,14){\circle{3}}\put(2.,-8){\circle*{3}}
\put(15,3){\circle{3}}
 \put(16,3){\line(1,0){12}}

\put(29,3){\circle*{3}}

\put(30,3){\line(1,0){12}}
\put(41,0.5){$\scriptstyle\lozenge$}
\put(45,3){\line(1,0){12}}
\put(56,-.5){$\blacklozenge$}
\put(36,8){$\scriptscriptstyle \eps^r=1$}
\put(65,1){$\scriptstyle +$}
\end{picture}
\quad
$$
Minimal end-to-end  diagrams with different connection colour:
$$
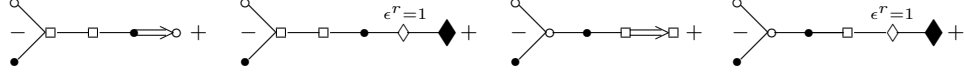
\begin{picture}(70,10)
\put(-10,1){$\scriptstyle -$}
\put(0,3){\circle{3}}
\put(1,2){\line(1,0){12}}
\put(1,4){\line(1,0){12}}
\put(10.5,1){$\scriptstyle >$}
\put(16.5,1.5){\framebox(3,3)}
\put(20,3){\line(1,0){12}}
\put(33,3){\circle*{3}}
\put(34,2){\line(1,0){12}}
\put(34,4){\line(1,0){12}}
\put(43,1){$\scriptstyle >$}
\put(50,3){\circle{3}}
\put(55,1){$\scriptstyle +$}
\end{picture}
\quad
\begin{picture}(70,10)
\put(-10,1){$\scriptstyle -$}
\put(0,3){\circle{3}}
\put(1,2){\line(1,0){12}}
\put(1,4){\line(1,0){12}}
\put(10.5,1){$\scriptstyle >$}
\put(16.5,1.5){\framebox(3,3)}
\put(20,3){\line(1,0){12}}
\put(33,3){\circle*{3}}
\put(34,3){\line(1,0){12}}
\put(45,0.5){$\scriptstyle\lozenge$}
\put(49,3){\line(1,0){12}}
\put(60,-.5){$\blacklozenge$}
\put(40,8){$\scriptscriptstyle \eps^r=1$}
\put(70,1){$\scriptstyle +$}
\end{picture}
 \quad
$$
$$
\begin{picture}(75,25)
\put(0,1){$\scriptstyle -$}
    \put(13.5,3.2){\line(-1,1){10}}\put(13.5,3.2){\line(-1,-1){10}}
\put(2.,14){\circle{3}}\put(2.,-8){\circle*{3}}

\put(14,1.5){\framebox(3,3)}
\put(18,3){\line(1,0){12}}
\put(30,1.5){\framebox(3,3)}
\put(34,3){\line(1,0){12}}

\put(47,3){\circle*{3}}

\put(47,2){\line(1,0){12}}
\put(47,4){\line(1,0){12}}
\put(56,1){$\scriptstyle >$}
\put(63,3){\circle{3}}
\put(68,1){$\scriptstyle +$}
\end{picture}
\quad
\begin{picture}(90,25)
\put(0,1){$\scriptstyle -$}
    \put(13.5,3.2){\line(-1,1){10}}\put(13.5,3.2){\line(-1,-1){10}}
\put(2.,14){\circle{3}}\put(2.,-8){\circle*{3}}

\put(14,1.5){\framebox(3,3)}
\put(18,3){\line(1,0){12}}
\put(30,1.5){\framebox(3,3)}
\put(34,3){\line(1,0){12}}

\put(47,3){\circle*{3}}

\put(48,3){\line(1,0){12}}
\put(59,0.5){$\scriptstyle\lozenge$}
\put(63,3){\line(1,0){12}}
\put(74,-.5){$\blacklozenge$}
\put(54,8){$\scriptscriptstyle \eps^r=1$}
\put(83,1){$\scriptstyle +$}
\end{picture}
 \quad
\begin{picture}(60,25)
\put(0,1){$\scriptstyle -$}
 \put(16,3){\line(1,0){12}}
\put(15,3){\circle{3}}
    \put(13.5,3.2){\line(-1,1){10}}\put(13.5,3.2){\line(-1,-1){10}}
\put(2.,14){\circle{3}}\put(2.,-8){\circle*{3}}
\put(29,3){\circle*{3}}
\put(30,3){\line(1,0){12}}
\put(42,1.5){\framebox(3,3)}

\put(45,2){\line(1,0){12}}
\put(45,4){\line(1,0){12}}
\put(54,1){$\scriptstyle >$}
\put(60,1.5){\framebox(3,3)}
\put(66,1){$\scriptstyle +$}
\end{picture}
\quad
\quad
\begin{picture}(50,25)
\put(0,1){$\scriptstyle -$}
 \put(16,3){\line(1,0){12}}
\put(15,3){\circle{3}}
    \put(13.5,3.2){\line(-1,1){10}}\put(13.5,3.2){\line(-1,-1){10}}
\put(2.,14){\circle{3}}\put(2.,-8){\circle*{3}}

\put(29,3){\circle*{3}}
\put(30,3){\line(1,0){12}}
\put(30,3){\line(1,0){12}}
\put(42,1.5){\framebox(3,3)}
\put(46,3){\line(1,0){12}}
\put(58,0.5){$\scriptstyle\lozenge$}

\put(62,3){\line(1,0){12}}
\put(72,-.5){$\blacklozenge$}
\put(52,8){$\scriptscriptstyle \eps^r=1$}
\put(81,1){$\scriptstyle +$}
\end{picture}
\quad
$$
The signs indicate the relation betweeen eigenvalues of the standard invariants of the end sub-diagrams
which obstruct their extention to a global invariant.

Recall that ends of different connection colour are bridged with a shaft diagram of odd white rank.
On the contrary, the white rank of the shaft sub-diagram linking ends with the same connection colour  is even.
Remove the white affine root $\al_0$ from $S$ and denote the resulting Satake sub-diagram by $S^\flat$. Let $\k^\flat\subset \k$
be the corresponding finite dimensional spherical subalgebra and $\Ic$ the standard $\k^\flat$-invariant.
In all cases, its  eigenvalues are related by $\mu=\la$. The matrix $\Nc=\Ic-\la$ is nilpotent, $\Nc^2=0$, and $\k^\flat$-invariant.
However  not with respect to $\k_z$  for any $z$, because $\k^l$-invariance requires $\mu=-\la$.

The operator $\ad \Nc\colon \End(V)\to \End(V)$ has degree of nilpotence $3$. A direct check shows that $\g\not \subset \ker (\ad \Nc)^2$:
for instance, $[\Nc,[\Nc,e_\xi]]\not =0$ for the maximal root vector $e_\xi$.
\begin{propn}
  For a special choice of the loop parameter, $(\ad \Nc)^2(\k_z)=0$.
  As a consequence, the subalgebra $\k$ is proper in $\hat \g$.
\end{propn}
\begin{proof}
There are the following possible one step extensions of the left tail to the right:
$$
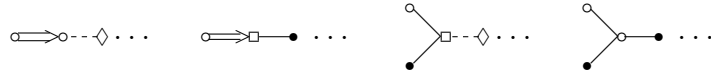
\begin{picture}(60,10)
\put(0,3){\circle{3}}
\put(1,1.5){\line(1,0){12}}
\put(1,4.4){\line(1,0){12}}
\put(10.5,1){$\scriptstyle >$}
 \put(18,3){\circle{3}}
\multiput(21,3)(4,0){3}{\line(1,0){2}}
    \put(30,0.5){$\scriptstyle\lozenge$}
\put(37,2){$\ldots$}
\end{picture}
\quad
\begin{picture}(63,10)
\put(0,3){\circle{3}}
\put(1,2){\line(1,0){12}}
\put(1,4){\line(1,0){12}}
\put(10.5,1){$\scriptstyle >$}
\put(16.5,1.5){\framebox(3,3)}
\put(20,3){\line(1,0){12}}
\put(33,3){\circle*{3}}
\put(40,2){$\ldots$}
\end{picture}
\quad
\begin{picture}(55,25)
    \put(13.5,3.2){\line(-1,1){10}}\put(13.5,3.2){\line(-1,-1){10}}
\put(2.,14){\circle{3}}\put(2.,-8){\circle*{3}}

\put(14,1.5){\framebox(3,3)}

\multiput(18,3)(4,0){3}{\line(1,0){2}}
\put(27,0.5){$\scriptstyle\lozenge$}
\put(34,2){$\ldots$}
\end{picture}
\quad
\begin{picture}(50,25)
\put(29,3){\circle*{3}}
 \put(16,3){\line(1,0){12}}
\put(15,3){\circle{3}}
    \put(13.5,3.2){\line(-1,1){10}}\put(13.5,3.2){\line(-1,-1){10}}
\put(2.,14){\circle{3}}\put(2.,-8){\circle*{3}}
\put(36,2){$\ldots$}
\end{picture}
\quad
$$
The piece
$
\begin{picture}(20,25)
\multiput(1,3)(4,0){3}{\line(1,0){2}}
\put(10,0.5){$\scriptstyle\lozenge$}
\end{picture}
$
stands either for
$
\begin{picture}(20,10)
\put(1,1.5){\line(1,0){12}}
\put(1,4.4){\line(1,0){12}}
\put(10.5,1){$\scriptstyle >$}
 \put(16.5,1.5){\framebox(3,3)}
\end{picture}
$
in the case of empty $S_\s$ or  for
$
\begin{picture}(20,10)
 \put(1,3){\line(1,0){12}}
\put(15,3){\circle{3}}
\end{picture}
$
otherwise.
The rest of the diagram (designated by the dots) is immaterial because the relative generators commute with $x_{\al_0}$.

 It is sufficient to consider only two upper blocks of the matrix $\Nc$, which acts in the subspaces $\C^{m_1}\oplus \C^{m_2}\oplus \C^{m_2}\oplus\C^{m_1}\subset V$:
$$
\Nc=\left(
\begin{array}{ccccc}
\la 1_{m_1}0& 0 & 0&0&a_1\si_{m_1}\\
0 & \la 1_{m_2}& 0&a_2\si_{m_2}&0 \\
0 & 0 &\Nc''&0&0\\
0 &- \frac{\la^2}{a_2}\si_{m_2} & 0&-\la 1_{m_1}&0\\
- \frac{\la^2}{a_1}\si_{m_1} & 0 & 0&0&-\la 1_{m_2}\\
\end{array}
\right).
$$
The dimensions $m_i$ and the blocks $\si_i$ are defined as in (\ref{transition-bridge}).
All four  different combinations are possible. The scalars $a_1,a_2$ are determined by the set of mixture parameters of $\k^\flat$.

Choose the mixed affine generator as
$$
x_0=ze_{1',1}+\frac{c_0^2}{z}e_{1,1'}, \quad x_0=z(e_{2',1}-e_{1',2})+\frac{c_0^2}{z}(e_{1,2'}-e_{2,1'})
$$
for, respectively, the $\Cg$- and $\Dg$-shapes of the left tail.

With the choice $z=\pm \frac{c_0 \la}{a_1}$ of the loop parameter,
$
[\Nc,[\Nc,x_0]]=0.
$
Since $[\Nc,[\Nc,\g]]\not =0$, this proves that $\k_z$ is proper in $\g$. Therefore, $\k$ is proper in $\hat \g$.
\end{proof}
\subsection{Adjoint invariants for waif diagrams}
\label{Sec_waif_inv}
To construct  an invariant matrix in $\End(V_z)$ for a waif Satake diagram $S$ studied in the previous section, we have to lift our minimalist restriction to
the toric parametrization of $\k$.
As shown in the previous section, the left end sub-diagrams in such $S$  are of the following three types:
\be
\label{waif left ends}
\begin{picture}(30,10)
\put(3,3){\circle{3}}
\put(0,9){$\scriptstyle \al_0$}
\end{picture}
\quad
\begin{picture}(40,25)
    \put(13.5,3.2){\line(-1,1){10}}\put(13.5,3.2){\line(-1,-1){10}}
\put(2.,14){\circle{3}}\put(2.,-8){\circle*{3}}
\put(14,1.5){\framebox(3,3)}
\put(-1,20){$\scriptstyle \al_0$}
\end{picture}
\quad
\begin{picture}(45,25)
\put(29,3){\circle*{3}}
 \put(16,3){\line(1,0){12}}
\put(15,3){\circle{3}}
    \put(13.5,3.2){\line(-1,1){10}}\put(13.5,3.2){\line(-1,-1){10}}
\put(2.,14){\circle{3}}\put(2.,-8){\circle*{3}}
\put(-1,20){$\scriptstyle \al_0$}
\end{picture}
\quad
\nn
\ee
In all of three, the affine root $\al_0$ is $\tau$-fixed and isolated from $\hat \Pi_\bullet$.
Therefore $\al_0\in \Omega$.
Let $\k^\flat$ be a subalgebra from the toric family
defined by the sub-diagram $S^\flat=S\backslash \{\al_0\}$ and let $\k$ be its extension by the  generator
\be
\label{affine-affine}
x_0=e_0+c_0 f_0+\grave{c}_0 h_0.
\ee
\begin{propn}
The matrix (\ref{I-osp-block}) with $\mu=\la$  is $\k$-invariant provided  $c_0=-\grave{c}_0^2$ and
$z = \frac{\grave{c}_0}{ \lambda  a_\ell}$ in  $x_0$.
\end{propn}
\begin{proof}
The   standard $\k^\flat$-invariant has a form (\ref{I-osp-block}) with $\mu=\la$.
A direct calculation reduced to   $\Span(v_1,v_{1'})$ for the $\Cg$-tail  and
in $\Span(v_1,v_2,v_{2'},v_{1'})$ for the $\Dg$-tail shows that this matrix also commutes with $x_0$.
\end{proof}
\section{Orthosymplectic  Satake diagrams, flipped tails}
\label{Sec_OSP-Tails flipped}
In this section we study Satake diagrams with  $\tau_\s\not =\id$.
They are built upon the Dynkin diagrams of shapes $\Dg-\Dg$ and $\Cg-\Cg$. The tails are either white or, in the case of $\Dg-\Dg$, are  mixed-coloured.
They are flipped by $\tau$. Without loss of generality we may think that $\al_0$ and $\al_n$ are white and $\tau(\al_0)=\al_n$.

Removing the nodes $\al_0,\al_n$ we obtain  the  Satake sub-diagram $S^\flat=\Sg(\al_1, \ldots, \al_{n-1})$ of shape $\Ag$.
By $(\mg,\k^\flat)$ we
denote the corresponding spherical pair, where $\mg\subset \g$ is the total general linear Lie superalgebra of rank $n-1$.

Big rhombi $\blacklozenge$ occurring in this section  denote finite dimensional general linear subalgebras in $\l$ of arbitrary polarization.
\subsection{Dualization of general linear  invariants}
Orthosymplectic affine Satake diagrams whose left and right tails are flipped by $\tau$ contain sub-diagrams of $\Ag$-shape
supported on $D_\s$. Let us look at how their adjoint invariants transform under $\tau$.

From now on $\g$ is an orthosymplectic Lie superalgebra of rank $n$ and with  natural module $V$ of even dimension dimension $\dim V=2n$.
The vector subspace
$W =\Span\{v_{i}\}_{i=1}^n\subset V$ is the natural module of the general linear subalgebra $\mg$. Denote by
  $W^\sharp$ the vector subspace spanned by $\{v_{i'}\}_{i=1}^n\subset V$.
The direct sum $W \oplus W^\sharp = V$ is an $\m$-module whose projection to $W^\sharp$
defines a representation $\pi\colon \mg\to \End(W^\sharp)$. It acts by the assignment
$$
e_{k,m}\mapsto -(-1)^{\bar m(\bar k+\bar m)} e_{m',k'}, \quad e_{m,k} =-(-1)^{\bar k(\bar k+\bar m)} e_{k',m'}, \quad k<m.
$$
Then $x\op \pi(x)\subset \g$ for all $x\in \mg$, as follows from (\ref{roots-e}) and (\ref{roots-f}).
The homomorphism can be written as a conjugation
$$
\pi(x)=-S  x^tS^{-1},
$$
where $S=\sum_{i=1}^{n} \eps_i e_{i',i}$ is the "lower" part of the invariant matrix $C$. It can be  regarded  as a linear map $S\colon W \to V^\sharp$.
With $S^{-1}=\sum_{i=1}^{n} \eps_i e_{i,i'}$, an even  matrix $A$ transforms to
$$
-A^\sharp   =\sum_{i=1}^{n} \eps_i e_{i',i}\sum_{l,k=1}^nA_{k,l}e_{l,k}\sum_{j=1}^{n} \eps_j e_{j,j'}
=\sum_{i,j=1}^n \eps_i \eps_j  A_{j,i}  e_{i',j'}=\sum_{i,j=1}^n   A_{j,i}  e_{i',j'}.
$$
The last equality holds for even $A$ because $\eps_i=\eps (-1)^{\bar i}$, $i=1,\ldots,n$.
Here $\eps$ measures the difference between the parity and the signature of $C$: $\eps_i=\eps (-1)^{\bar i}$.
In terms of $\End(V)$, the entries of  $A$ are transposed with respect to the skew diagonal.

Now suppose that $S^\flat=\Sg(\al_1,\ldots, \al_{n-1})$ is a Satake sub-diagram of type (\ref{GL-I}) and $\k^\flat$ its spherical subalgebra.
Let $A\in \End(W)$ be its standard adjoint invariant.
Then
\be
\Ic=
\label{flipped bridge}
\left(
\begin{array}{ccc}
  A+\nu 1_n & 0 \\
   0 & A^\sharp+\eta 1_n
\end{array}
\right)
\in \End(W\oplus W^\sharp), \quad \nu, \eta\in \C,
\ee
is invariant with respect to the subalgebra $\k^\flat\subset \g$.
With appropriate choice of parameters, it will be made $\k$-invariant.
\subsubsection{Black rank  $\rk \>\g-1$}
Consider the case with $\l=\mg$ first:
\be
\begin{picture}(55,20)

\put(50,10){$\scriptstyle \al_n$}\put(10,10){$\scriptstyle \al_0$}

\put(17.6,6.5){\vector(-3,-2){2}}\put(49.4,6.5){\vector(3,-2){2}}
\qbezier(17,6)(34.5,20)(50,6)

\put(30,0){$\blacklozenge$}

\put(36,2){\line(1,0){11}} \put(36,5){\line(1,0){11}}
\put(16,1.6){$\scriptstyle <$} \put(45,1.6){$\scriptstyle >$}
\put(31,2){\line(-1,0){11}}\put(31,5){\line(-1,0){11}}
 \put(15,3.5){\circle{3}} \put(52,3.5){\circle{3}}
\end{picture}
\quad\quad\quad
\begin{picture}(55,30)
\put(46,25){$\scriptstyle \al_n$}\put(9,25){$\scriptstyle \al_0$}

\put(23.2,28.6){\vector(-1,-1){2}}\put(41.8,28.6){\vector(1,-1){2}}
\qbezier(22,27.5)(32,38)(42,27.5)

 \put(19,23){$\scriptscriptstyle\lozenge$}\put(19,0){$\scriptscriptstyle\blacklozenge$}
\put(29,10){$\blacklozenge$}

\put(32.5,13.5){\line(-1,1){10}}\put(32.5,13.2){\line(-1,-1){10}}
\put(32.5,13.5){\line(1,1){10}}\put(32.5,13.2){\line(1,-1){10}}
\put(41.5,23){$\scriptscriptstyle\lozenge$}\put(41.5,0){$\scriptscriptstyle\blacklozenge$}

 \multiput(21,4)(0,5){4}{\line(0,1){3}}
\multiput(43.5,4)(0,5){4}{\line(0,1){3}}

\end{picture}
\nn
\ee
The subalgebra $\l$ for the right diagram is regular only if its polarization is symmetric. Therefore  all tail roots carry the same parity.
This is also true for the left diagram, because $\theta(\zt_1)=\zt_n$.
The dashed lines indicate double links
in the case when the tail roots are odd and no link otherwise.
The black subalgebra has an  even polarization.

The subalgebra $\k^\flat$ coincides with $\l$.
It is easy to see  that  $\rho_z(e_{\al_0})\propto \rho_z(f_{\tilde \al_n})$ and $\rho_z(e_{\al_n})\propto \rho_z(f_{\tilde \al_0})$.
Therefore the loop parameter $z$ can be taken such that $x_0$ and $x_n$ vanish on $V_z$.  For instance, one can normalize the  root vectors  so
that   $z=c_0=\frac{1}{c_n}$. The weights $\al_0$ and $\tilde \al_n$ coincide modulo $\C \dt$.
Then, with $A=\la 1_n$, the diagonal matrix  $\Ic=(\la+\nu)\sum_{i=1}^{n}e_{ii}+(\la+\eta)\sum_{i=}^{n}e_{i',i'}$ is a degeneration of (\ref{flipped bridge}).
It  is $\k^\flat$- and therefore $\k$-invariant.

\subsection{$\Dg-\Dg$-shape}
Diagrams of this shape give rise to  $\l$ of rank $n-2k-1$, with positive integer $k$. Zero $\rk  \>\l$ is  also admissible.  We separately study
the case of rank $\rk  \>\l = n-3\geqslant 0$ first.  In this situation  $\Sg(\al_0)$ equals the entire diagram, and  the tails are not separable.

\subsubsection{Black rank $\rk\>\g-3$}
This corresponds to the Satake diagram
$$
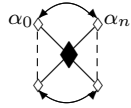
\begin{picture}(55,30)
\put(46,25){$\scriptstyle \al_n$}\put(10,25){$\scriptstyle \al_0$}

\put(23.2,28.6){\vector(-1,-1){2}}\put(41.8,28.6){\vector(1,-1){2}}
\qbezier(22,27.5)(32,38)(42,27.5)

 \put(19,23){$\scriptscriptstyle\lozenge$}\put(19,0){$\scriptscriptstyle\lozenge$}
\put(29,10){$\blacklozenge$}

\put(32.5,13.5){\line(-1,1){10}}\put(32.5,13.2){\line(-1,-1){10}}
\put(32.5,13.5){\line(1,1){10}}\put(32.5,13.2){\line(1,-1){10}}
\put(41.5,23){$\scriptscriptstyle\lozenge$}\put(41.5,0){$\scriptscriptstyle\lozenge$}

\put(23.2,-1.6){\vector(-1,1){2}}\put(41.8,-1.6){\vector(1,1){2}}
\qbezier(22,-0.5)(32,-9)(42,-0.5)

\multiput(21,4)(0,5){4}{\line(0,1){3}}
\multiput(43.5,4)(0,5){4}{\line(0,1){3}}
\end{picture}
$$
Left tail roots carry the same parity as well as right. The parities of the left and right tailss coincide if the
polarization of $\l$ is even (when $\bar 2=\overline{n-1}$) and different otherwise.
In all cases, the equality $\bar 1=\overline{n}$  follows from consideration of the Satake sub-diagram $S^\flat=\langle \al_1,\dots, \al_{n-1}\rangle$.
The mixed generators are represented by
\be
x_0&\mapsto &z \bigl(-\eps_2(-1)^{\bar 1(\bar 1+\bar 2)}e_{2',1}+e_{1',2}\bigr)-c_0 \bigl(-\eps_{1}(-1)^{\bar 2(\bar 2+\bar n)}e_{n',2}+e_{2',n}\bigr), \quad
\nn\\
x_n&\mapsto&
 \bigl(-\eps_{n-1} (-1)^{\overline{n-1}(\overline{n-1}+\bar n) }e_{n-1,n'}+e_{n,(n-1)'}\bigr)-\frac{c_n}{z} \bigl(-\eps_1 (-1)^{\bar 1(\bar 1+\overline{n-1})} e_{1,(n-1)'}+e_{n-1,1'}\bigr),
\nn\\
y_0&\mapsto &
-\eps_n(e_{11} +e_{n,n})+\eps_1(e_{1',1'}
+e_{n',n'})\mod \l.
\nn
\ee
The $\k^\flat$-invariant matrix $\Ic$ in  (\ref{flipped bridge}) comprises the blocks
$$
A=
\left(
\begin{array}{cccccc}
 \mu+\la & 0& a \\
  0 & \la 1_{n-2}&0 \\
  -\frac{\mu\la}{a}&  0&0
\end{array}
\right),
\quad
 A^\sharp=
\left(
\begin{array}{cccccc}
 0 & 0& a \\
  0 & \la 1_{n-2}&0 \\
  -\frac{ \mu\la}{a}&  0&\mu+\la
\end{array}
\right),
\quad
$$
The parameters of the matrices are related with the mixture parameters of  $\k^\flat$.
The matrix $\Ic$ is also $x_0$-and $x_{n}$-invariant
if
$$
c_0\la=(-1)^{\bar 1+\bar 2+\bar 1\bar 2 + \overline{n-1}\>\bar n} c_n\mu, \quad z=(-1)^{\bar 1(\bar 1+\bar 2)}\frac{\eps_1c_0 \la }{a}.
$$
With this relation between $c_0$ and $c_n$, the subalgebra $\k$ is proper in $\hat \g$.
\subsubsection{Black rank $\geqslant \rk \>\g-5$}
Now we study the diagram of shape $\Dg-\Dg$ whose black block is separated from the white tails
In particular, it can be empty:
$$
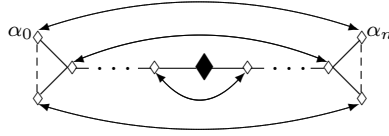
\begin{picture}(100,50)

 \put(9,43){$\scriptscriptstyle\lozenge$}\put(9,20){$\scriptscriptstyle\lozenge$}
\multiput(11,24)(0,5){4}{\line(0,1){3}}
\put(22.5,33.5){\line(-1,1){10}}\put(22.5,33.2){\line(-1,-1){10}}
\put(22,31.5){$\scriptscriptstyle\lozenge$}
 \put(25.5,33){\line(1,0){5}}
 \put(33,32.5){$\ldots$}
  \put(49,33){\line(1,0){5}}

\put(53,31.5){$\scriptscriptstyle\lozenge$}
\put(57,33){\line(1,0){14}}
\put(70,30){$\blacklozenge$}
\put(88,31.5){$\scriptscriptstyle\lozenge$}
\put(76,33){\line(1,0){13}}
 \put(91.5,33){\line(1,0){5}}
 \put(99,32.5){$\ldots$}
  \put(114,33){\line(1,0){5}}

 \put(118.5,31.5){$\scriptscriptstyle\lozenge$}

\put(122.5,33.5){\line(1,1){10}}\put(122.5,33.2){\line(1,-1){10}}

 \put(131,43){$\scriptscriptstyle\lozenge$}\put(131,20){$\scriptscriptstyle\lozenge$}
 \multiput(133,24)(0,5){4}{\line(0,1){3}}

\put(13.6,48.2){\vector(-3,-1){2}}\put(130.4,48.2){\vector(3,-1){2}}
\qbezier(13,48)(72,68)(131,48)
\put(26.6,36.2){\vector(-3,-1){2}}\put(118.4,36.2){\vector(3,-1){2}}
\qbezier(26,36)(72,55)(119,36)
\put(57.6,29.4){\vector(-1,1){2}}\put(87.4,29.4){\vector(1,1){2}}
\qbezier(57,30)(72,12)(88,30)
\put(13.6,18.8){\vector(-3,1){2}}\put(130.4,18.8){\vector(3,1){2}}
\qbezier(13,19)(72,1)(131,19)
\put(135,45){$\scriptstyle \al_n$}\put(0,45){$\scriptstyle \al_0$}

\end{picture}
$$
We fix the mixed operators $x_0$ and $x_n$ as
\be
x_0&\mapsto &z \bigl(-\eps_2(-1)^{\bar 1(\bar 1+\bar 2)}e_{2',1}+e_{1',2}\bigr)-c_0 \bigl(-\eps_{n}(-1)^{\overline{n-1}(\overline{n-1}+\bar n)}e_{n',n-1}+e_{(n-1)',n}\bigr), \quad
\nn\\
x_n&\mapsto&
 \bigl(-\eps_{n-1} (-1)^{\overline{n-1}(\overline{n-1}+\bar n) }e_{n-1,n'}+e_{n,(n-1)'}\bigr)-\frac{c_n}{z} \bigl(-\eps_1 (-1)^{\bar 1(\bar 1+\bar 2)} e_{1,2'}+e_{2,1'})\bigr).
 \nn
\ee
Notice that
$\bar 1=\bar n, \quad \bar 2=\overline{n-1}$. Then we simplify them to
\be
x_0&\mapsto &z \bigl(-\eps_2(-1)^{\bar 1(\bar 1+\bar 2)}e_{2',1}+e_{1',2}\bigr)-c_0 \bigl(-\eps_{1}(-1)^{\overline{2}(\overline{2}+\bar 1)}e_{n',n-1}+e_{(n-1)',n}\bigr), \quad
\nn\\
x_n&\mapsto&
 \bigl(-\eps_{2} (-1)^{\overline{2}(\overline{2}+\bar 1) }e_{n-1,n'}+e_{n,(n-1)'}\bigr)-\frac{c_n}{z} \bigl(-\eps_1 (-1)^{\bar 1(\bar 1+\bar 2)} e_{1,2'}+e_{2,1'})\bigr).
\nn
\ee
The additional Cartan generator is
\be
 y_0&\mapsto &
-\eps_1(e_{11} +e_{n,n})+\eps_1(e_{1',1'}
+e_{n',n'})-\eps_{2}(e_{2,2}+e_{n-1,n-1})+\eps_{2}(e_{(n-1)',(n-1)'}+e_{2',2'}).
\nn
\ee
Let $\k^\flat$ be the subalgebra in $\k$ restriced to the subdiagram $\Sg(\al_1,\ldots, \al_{n-1})$.
The matrix blocks constituting a $\k^\flat$-invariant $\Ic$ in (\ref{flipped bridge}) are
$$
A=
\left(
\begin{array}{cccccccc}
 \mu+\la & 0&0&0 &a_1 \\
 0& \mu+\la & 0&a_2& 0& \\
  0 & 0 &B&0&0 \\
  0&-\frac{\mu\la}{a_2}&  0&0&0\\
 -\frac{\mu\la}{a_1}&  0&0&0&0
\end{array}
\right),
\quad
A^\sharp=
\left(
\begin{array}{cccccccc}
 0 & 0&0&0 &a_1 \\
 0& 0 & 0&a_2& 0& \\
  0 & 0 &B^\sharp&0&0 \\
  0&-\frac{\mu\la}{a_2}&  0&\mu+\la&0\\
 -\frac{\mu\la}{a_1}&  0&0&0&\mu+\la
\end{array}
\right),
\quad
$$
where $B$ and $B^\sharp$ are square  matrices of size $n-4$.
The scalars in $\Ic$ are related to the mixture parameters of  $\k^\flat$.
The matrix $\Ic$ is $\k$-invariant if and only if
$$
z=-\eps_1 (-1)^{\bar 1(\bar 1+\bar 2)}\frac{c_0\mu \lambda}{ a_1 a_2}, \quad c_n=(-1)^{\bar 1+\bar 2}c_0.
$$
Then the subalgebra $\k$ is proper.
Remark that $\mu=-\la$  if $\hat \Pi_\bullet=\varnothing$, cf. Section \ref{Sec_ad-gl-in}.
\subsection{$\Cg-\Cg$-shape}
Finally, suppose that the black block of the diagram is isolated from the tails of shape $\Cg$.
\be
\begin{picture}(140,30)

\put(4,10){$\scriptstyle \al_0$}\put(140,10){$\scriptstyle \al_n$}

\put(9,3.5){\circle{3}}
\put(27,3.5){\line(1,0){5}}
 \put(34,3){$\ldots$}
  \put(49,3.5){\line(1,0){5}}

\put(56,3.5){\circle{3}}

\put(25,3.5){\circle{3}}

\put(18,1.6){$\scriptstyle >$}
\put(20,2){\line(-1,0){11}}\put(20,5){\line(-1,0){11}}

\put(72,3.5){\line(-1,0){14}}
\put(70,0){$\blacklozenge$}
\put(76,3.5){\line(1,0){14}}

 \put(92,3.5){\circle{3}}
 \put(93.5,3.5){\line(1,0){5}}
 \put(101,3){$\ldots$}
  \put(116,3.5){\line(1,0){5}}
\put(123,3.5){\circle{3}}

\put(128,2){\line(1,0){11}} \put(128,5){\line(1,0){11}}
\put(124,1.6){$\scriptstyle <$}
\put(140,3.5){\circle{3}}

\put(57.8,6.8){\vector(-3,-2){2}}\put(89.2,6.8){\vector(3,-2){2}}
\qbezier(57,6)(74.5,20)(90,6)
\put(27.6,6.6){\vector(-3,-2){2}}\put(120.4,6.6){\vector(3,-2){2}}
\qbezier(27,6)(74.5,40)(121,6)
\put(11.8,6.8){\vector(-1,-1){2}}\put(137.2,6.8){\vector(1,-1){2}}
\qbezier(11,6)(74.5,60)(138,6)

\end{picture}
\nn
\ee
As in the previous section, we admit empty $\hat \Pi_\bullet$.
The matrix blocks  $A$ and $A^\sharp$ entering a $\k^\flat$-invariant $\Ic$ are
$$
A=
\left(
\begin{array}{cccccccc}
 \mu+\la & 0&a_1\\
  0 &B&0 \\
 -\frac{\mu\la}{a_1}&  0&0\\
  \end{array}
\right),
\quad
\tilde A^\sharp=
\left(
\begin{array}{cccccccc}
 0 & 0&a_1\\
  0 &B^\sharp&0 \\
 -\frac{\mu\la}{a_1}&  0&\mu+\la\\
  \end{array}
\right),
\quad
$$
where $B$ and $B^\sharp$ are square  matrices of size $n-2$.
Their parameters are related to the mixture parameters of $\k^\flat$ in the standard way.
With the mixed generators
$$
x_0=e_{1',1}+c_0 e_{n',n}, \quad x_n:=e_{n,n'}+\frac{c_n}{z}e_{1,1'},
$$
$$
y_0=e_{1',1'}+e_{n',n'}-(e_{1,1}+e_{n,n}),
$$
the matrix $\Ic$ is $\k$-invariant if and only if
$
z=-\frac{c_0\mu \lambda}{ a_1^2 }$, $c_n=c_0$.
Then $\k$ is proper in $\hat \g$.
Again,  if $\hat \Pi_\bullet=\varnothing$, then the eigenvalues of $A$ are not independent: $\mu=-\la$.
This completes our analysis.
\vspace{20pt}

\noindent
\underline{\large \bf Acknowledgement}

\vspace{10pt}
\noindent
This work is done at the Center of Pure Mathematics MIPT. It is financially supported by Russian Science Foundation grant 26-11-00115.

\vspace{20pt}
\noindent
\underline{\large \bf Data Availability}

\noindent
Data sharing not applicable to this article as no datasets were generated or analysed during the current study.

\vspace{10pt}
\subsection*{Declarations}

\underline{\large \bf Competing interests}

\vspace{10pt}
\noindent
The authors have no competing interests to declare that are relevant to the content of this article.

 \end{document}